%% file: main.tex
\documentclass[11pt,leqno]{article}
\usepackage[T1]{fontenc}
\usepackage{authblk}
\usepackage[margin=1.25in]{geometry}
\usepackage{amsmath,amssymb,amsthm,mathtools,booktabs,longtable,array}
\usepackage{dsfont,cite,microtype,hyperref}
\hypersetup{colorlinks,breaklinks,linkcolor=blue,citecolor=blue,urlcolor=blue,
 pdftitle={The finite-horizon five-expert prediction problem},
 pdfauthor={Jeff Calder and Nadejda Drenska}}

\newcommand{\R}{\mathbb{R}}
\newcommand{\Z}{\mathbb{Z}}
\renewcommand{\v}{\mathbf{v}}
\newcommand{\one}{\mathds{1}}
\newcommand{\E}{{\mathbb E}}
\renewcommand{\O}{{\mathcal O}}
\renewcommand{\phi}{\varphi}
\DeclareMathOperator{\csch}{csch}
\DeclareMathOperator{\sech}{sech}
\DeclareMathOperator{\erfc}{erfc}
\DeclareMathOperator{\Dawson}{Dawson}
\newcommand{\dd}{\,dt}
\newcommand{\Iop}{\mathcal I}
\newcommand{\Kop}{\mathcal K}
\newcommand{\Lap}{\mathcal L}
\newcommand{\Rop}{\mathcal R}
\newcommand{\Top}{\mathcal T}
\newcommand{\Aop}{\mathcal A}
\newcommand{\coll}{\mathcal C}
\newcommand{\Gk}{\mathcal G}  %
\newcommand{\Ek}{\mathcal E}  %
\newcommand{\Hk}{\mathcal H}  %
\newcommand{\Lam}{\Lambda}
\newcommand{\gen}[1]{\mathsf{#1}}
\newcommand{\CM}{\mathrm{CM}}

\newtheorem{theorem}{Theorem}
\newtheorem{proposition}[theorem]{Proposition}
\newtheorem{lemma}[theorem]{Lemma}
\newtheorem{corollary}[theorem]{Corollary}
\theoremstyle{definition}
\newtheorem{remark}[theorem]{Remark}

\numberwithin{equation}{section}
\numberwithin{theorem}{section}

\title{The finite-horizon five-expert prediction problem\thanks{{\bf Funding:} JC acknowledges funding from an Albert and Dorothy Marden Professorship, a Simons Fellowship, and National Science Foundation Grant DMS:2436333. ND acknowledges funding from National Science Foundation Grant DMS:2407839. {\bf Source Code:} The certificate scripts, their exact rational data and the verification notes are archived at \url{https://doi.org/10.5281/zenodo.23012438}. {\bf Formalization:} The Lean~4 formalization of the proofs of Theorems~\ref{thm:main} and~\ref{thm:comb}, the interval certificates included, is part of the same archive. The proof development and the preparation of this paper were assisted by ChatGPT and Claude. The authors take full responsibility for the contents of the paper.}}

\author[1]{Jeff Calder}
\affil[1]{School of Mathematics, University of Minnesota}
\author[2]{Nadejda Drenska}
\affil[2]{Department of Mathematics, Louisiana State University}

\begin{document}
\maketitle
\begin{abstract}
We give an explicit solution to the five expert prediction with expert advice partial differential equation (PDE) in the finite-time horizon setting. The solution formula establishes that the adversary's rank strategy $(1,0,1,0,0)$ is globally optimal, and the COMB strategy $(1,0,1,0,1)$ is optimal exactly on the set where $x_1=x_2$ and $x_3=x_4$. The formula is derived from the solution of the geometric-stopping problem given in our companion paper through the transform principle of Bayraktar, Ekren and Zhang, which links the two problems by a Laplace transform. Inverting the transform term by term expresses the solution through a series of Gaussian and complementary error function kernels. The optimality of $(1,0,1,0,0)$ is reduced to the signs of $41$ one-variable Gaussian series, which are certified with computer assistance by Poisson summation, first-mode domination and interval arithmetic on $1616$ rational cells. The proofs of our main theorems, certificates included, are also formalized in the Lean proof assistant.
\end{abstract}
\tableofcontents
\input{introduction}
\input{main_results}
\input{derivation}
\input{verification}

\appendix
\input{appendix}

\bibliographystyle{abbrv}
\bibliography{main}

\end{document}

%% file: introduction.tex
\section{Introduction}\label{sec:introduction}

In finite-horizon prediction with expert advice, a player and an adversary play a known number of rounds $M$. In each round the player picks, possibly at random, one of $n$ experts to follow, the adversary decides which experts are correct, and at the end of the game the player's performance is measured by their regret---the excess of the player's cumulative loss over that of the best expert. The game dates back to Cover~\cite{Cover65}; for an extensive literature review on regret bounds we refer to~\cite{CBL06}, and for the history of the exact minimax problem we refer to the introduction of our companion paper~\cite{CD26}. Drenska and Kohn~\cite{Drenska17,DK20} showed that if $V^M(x)$ is the value of the $M$-round game with initial regrets $x$, then $V^M(\sqrt M\,x)/\sqrt M\to U(1,x)$ locally uniformly as $M\to \infty$, where $U$ is the viscosity solution of the nonlinear parabolic partial differential equation (PDE) 
\begin{equation}\label{eq:parabolic-intro}
 U_\tau=\tfrac12\max_{\v\in\{0,1\}^n}\v^T\nabla^2U\,\v,\qquad U(0,x)=\max_{1\leq i\leq n}x_i,
\end{equation}
on $(0,\infty)\times\R^n$; in particular $V^M(0)=U(1,0)\sqrt M+o(\sqrt M)$. Here $\tau$ is the time remaining and $x_i$ the player's regret with respect to expert $i$. The gradient $\nabla U$ gives the player's asymptotically optimal mixed strategy, and the maximizing controls $\v$ give the adversary's.

Explicit solutions of \eqref{eq:parabolic-intro} allow for the exact characterization of optimal strategies, but only the cases of $n\leq 4$ have so far been addressed.  For $n\leq3$ the solution is elementary, and Abbasi-Yadkori, Bartlett and Gabillon~\cite{ABG17} constructed near-minimax strategies for the three-expert game. For $n=4$, Bayraktar, Ekren and Zhang~\cite{BEZ20b} derived the solution of the $n=4$ finite-horizon game from their solution~\cite{BEZ20} of the geometric-horizon version of the problem via the inversion of a Laplace transform. In both settings COMB, which ranks the experts by regret and makes either the odd-ranked or the even-ranked ones correct, with probability $\tfrac12$ each, is optimal for four experts (as is the strategy that makes the two middle-ranked experts correct) and Gravin, Peres and Sivan~\cite{GPS16} conjectured that COMB stays asymptotically optimal for every $n$. Numerical evidence~\cite{Chase19,CDM26} pointed the other way for $n\geq 5$, and in~\cite{CD26} the authors gave an explicit solution to the five-expert geometric-horizon problem: 
the non-COMB rank direction $\v_*=(1,0,1,0,0)$ attains the maximum everywhere in the ordered sector $x_1\geq\cdots\geq x_5$, while COMB is optimal only where $x_1=x_2$ and $x_3=x_4$. The same formula for geometric stopping was derived independently by Bayraktar, Ekren and Kolliopoulos~\cite{BEK26} in a preprint that appears shortly after \cite{CD26}. 

In this paper we solve the five-expert finite time horizon problem. Theorem~\ref{thm:main} gives an explicit function $U$ that is a classical solution of \eqref{eq:parabolic-intro} with $n=5$ for which $\v_*$ attains the maximum at every point. In particular $V^M(0)/\sqrt M\to U(1,0)=45\pi^{3/2}/(256\sqrt2)$ as $M\to\infty$, which is $2/\sqrt\pi$ times the geometric-horizon constant $45\pi^2/(512\sqrt2)$ of~\cite{CD26}. This is the ratio conjectured by Gravin, Peres and Sivan~\cite[\S1.2]{GPS16}, and Bayraktar, Ekren and Zhang~\cite{BEZ20b} observed that it holds whenever one strategy is optimal for both problems. Theorem~\ref{thm:comb} shows that COMB attains the maximum exactly on the set $\{x_1=x_2,\ x_3=x_4\}$ found in the geometric problem, so COMB is not globally optimal for five experts at any horizon.

The construction uses the transform principle of Bayraktar, Ekren and Zhang~\cite{BEZ20b}, stated here as Proposition~\ref{prop:transform}: if one control $\v$ attains the maximum everywhere for the geometric-horizon solution $u$, then $\lambda^{-3/2}u(\sqrt\lambda\,x)$ solves the equation obtained by Laplace transforming  the finite-horizon problem in $\tau$ with the adversary frozen at that control, and its inverse transform is a natural candidate for $U$. For four experts they inverted this transform along the branch cut of $\sqrt\lambda$, which produces an oscillatory integral. We invert it term by term instead. We expand the formulas of~\cite{CD26} in exponentials of linear functions of $x$ with polynomial prefactors, and each such term inverts to an explicit kernel built from a Gaussian and the complementary error function. The result is a series with rational coefficients; for four experts the same procedure recovers the solution of~\cite{BEZ20b} as such a series (Proposition~\ref{prop:four}).

Most of the paper is devoted to proving that $U$ solves \eqref{eq:parabolic-intro}. We obtain its regularity where experts tie, including the collision set on which the series degenerate, by continuing the stationary formulas to complex scaling parameters and inverting the transform on a Hankel contour. By the transform principle, the inequalities expressing that $\v_*$ attains the maximum are equivalent to the complete monotonicity in $\lambda$ of finitely many curvature gaps of the stationary solution, which we reduce to the signs of $41$ explicit one-variable Gaussian series. These signs are proved with the assistance of a computer, on the whole positive axis rather than at sample points: by Poisson summation for small arguments, by first-mode domination for large ones, and in between by outward-rounded interval arithmetic over $1616$ rational cells, with scripts and data in the supplement~\cite{CD26Supp}. The proofs of Theorems~\ref{thm:main} and~\ref{thm:comb}, these certificates included, have also been formalized in the Lean proof assistant with the Mathlib library~\cite{CD26Supp}, and Table~\ref{tab:lean} lists the Lean theorem behind each claim; the $\sqrt M$ asymptotics are the cited theorem of Drenska and Kohn and are not part of the formalization. %

%% file: main_results.tex
\section{Main results}\label{sec:main}

We consider the finite-horizon prediction with expert advice PDE \eqref{eq:parabolic-intro} with five experts, which in the time to maturity $\tau=T-t$ reads
\begin{equation}\label{eq:forward}
 U_\tau=\tfrac12\max_{\v\in\{0,1\}^5}\v^T\nabla^2U\,\v,
 \qquad U(0,x)=\phi(x):=\max_{1\leq i\leq 5}x_i,
\end{equation}
on $(0,\infty)\times\R^5$. The stationary 5-expert problem solved in the companion paper~\cite{CD26} is
\begin{equation}\label{eq:stationary}
 u(x)-\tfrac12\max_{\v\in\{0,1\}^5}\v^T\nabla^2u(x)\,\v=\phi(x),\qquad x\in\R^5,
\end{equation}
whose solution is the value of the geometric-horizon game. Throughout we write $D^2_\v w=\v^T\nabla^2w\,\v$ for the second derivative of a function $w$ in the direction $\v$, and we work in the ordered sector $\{x_1\geq x_2\geq x_3\geq x_4\geq x_5\}$, with $k=\sqrt2$.

We use the scaled gaps $y_i=k(x_i-x_{i+1})\geq0$ and the coordinates of the companion paper,
\begin{equation}\label{eq:coords}
\begin{aligned}
 z_1&=\frac{x_1+x_2-x_3-x_4}{k},\ \ z_2=\frac{x_1-x_2+x_3-x_4}{k},\\
 z_3&=\frac{x_1-x_2-x_3+x_4}{k},\ \ z_4=y_4-\max(z_3,0),\\
 a_1&=\frac{y_1-y_3-2y_4}3,\ \  a_2=a_1+y_4,\ \ a_3=a_2+y_3,\ \  a_4=a_3+y_2 .
\end{aligned}
\end{equation}
Note that the ordered sector in $x$ transforms to the set $\{z_1\geq z_2\geq|z_3|\}$, and the fifth expert enters only through $z_4$. As in the companion paper, the sector splits into three regions:
\begin{center}
\begin{tabular}{lll}\toprule
Region & Scaled gaps & Coordinates\\\midrule
I & $y_1\leq y_3$ & $z_3\leq0\leq z_4$\\
II & $y_3\leq y_1\leq y_3+2y_4$ & $z_3\geq0$, $z_4\geq0$\\
III & $y_1\geq y_3+2y_4$ & $z_3\geq0$, $z_4\leq0$\\\bottomrule
\end{tabular}
\end{center}
In Region III one has $0\leq a_1\leq a_2\leq a_3\leq a_4$ and $y_1=a_1+a_2+a_3$. Where $z_3\geq0$ one has $a_1=-\tfrac23z_4$, so Region III is $\{a_1\geq0\}$, its interface $z_4=0$ with Region II is $\{a_1=0\}$, and on that interface $(a_2,a_3,a_4)=(z_3,z_2,z_1)$.

A rank strategy acts on the coordinates in decreasing order, so a function defined on the ordered sector is extended to $\R^5$ by sorting the coordinates. The stationary solution $u$ and the function $U$ constructed below are both $x_1$ plus a function of the differences of the coordinates. Their Hessians therefore annihilate $\one$, so a control $\v$ and its complement $\one-\v$ give the same second derivatives, and we identify complementary controls throughout. We write $\v_*=(1,0,1,0,0)$ for the rank direction that attains the Hamiltonian maximum of \eqref{eq:stationary} throughout the sector~\cite{CD26}, and we write $\coll$ for the set of states at which the four largest coordinates coincide; in the ordered sector $\coll=\{x_1=x_2=x_3=x_4\}=\{z_1=0\}$, which contains the diagonal $\{a_4=0\}$. The mode series below degenerate exactly on $\coll$, and the regularity of $U$ there is proved separately in Section~\ref{sec:collision}.

The formula for $U$ is written in terms of kernels $\Hk_j:\R \times (0,\infty) \to \R$, $j\geq 0$, defined as follows. With
\begin{equation}\label{eq:gE}
 \Gk(X,\tau)=\frac{1}{\sqrt{\pi\tau}}e^{-X^2/(4\tau)},\qquad
 \Ek(X,\tau)=\erfc\bigl(X/(2\sqrt\tau)\bigr),
\end{equation}
the first three kernels are
\begin{equation}\label{eq:H012}
 \Hk_0=2\tau\,\Gk-X\,\Ek,\qquad \Hk_1=\Ek,\qquad \Hk_2=\Gk,
\end{equation}
and the others are defined by the recursion
\begin{equation}\label{eq:Hrec}
 \Hk_{j+1}=\frac X{2\tau}\Hk_j-\frac{j-2}{2\tau}\Hk_{j-1},\qquad j\geq2.
\end{equation}
Every $\Hk_j$ is an elementary combination of a Gaussian and the complementary error function. We also note that the kernels can be expressed as inverse Laplace transforms: for every $\tau>0$,
\begin{equation}\label{eq:Hdef}
 \Hk_j(X,\tau)=\Lap^{-1}\!\left[\lambda^{(j-3)/2}e^{-X\sqrt\lambda}\right](\tau)
 \qquad\text{when $X>0$, or $X=0$ and $j\leq2$,}
\end{equation}
where $\Lap$ denotes the Laplace transform in $\tau$ and $\lambda$ the Laplace variable. At $X=0$ the identity fails for $j\geq3$. For $j=3$ its right side $\Lap^{-1}[1]$ is not a function while $\Hk_3(0,\tau)=0$. For $j=4$ the kernel $\Hk_4(0,\tau)=-1/(2\sqrt\pi\,\tau^{3/2})$ is not integrable at $\tau=0$. The formula below uses kernels of order $j\geq3$ only where their argument $X$ is positive. The transform identity, together with the differentiation rules and the bounds on the $\Hk_j$ that we use, is proved in Lemma~\ref{lem:kernels}.

The formula for $U$ is derived from the stationary solution $u=x_1+F/k$ of \eqref{eq:stationary}. In each region we expand $F$ in a series of terms $P(x)\,e^{-X(x)}$, which we call modes, with $X$ linear and $P$ a homogeneous polynomial, and replace each mode by $P(x)\,\Hk_j(X(x),\tau)$, where $j$ is the degree of $P$ (Corollary~\ref{cor:recipe}). We now give the coefficients and exponents of these expansions.

In Region III the formula is built from the trace $e$ of the companion paper, defined by the quadrature
\begin{equation}\label{eq:equad}
 e(X)=3\cosh X\int_X^\infty I(t)\,\sech^2t\,\csch^5t\dd,
 \qquad
 I(X)=\int_0^X\sinh^4s\cosh^2s\,ds.
\end{equation}
The integrand of $e$ is regular at $t=0$ since $I(t)=\frac15 t^5+\O(t^7)$, and decays like $\tfrac13e^{-t}$ at infinity; hence $e$ is positive and bounded. We also need the function $\Lam(L)=\log\coth(L/2)$.\footnote{The companion paper writes $\lambda$ for $\Lam$, but in this paper $\lambda$ is the Laplace variable.} Set
\begin{equation}\label{eq:AB}
 A_0=\tfrac12,\ \
 A_m=-\frac{12m^6-25m^4-m^2-4}{4(4m^2-1)^2}\ \ (m\geq1),
 \ \
 B_m=\frac{m(m^2-1)(m^2-4)}{2(4m^2-1)},
\end{equation}
\begin{equation}\label{eq:JK}
 J_0=K_0=1,\qquad J_m=\frac{-2}{4m^2-1},\qquad K_m=\frac{4m}{4m^2-1}\quad(m\geq1),
\end{equation}
and, with $\delta_{m0}$ the Kronecker delta,
\begin{equation}\label{eq:alphabeta}
\begin{aligned}
 \alpha^{(0)}_m&=A_m, &\ \beta^{(0)}_m&=B_m,\\
 \alpha^{(1)}_m&=\tfrac16\bigl(B_m-2mA_m-3\delta_{m0}\bigr), &\ \beta^{(1)}_m&=-\tfrac13mB_m,\\
 \alpha^{(2)}_m&=\tfrac1{21}\bigl((2m^2+3)A_m-2mB_m+9J_m\bigr), &\ \beta^{(2)}_m&=\tfrac1{21}(2m^2+3)B_m,\\
 \alpha^{(3)}_m&=\tfrac1{84}\bigl(-2m(m^2+5)A_m+(3m^2+5)B_m\bigr)\\
 &\quad+\tfrac1{14}\bigl(6K_m-13\delta_{m0}\bigr),
  &\ \beta^{(3)}_m&=-\tfrac1{42}m(m^2+5)B_m .
\end{aligned}
\end{equation}
Proposition~\ref{prop:trace} shows that $A_m$ and $B_m$ are the coefficients of the exponential expansion
\[
 e(L)=\sum_{m\geq0}(A_m+B_mL)e^{-2mL}
\]
of the trace, $J_m$ and $K_m$ are those of $\Lam(L)\sinh L$ and $\Lam(L)\cosh L$, and $\alpha^{(j)}_m,\beta^{(j)}_m$ are those of the four coefficient functions of the Region III stationary formula from \cite{CD26}. For $\epsilon\in\{\pm1\}^3$ let $s_j(\epsilon)$ denote the $j$-th elementary symmetric polynomial of $\epsilon_1,\epsilon_2,\epsilon_3$,\footnote{That is, $s_0(\epsilon)=1$, $s_1(\epsilon)=\epsilon_1+\epsilon_2+\epsilon_3$, $s_2(\epsilon)=\epsilon_1\epsilon_2+\epsilon_1\epsilon_3+\epsilon_2\epsilon_3$ and $s_3(\epsilon)=\epsilon_1\epsilon_2\epsilon_3$.} and define
\begin{equation}\label{eq:cab}
 c^a_{m,\epsilon}=\sum_{j=0}^3\alpha^{(j)}_m\,s_j(\epsilon),
 \ \
 c^b_{m,\epsilon}=\sum_{j=0}^3\beta^{(j)}_m\,s_j(\epsilon),
 \ \
 X_{m,\epsilon}=2m\,a_4-\epsilon_1a_1-\epsilon_2a_2-\epsilon_3a_3 .
\end{equation}
They are the coefficients and exponents of the mode expansion of the stationary solution in Region III, derived in the proof of Proposition~\ref{prop:region3}: for $a_4>0$,
\[
 F=\frac18\sum_{m\geq0}\sum_{\epsilon\in\{\pm1\}^3}\bigl(c^a_{m,\epsilon}+c^b_{m,\epsilon}\,a_4\bigr)e^{-X_{m,\epsilon}} .
\]

In Regions I and II the formula has a four-expert background and a correction. Write $\theta(L)=\arctan(e^{-L})$. The stationary four-expert solution of Bayraktar, Ekren and Zhang~\cite[Theorem~3.1]{BEZ20} is, in the present notation, $u_4=x_1+\frac1kF_4$ with
\begin{equation}\label{eq:F4}
 F_4=\theta(z_1)\cosh z_1\cosh z_2\cosh z_3
   +\tfrac12\Lam(z_1)\sinh z_1\sinh z_2\sinh z_3-\tfrac12\sinh(z_2+z_3),
\end{equation}
and Lemma~\ref{lem:F4expansion} gives its mode expansion
\[
 F_4=\sum_{m\geq0}\ \sum_{\substack{\epsilon\in\{\pm1\}^3\\ \epsilon_1\epsilon_2\epsilon_3=(-1)^m}}
      c^{(4)}_{m,\epsilon}\,e^{-X^{(4)}_{m,\epsilon}},
 \qquad z_1>0,
\]
where
\begin{equation}\label{eq:c4D4}
\begin{gathered}
 c^{(4)}_{m,\epsilon}=\frac{(-1)^m}{4(2m+1)}
 \quad\text{for }\epsilon\in\{\pm1\}^3\text{ with }\epsilon_1\epsilon_2\epsilon_3=(-1)^m,\\
 X^{(4)}_{m,\epsilon}=(2m+1-\epsilon_1)z_1-\epsilon_2z_2-\epsilon_3z_3,
\end{gathered}
\end{equation}
with the two exceptions $c^{(4)}_{0,(1,1,1)}=0$ and $c^{(4)}_{0,(1,-1,-1)}=\tfrac12$. The fifth expert enters through the density
\begin{equation}\label{eq:pdef}
 p(X)=\tfrac12\tanh X-3\coth X+\frac{15\,I(X)}{\sinh^6X},
\end{equation}
which is positive. For integers $j\geq0$ define the rationals $r_{j,n}$ and $s_{j,n}$, $n\geq0$, as the coefficients of the power series, convergent for $|w|<1$,
\begin{equation}\label{eq:rs}
\begin{gathered}
 \sum_{n\geq0}r_{j,n}w^n=\frac{w^{j+2}\,(1-14w-94w^2-14w^3+w^4)}{2(1+w)(1-w)^{j+7}},\\
 \sum_{n\geq0}s_{j,n}w^n=\frac{60\,w^{j+4}}{(1-w)^{j+8}}.
\end{gathered}
\end{equation}
Finally, for $\epsilon_2,\epsilon_3\in\{\pm1\}$ and $n\geq2$ put
\begin{equation}\label{eq:cD12}
 c=c(n,\epsilon_2,\epsilon_3)=2n-\epsilon_2-\epsilon_3,
 \qquad
 X_{j,n,\epsilon}=c\,z_1+\epsilon_2z_2+\epsilon_3|z_3|+2z_4,
\end{equation}
and
\begin{equation}\label{eq:ab12}
 \alpha_{j,n}(c)=\frac{2r_{j,n}}{c^2-1}+\frac{4c\,s_{j,n}}{(c^2-1)^2},
 \qquad
 \beta_{j,n}(c)=\frac{2s_{j,n}}{c^2-1}.
\end{equation}
Since $c\geq2$, the denominators in \eqref{eq:ab12} do not vanish. The mode expansion of the correction, derived in Lemma~\ref{lem:dF}, has the exponents $X_{j,n,\epsilon}$, and its coefficients are built from $\alpha_{j,n}(c)$ and $\beta_{j,n}(c)$. We are now ready to state our main result.

\begin{theorem}\label{thm:main}
Define $U$ on the ordered sector $\{x_1\geq\cdots\geq x_5\}$ by
\begin{equation}\label{eq:U3}
 U(\tau,x)=x_1+\frac1{8k}\sum_{m\geq0}\sum_{\epsilon\in\{\pm1\}^3}
 \Bigl[c^a_{m,\epsilon}\,\Hk_0\bigl(X_{m,\epsilon},\tau\bigr)
      +c^b_{m,\epsilon}\,a_4\,\Hk_1\bigl(X_{m,\epsilon},\tau\bigr)\Bigr]
\end{equation}
in Region III outside $\coll$, by
\begin{equation}\label{eq:U12}
\begin{aligned}
 U(\tau,x)=x_1
 &+\frac1k\sum_{m\geq0}\ \sum_{\substack{\epsilon\in\{\pm1\}^3\\ \epsilon_1\epsilon_2\epsilon_3=(-1)^m}}
      c^{(4)}_{m,\epsilon}\,\Hk_0\bigl(X^{(4)}_{m,\epsilon},\tau\bigr)\\
 &+\frac1{2k}\sum_{\epsilon_2,\epsilon_3\in\{\pm1\}}\epsilon_2\epsilon_3
   \sum_{j\geq0}\frac{(-4)^jz_4^{\,j}}{j!}\sum_{n\geq j+2}
   \Bigl[\alpha_{j,n}(c)\,\Hk_j(X_{j,n,\epsilon},\tau)\\
 &\hspace{5cm}
        +\beta_{j,n}(c)\,z_1\Hk_{j+1}(X_{j,n,\epsilon},\tau)\Bigr]
\end{aligned}
\end{equation}
in Regions I and II outside $\coll$, and by
\begin{equation}\label{eq:face}
 U(\tau,x)=x_1+\frac{\sqrt{\pi\tau}}{2\sqrt2}
   +\frac1k\int_0^\infty\sinh t\,p(t)\,\Hk_0\bigl(2z_4\coth t,\tau\bigr)\dd
\end{equation}
on $\coll$, and extend $U$ to $\R^5$ by sorting the coordinates. Then the series converge absolutely, locally uniformly on the sets where they are used, together with the series of their derivatives of every order, each series being differentiated within its own region, where its coordinates are linear in $x$; and the following hold.
\begin{enumerate}
\item[(i)] $U$ is a classical solution of \eqref{eq:forward} on $(0,\infty)\times\R^5$: for every $\tau>0$, $U(\tau,\cdot)\in C^2(\R^5)$, the derivatives $U_\tau$, $\nabla U$ and $\nabla^2U$ are jointly continuous in $(\tau,x)$, and the equation holds at every point. Hence $U$ is a viscosity solution, and $|U(\tau,x)-\phi(x)|\leq C\sqrt\tau$ uniformly in $x$.
\item[(ii)] For every $\tau>0$ and at every point $x$ of the ordered sector, the rank direction $\v_*=(1,0,1,0,0)$ attains the Hamiltonian maximum,
\[
 D^2_{\v_*}U(\tau,x)=2U_\tau(\tau,x)\geq D^2_\v U(\tau,x)\qquad\text{for all }\v\in\{0,1\}^5 .
\]
\item[(iii)] $U$ has the parabolic scaling $U(\ell^2\tau,\ell x)=\ell\,U(\tau,x)$ for $\ell>0$, its Laplace transform in $\tau$ is $\lambda^{-3/2}u(\sqrt\lambda\,x)$, where $u$ is the stationary solution of \eqref{eq:stationary}, and at the origin
\begin{equation}\label{eq:constant}
 U(\tau,0)=\frac{45\pi^{3/2}}{256\sqrt2}\sqrt\tau,
 \qquad
 \nabla^2U(\tau,0)=\frac{75\pi^{3/2}}{512\sqrt{2\tau}}\Bigl(I-\tfrac15\one\one^T\Bigr).
\end{equation}
\end{enumerate}
\end{theorem}

Since $\coll$ meets Region III only in the diagonal $\{a_4=0\}$, \eqref{eq:U3} applies at the points of Region III with $a_4>0$. It agrees with \eqref{eq:U12} on the interface $z_4=0$ outside $\coll$, where both apply, and \eqref{eq:face} is the limit of both at $\coll$ (Theorem~\ref{thm:regularity}). The constant in \eqref{eq:constant} is $2/\sqrt\pi$ times the stationary value $u(0)=45\pi^2/(512\sqrt2)$ of~\cite{CD26}, and the Hessian is $(\pi\tau)^{-1/2}$ times the stationary Hessian $\tfrac53u(0)(I-\tfrac15\one\one^T)$ at the origin. As Bayraktar, Ekren and Zhang observed~\cite{BEZ20b}, the transform identity in (iii) forces the factor $2/\sqrt\pi=\Lap^{-1}[\lambda^{-3/2}](1)$ whenever a single strategy is optimal for both problems; it is the ratio between the finite-horizon and geometric-horizon values conjectured in general by Gravin, Peres and Sivan~\cite[\S1.2]{GPS16}. Since $U$ is a classical solution of \eqref{eq:forward} with $|U-\phi|\leq C\sqrt\tau$, it is the viscosity solution to which the rescaled values of the discrete game converge by Drenska and Kohn~\cite{DK20}, so that the value $V^M$ of the $M$-round game satisfies $\lim_{M\to\infty}V^M(0)/\sqrt M=U(1,0)=45\pi^{3/2}/(256\sqrt2)$.

Our second result identifies the set on which the COMB control is optimal. In decreasing rank order, COMB is $\v_C=(1,0,1,0,1)$. Define the curvature gap
\begin{equation}\label{eq:comb-gap}
 \Delta(\tau,x)=D^2_{\v_*}U(\tau,x)-D^2_{\v_C}U(\tau,x).
\end{equation}
By Theorem~\ref{thm:main}, $\Delta\geq0$ in the ordered sector, and the following theorem identifies the set where equality holds.

\begin{theorem}\label{thm:comb}
For every $\tau>0$ and every point $x$ of the ordered sector, COMB attains the Hamiltonian maximum, i.e., $\Delta(\tau,x)=0$, if and only if
\begin{equation}\label{eq:comb-set}
 x_1=x_2\quad\text{and}\quad x_3=x_4 .
\end{equation}
At every other point $\Delta>0$.
\end{theorem}

The equality set \eqref{eq:comb-set} is the same as for the stationary problem~\cite[Theorem~2.2]{CD26}. In addition, the optimal control is not unique. Corollary~\ref{cor:ties} shows that, for every $\tau>0$, the controls $(1,0,0,1,1)$ in Region I, $(1,0,0,1,0)$ in Region II, and $(1,0,0,0,1)$ and $(1,0,0,1,0)$ in Region III have the same curvature as $\v_*$; these are exactly the ties of the stationary problem~\cite[Remark~2.3]{CD26}, transferred by the Laplace transform.

The rest of the paper is organized into two sections and three appendices. Section~\ref{sec:derivation} proves the transform principle, collects the properties of the kernels $\Hk_j$, and derives the three expressions for $U$ from the mode expansions of the stationary solution. Section~\ref{sec:verification} proves that the resulting function satisfies \eqref{eq:forward} and has the properties claimed in Theorem~\ref{thm:main}, and locates the COMB equality set of Theorem~\ref{thm:comb}.  The Hamiltonian inequality in Theorem \ref{thm:main} (ii) is proved with the assistance of a computer, in the sense made precise at the beginning of Section~\ref{sec:verification}. In particular, the reduction to finitely many one-variable Gaussian series is proved in Section~\ref{sec:verification} and Appendices~\ref{app:region3} and~\ref{app:region12}---their signs are certified by exact rational arithmetic with scripts and data in the supplement~\cite{CD26Supp}, and Appendix~\ref{app:numerics} reports independent numerical checks.

%% file: derivation.tex
\section{Derivation of the finite-horizon formula}\label{sec:derivation}

This section derives the explicit formula of Theorem~\ref{thm:main}. The derivation has three steps. We first explain the transform principle, Proposition~\ref{prop:transform}, which produces from the stationary solution a solution of the Laplace-transformed problem in which the adversary is frozen at $\v_*$; its inverse transform is our candidate. We then show that every exponential mode of the stationary solution inverts, by itself, to an exact solution of the linear equation, expressed through the kernels $\Hk_j$ of \eqref{eq:H012} and \eqref{eq:Hrec}. Finally we expand the stationary solution into such modes, region by region, and read off the finite-horizon formula. The expansions are exact and every mode inverts exactly; inverting a whole series term by term needs the summability of Lemma~\ref{lem:convergence}(c), and the passage from the series to the properties of $U$ claimed in Theorem~\ref{thm:main} requires the convergence and regularity results of Section~\ref{sec:verification}, where the formula is proved to solve \eqref{eq:forward}.

\subsection{The transform principle}\label{sec:transform}

Suppose the adversary is frozen at the rank direction $\v_*$. The resulting linear problem is
\begin{equation}\label{eq:frozen}
 U_\tau=\tfrac12D^2_{\v_*}U,\qquad U(0,\cdot)=\phi,
\end{equation}
in the ordered sector with the permutation face conditions $\partial_{e_i-e_{i+1}}U=0$ on $\{x_i=x_{i+1}\}$, $1\leq i\leq4$. These are exactly the conditions under which the extension of $U$ to $\R^5$ by sorting the coordinates is $C^1$ across the walls $\{x_i=x_{i+1}\}$. Since $\Lap[U_\tau]=\lambda\widehat U-\phi$, the Laplace transform $\widehat U(\lambda,x)=\int_0^\infty e^{-\lambda\tau}U(\tau,x)\,d\tau$ of a solution should solve the resolvent equation
\begin{equation}\label{eq:resolvent}
 \lambda\widehat U-\tfrac12D^2_{\v_*}\widehat U=\phi .
\end{equation}
The following observation, which we call the transform principle, produces a solution of \eqref{eq:resolvent} from the stationary solution alone. For four experts it is due to Bayraktar, Ekren and Zhang~\cite{BEZ20b}, who also state it for the nonlinear equations when the same strategy is optimal for both problems.

\begin{proposition}\label{prop:transform}
Let $u$ be the $C^2$ solution of \eqref{eq:stationary} with $u-\phi$ bounded, and suppose $\v_*$ attains its Hamiltonian maximum at every $x$ of the sector. For $\lambda>0$ put $u^\lambda(x)=\lambda^{-3/2}u(\sqrt\lambda\,x)$. Then $u^\lambda$ solves \eqref{eq:resolvent} and satisfies the permutation face conditions.
\end{proposition}

\begin{proof}
By hypothesis $u$ solves the linear equation $u-\tfrac12D^2_{\v_*}u=\phi$. Since $\nabla^2u^\lambda(x)=\lambda^{-1/2}(\nabla^2u)(\sqrt\lambda\,x)$ and $\phi$ is positively homogeneous of degree one,
\[
 \lambda u^\lambda-\tfrac12D^2_{\v_*}u^\lambda
 =\lambda^{-1/2}\Bigl[u-\tfrac12D^2_{\v_*}u\Bigr](\sqrt\lambda\,x)
 =\lambda^{-1/2}\phi(\sqrt\lambda\,x)=\phi(x).
\]
The face conditions hold for $u$~\cite{CD26}, and they are invariant under the dilation $x\mapsto\sqrt\lambda\,x$.
\end{proof}

Proposition~\ref{prop:transform} suggests that a solution of \eqref{eq:frozen} has the Laplace transform
\begin{equation}\label{eq:transform}
 \widehat U(\lambda,x)=\lambda^{-3/2}\,u\!\left(\sqrt\lambda\,x\right),
\end{equation}
that is,
\begin{equation}\label{eq:inversion}
 U(\tau,x)=\Lap^{-1}\!\left[\lambda\mapsto\lambda^{-3/2}u(\sqrt\lambda\,x)\right](\tau),
\end{equation}
and we use \eqref{eq:inversion} as the definition of our candidate $U$. Bayraktar, Ekren and Zhang~\cite{BEZ20b} computed this inverse transform for four experts, heuristically, by moving the contour onto the branch cut, and then verified the result; we invert it mode by mode instead. It is important to point out that we do not need to prove that \eqref{eq:inversion} is the unique solution of \eqref{eq:resolvent}; instead in Section \ref{sec:verification} we prove directly that the function defined by \eqref{eq:inversion} is a classical solution of \eqref{eq:forward}. %

The most useful immediate consequence of \eqref{eq:transform} is at the origin.

\begin{corollary}\label{cor:constant}
The candidate \eqref{eq:inversion} satisfies $U(\tau,0)=\tfrac2{\sqrt\pi}\,u(0)\sqrt\tau$. For five experts $u(0)=45\pi^2/(512\sqrt2)$, which gives the value of $U(\tau,0)$ in \eqref{eq:constant}.
\end{corollary}

\begin{proof}
By homogeneity $u(\sqrt\lambda\cdot0)=u(0)$, so \eqref{eq:transform} at $x=0$ reads $\widehat U(\lambda,0)=u(0)\lambda^{-3/2}$, and $\Lap^{-1}[\lambda^{-3/2}]=2\sqrt{\tau/\pi}$. The value $u(0)$ is given in \cite[Theorem~2.1]{CD26}. 
\end{proof}
Proposition~\ref{prop:properties}(i) proves the same identity for the function of Theorem~\ref{thm:main}.

\subsection{The inversion kernels}\label{sec:kernels}

We now collect several useful properties of the kernels $\Hk_j$.

\begin{lemma}\label{lem:kernels}
The kernels defined by \eqref{eq:H012} and \eqref{eq:Hrec}, with $\Gk$ and $\Ek$ as in \eqref{eq:gE}, satisfy
\begin{equation}\label{eq:Hprops}
 \partial_X\Hk_j=-\Hk_{j+1},\qquad \partial_\tau \Hk_j=\partial_X^2\Hk_j=\Hk_{j+2}
\end{equation}
for every $j\geq0$, every real $X$ and every $\tau>0$. They satisfy \eqref{eq:Hdef} for every $j\geq0$ when $X>0$, and for $j\leq2$ also when $X=0$. Moreover $\Hk_0\geq0$ and $\Hk_1\geq0$, $\Hk_0(0,\tau)=2\sqrt{\tau/\pi}$, and for $X>0$ every $\Hk_j(X,\tau)\to0$ as $\tau\downarrow0$, faster than any power of $\tau$.
\end{lemma}

\begin{proof}
Since $\partial_X\Gk=-\tfrac X{2\tau}\Gk$ and $\partial_X\Ek=-\Gk$, differentiating \eqref{eq:H012} gives $\partial_X\Hk_j=-\Hk_{j+1}$ for $j=0,1,2$. If this holds for all indices up to some $j\geq2$, then differentiating \eqref{eq:Hrec} gives
\[
 \partial_X\Hk_{j+1}=\frac1{2\tau}\Hk_j-\frac X{2\tau}\Hk_{j+1}+\frac{j-2}{2\tau}\Hk_j
 =-\Bigl(\frac X{2\tau}\Hk_{j+1}-\frac{j-1}{2\tau}\Hk_j\Bigr)=-\Hk_{j+2},
\]
so $\Hk_{j+1}=-\partial_X\Hk_j$ for every $j$, every real $X$ and every $\tau>0$, and $\partial_X^2\Hk_j=\Hk_{j+2}$. For the time derivative, differentiating \eqref{eq:H012} gives $\partial_\tau \Ek=\tfrac X{2\tau}\Gk=\Hk_3$ and $\partial_\tau \Hk_0=\Gk=\Hk_2$, and since $\partial_\tau$ commutes with $\partial_X$, induction on $j$ with $\Hk_{j+1}=-\partial_X\Hk_j$ gives $\partial_\tau \Hk_j=\Hk_{j+2}$ for every $j$. In particular the identity holds at $X=0$, where some of the exponents below vanish on walls of the sector.

The transforms of $\Gk$, $\Ek$ and $2\tau \Gk-X\Ek$ are the classical $\lambda^{-1/2}e^{-X\sqrt\lambda}$, $\lambda^{-1}e^{-X\sqrt\lambda}$ and $\lambda^{-3/2}e^{-X\sqrt\lambda}$ for $X\geq0$, which is \eqref{eq:Hdef} for $j\leq2$. For $j\geq3$ and $X>0$ we argue by induction on $j$. On $X\geq\delta>0$ the bound \eqref{eq:Hgrowth} of Lemma~\ref{lem:Hgrowth} below, whose proof uses only \eqref{eq:H012} and \eqref{eq:Hrec}, dominates $e^{-\lambda\tau}|\Hk_j(X,\tau)|$ by an integrable function of $\tau$. So the transform of $\Hk_j=-\partial_X\Hk_{j-1}$ is $-\partial_X$ of the transform of $\Hk_{j-1}$, that is $-\partial_X\bigl(\lambda^{(j-4)/2}e^{-X\sqrt\lambda}\bigr)=\lambda^{(j-3)/2}e^{-X\sqrt\lambda}$. At $X=0$ the argument stops at $j=2$, because the bound is not integrable there once $j\geq3$. Nonnegativity of $\Hk_0$ follows from $\Hk_0=\int_X^\infty \Ek(s,\tau)\,ds$, and the small-$\tau$ statement from $\Ek(X,\tau)\leq e^{-X^2/(4\tau)}$ and the Gaussian factor of $\Gk$.
\end{proof}

The expansion in Regions I and II uses kernels $\Hk_j$ of unbounded order, so we record how fast they can grow.

\begin{lemma}\label{lem:Hgrowth}
There is an absolute constant $c_0$ such that for all $j\geq2$, $X\geq0$ and $\tau>0$,
\begin{equation}\label{eq:Hgrowth}
 \bigl|\Hk_j(X,\tau)\bigr|
 \leq\frac{c_0^{\,j}\,j^{j/2}}{\tau^{(j-1)/2}}\,e^{-X^2/(8\tau)} .
\end{equation}
Consequently, for every $Z\geq0$,
\begin{equation}\label{eq:Hfactorial}
 \frac{(4Z)^j}{j!}\bigl|\Hk_j(X,\tau)\bigr|
 \leq c_1\sqrt\tau\left(\frac{c_1Z}{\sqrt{j\tau}}\right)^{\!j}e^{-X^2/(8\tau)},
\end{equation}
which tends to $0$ faster than any geometric sequence in $j$, locally uniformly in $Z$ and $\tau>0$.
\end{lemma}

\begin{proof}
Write $\Hk_{n+2}=\Gk\,P_n$. By \eqref{eq:H012} and \eqref{eq:Hrec}, $P_0=1$, $P_1=X/(2\tau)$ and
\[
 P_{n+1}=\frac X{2\tau}P_n-\frac n{2\tau}P_{n-1},\qquad n\geq1 .
\]
Put $a=X/(2\tau)$, $b_n=\sqrt{n/(2\tau)}$ and $A_n=a+b_n$, which increases with $n$. We claim that $|P_n|\leq A_n^n$. This holds for $n=0,1$, and if it holds up to $n$, then
\[
 A_{n+1}^{n+1}=a\,A_{n+1}^n+b_{n+1}A_{n+1}\,A_{n+1}^{n-1}
 \geq a\,A_n^n+b_{n+1}^2A_{n-1}^{n-1}
 \geq a|P_n|+\frac n{2\tau}|P_{n-1}|\geq|P_{n+1}| .
\]
With $w=X/\sqrt{2\tau}$ we have $A_n^n=(2\tau)^{-n/2}(w+\sqrt n)^n$ and $e^{-X^2/(4\tau)}=e^{-w^2/2}$, and since $w^ne^{-w^2/4}\leq(2n/e)^{n/2}$,
\[
 (w+\sqrt n)^ne^{-w^2/4}\leq2^n\max\bigl(w,\sqrt n\bigr)^ne^{-w^2/4}\leq2^n(2n)^{n/2}.
\]
Hence
\[
 |\Hk_{n+2}|\leq\frac{2^n(2n)^{n/2}}{\sqrt{\pi\tau}\,(2\tau)^{n/2}}\,e^{-X^2/(8\tau)}
 =\frac{2^n\,n^{n/2}}{\sqrt\pi\,\tau^{(n+1)/2}}\,e^{-X^2/(8\tau)},
\]
which is \eqref{eq:Hgrowth} with $j=n+2$ after absorbing constants. For \eqref{eq:Hfactorial} use $j!\geq(j/e)^j$, which gives it with $c_1=\max(1,4ec_0)$, since $\tau^{-(j-1)/2}=\sqrt\tau\,\tau^{-j/2}$.
\end{proof}

The key structural fact is that every mode of the expansions below is by itself an exact solution of the linear problem.

\begin{lemma}\label{lem:mode}
Let $X(x)=d\cdot x$ be linear with $d\cdot\v_*=\pm k$, and let $P$ be a polynomial with $\partial_{\v_*}P=0$. Then for every $j\geq0$ the function $\Psi(\tau,x)=P(x)\Hk_j(X(x),\tau)$ satisfies
\begin{equation}\label{eq:modeheat}
 \partial_\tau\Psi=\tfrac12D^2_{\v_*}\Psi .
\end{equation}
Moreover, if $P$ is homogeneous of degree $p$, then at every $x$ with $X(x)>0$, and for $p\leq2$ also at every $x$ with $X(x)=0$,
\begin{equation}\label{eq:modetransform}
 \Lap\bigl[P\,\Hk_p(X,\cdot)\bigr](\lambda)=\lambda^{-3/2}\,P(\sqrt\lambda\,x)\,e^{-X(\sqrt\lambda\,x)} .
\end{equation}
\end{lemma}

\begin{proof}
Because $\partial_{\v_*}P=0$ and $\partial_{\v_*}X=d\cdot\v_*$,
\[
 D^2_{\v_*}\bigl[P\Hk_j\bigr]=P\,(d\cdot\v_*)^2\,\partial_X^2\Hk_j=k^2P\,\Hk_{j+2}=2P\,\Hk_{j+2},
\]
which equals $2\partial_\tau\Psi$ by \eqref{eq:Hprops}. For \eqref{eq:modetransform}, $P(\sqrt\lambda\,x)=\lambda^{p/2}P(x)$ and $\lambda^{-3/2}\lambda^{p/2}=\lambda^{(p-3)/2}$, which under the stated conditions on $X(x)$ is the transform of $\Hk_p(X(x),\cdot)$ by Lemma~\ref{lem:kernels}.
\end{proof}

Combining the definition \eqref{eq:inversion} with Lemma~\ref{lem:mode} gives the recipe we follow for the rest of this section.

\begin{corollary}\label{cor:recipe}
Suppose that on a region of the sector the stationary solution can be written as
\begin{equation}\label{eq:modeexpansion}
 u(x)=x_1+\frac1k\sum_{\iota}C_\iota\,P_\iota(x)\,e^{-X_\iota(x)},
\end{equation}
with $X_\iota$ linear, $X_\iota\geq0$ on the region, $X_\iota>0$ on the region whenever $p_\iota\geq3$, $d_\iota\cdot\v_*=\pm k$, and $P_\iota$ homogeneous of degree $p_\iota$ with $\partial_{\v_*}P_\iota=0$. Let $x$ be a point of the region at which the series in \eqref{eq:inverted} below converges for every $\tau>0$ to a continuous function of $\tau$, and at which, for every $\lambda>0$,
\[
 \sum_\iota\int_0^\infty e^{-\lambda\tau}\bigl|C_\iota\,P_\iota(x)\,\Hk_{p_\iota}(X_\iota(x),\tau)\bigr|\,d\tau<\infty .
\]
Then
\begin{equation}\label{eq:inverted}
 U(\tau,x)=x_1+\frac1k\sum_\iota C_\iota\,P_\iota(x)\,\Hk_{p_\iota}(X_\iota(x),\tau),
\end{equation}
the inverse transform in \eqref{eq:inversion} being the continuous function of $\tau$ with the transform \eqref{eq:transform}. Every term of \eqref{eq:inverted} satisfies \eqref{eq:modeheat}. If moreover $X_\iota>0$ at a point $x$ of the region for every $\iota$, and the series converges uniformly for small $\tau$ near $x$, then $U(\tau,\cdot)\to x_1=\phi$ as $\tau\downarrow0$ near $x$.
\end{corollary}

\begin{proof}
The summability lets the Laplace transform of the series in \eqref{eq:inverted} be taken term by term. By \eqref{eq:modetransform}, and by \eqref{eq:modeexpansion} at $\sqrt\lambda\,x$, which lies in the region because the regions are cones, the result is $\lambda^{-3/2}u(\sqrt\lambda\,x)$; since a continuous function is determined by its Laplace transform, this gives \eqref{eq:inverted}. The terms satisfy \eqref{eq:modeheat} by Lemma~\ref{lem:mode}, and the last statement holds because $\Hk_j(X,\tau)\to0$ faster than any power of $\tau$ for each fixed $X>0$.
\end{proof}

\begin{remark}\label{rem:terminal}
Nonnegativity of every exponent $X_\iota$ is what makes the terminal condition come out of the series at all. It is a nontrivial structural property of the expansions below: several families of modes that appear formally with negative exponents turn out to have identically vanishing coefficients (Lemma~\ref{lem:vanishing}). Nonnegativity alone is not quite enough for a pointwise statement, since $\Hk_1(0,\tau)=1$, $\Hk_2(0,\tau)=(\pi\tau)^{-1/2}$ is unbounded, and for $j\geq3$ the recursion \eqref{eq:Hrec} gives $\Hk_j(0,\tau)=0$ for odd $j$ and, for even $j$, nonzero multiples of $\tau^{(1-j)/2}$ with alternating signs; of the kernels that are not identically zero at $X=0$, only $\Hk_0(0,\tau)$ tends to zero as $\tau\downarrow0$. The exponents do vanish somewhere, since on the diagonal every $X_\iota$ is zero, and there the Region III series diverges. We therefore do not prove the terminal condition from the series: it is Theorem~\ref{thm:regularity}(d), which proves the stronger, uniform bound $|U(\tau,x)-\phi(x)|\leq C\sqrt\tau$ from the contour representation of Section~\ref{sec:collision}.
\end{remark}

\subsection{Four experts}\label{sec:four}

The four-expert problem was solved by Bayraktar, Ekren and Zhang, in the geometric case in~\cite{BEZ20} and in the finite-horizon case in~\cite{BEZ20b}. Nevertheless, we apply the recipe from this paper to the four-expert problem for two reasons: the mode expansion of the stationary four-expert solution is the background term in Regions~I and~II, and inverting it mode by mode gives the finite-horizon solution of~\cite{BEZ20b} as a series. Throughout this subsection $\v_*$ means the four-expert rank direction $(1,0,1,0)$, which is the restriction of the five-expert $\v_*$ to the first four coordinates. In the coordinates \eqref{eq:coords}, which involve only $x_1,\dots,x_4$, the ordered four-expert sector is $z_1\geq z_2\geq|z_3|$ and the stationary solution is $u_4=x_1+\tfrac1kF_4$ with $F_4$ as in \eqref{eq:F4}.

\begin{lemma}\label{lem:F4expansion}
For $z_1\geq z_2\geq|z_3|\geq0$ with $z_1>0$,
\begin{equation}\label{eq:F4series}
 F_4=\sum_{m\geq0}\ \sum_{\substack{\epsilon\in\{\pm1\}^3\\ \epsilon_1\epsilon_2\epsilon_3=(-1)^m}}
      c^{(4)}_{m,\epsilon}\,e^{-X^{(4)}_{m,\epsilon}},
\end{equation}
with $c^{(4)}_{m,\epsilon}$ and $X^{(4)}_{m,\epsilon}$ as in \eqref{eq:c4D4}, the series converging absolutely. Every exponent occurring with a nonzero coefficient satisfies $X^{(4)}_{m,\epsilon}\geq0$, and $\partial_{\v_*}X^{(4)}_{m,\epsilon}=-\epsilon_2k$.
\end{lemma}

\begin{proof}
Let $z_1>0$, and expand the two transcendental factors of \eqref{eq:F4} in their exponential series. With $n=2m+1$,
\[
\begin{gathered}
 \theta(z_1)\cosh z_1=\tfrac12\sum_{m\geq0}\frac{(-1)^m}{n}\sum_{\epsilon_1=\pm1}e^{-(n-\epsilon_1)z_1},\\
 \tfrac12\Lam(z_1)\sinh z_1=\tfrac12\sum_{m\geq0}\frac1{n}\sum_{\epsilon_1=\pm1}\epsilon_1e^{-(n-\epsilon_1)z_1},
\end{gathered}
\]
while $\cosh z_2\cosh z_3=\tfrac14\sum_{\epsilon_2,\epsilon_3}e^{\epsilon_2z_2+\epsilon_3z_3}$ and $\sinh z_2\sinh z_3=\tfrac14\sum_{\epsilon_2,\epsilon_3}\epsilon_2\epsilon_3e^{\epsilon_2z_2+\epsilon_3z_3}$. Adding the two products gives
\[
 \tfrac18\sum_{m\geq0}\frac1{2m+1}\sum_{\epsilon\in\{\pm1\}^3}
   \bigl[(-1)^m+\epsilon_1\epsilon_2\epsilon_3\bigr]e^{-X^{(4)}_{m,\epsilon}},
\]
an absolutely convergent series in which the bracket is $2(-1)^m$ when $\epsilon_1\epsilon_2\epsilon_3=(-1)^m$ and $0$ otherwise. This is \eqref{eq:F4series} with the stated coefficients, except for the $m=0$ entries. The single term with $\epsilon=(1,1,1)$ has $X^{(4)}_{0,(1,1,1)}=-(z_2+z_3)\leq0$ and coefficient $\tfrac14$, that is $\tfrac14e^{z_2+z_3}$; the remaining elementary term of \eqref{eq:F4} is $-\tfrac12\sinh(z_2+z_3)=-\tfrac14e^{z_2+z_3}+\tfrac14e^{-(z_2+z_3)}$. The growing exponentials cancel identically, and the surviving $\tfrac14e^{-(z_2+z_3)}$ adds to the coefficient $\tfrac14$ of the mode $\epsilon=(1,-1,-1)$, which carries the same exponent $X^{(4)}_{0,(1,-1,-1)}=z_2+z_3$. This is the stated modification.

For the sign of the exponents, if $m\geq1$ then $X^{(4)}_{m,\epsilon}\geq2mz_1-z_2-|z_3|\geq2mz_1-2z_1\geq0$; if $m=0$ and $\epsilon_1=-1$ then $X^{(4)}\geq2z_1-z_2-|z_3|\geq0$; and the two remaining $m=0$ modes are $\epsilon=(1,1,1)$, whose coefficient is now zero, and $\epsilon=(1,-1,-1)$, with $X^{(4)}=z_2+z_3\geq0$. Finally $\partial_{\v_*}z_1=\partial_{\v_*}z_3=0$ and $\partial_{\v_*}z_2=k$, which gives $\partial_{\v_*}X^{(4)}_{m,\epsilon}=-\epsilon_2k$.
\end{proof}

\begin{proposition}\label{prop:four}
For $x_1\geq x_2\geq x_3\geq x_4$ not all equal and $\tau>0$, the solution $U_4$ of the finite-horizon four-expert problem found in~\cite{BEZ20b} is
\begin{equation}\label{eq:U4}
 U_4(\tau,x)=x_1+\frac1k\sum_{m\geq0}\ \sum_{\substack{\epsilon\in\{\pm1\}^3\\ \epsilon_1\epsilon_2\epsilon_3=(-1)^m}}
   c^{(4)}_{m,\epsilon}\,\Hk_0\bigl(X^{(4)}_{m,\epsilon},\tau\bigr),
\end{equation}
the series converging absolutely and uniformly on compact sets. At the origin,
\begin{equation}\label{eq:U4origin}
 U_4(\tau,0)=\frac1k\cdot\frac\pi4\cdot2\sqrt{\tau/\pi}=\tfrac12\sqrt{\pi\tau/2}.
\end{equation}
\end{proposition}

\begin{proof}
The coordinates are not all equal exactly when $z_1>0$, since $z_1\geq z_2\geq|z_3|$. Absolute convergence follows from $X^{(4)}_{m,\epsilon}\geq2(m-1)z_1$ for $m\geq1$, $0\leq \Hk_0(X,\tau)\leq2\sqrt{\tau/\pi}\,e^{-X^2/(4\tau)}$ and $|c^{(4)}_{m,\epsilon}|\leq\tfrac1{2(2m+1)}$, which covers the exceptional coefficient $c^{(4)}_{0,(1,-1,-1)}=\tfrac12$. Since $\Hk_0\geq0$ has the transform $\lambda^{-3/2}e^{-X\sqrt\lambda}$, the same bounds let the Laplace transform of \eqref{eq:U4} be taken term by term, and by Lemmas~\ref{lem:mode} and~\ref{lem:F4expansion} it is $\lambda^{-3/2}u_4(\sqrt\lambda\,x)$. Bayraktar, Ekren and Zhang prove that $U_4$ is a classical solution with the same Laplace transform~\cite{BEZ20b}; both functions are continuous in $\tau$, so they coincide. At $x=0$ the transform is $\lambda^{-3/2}u_4(0)$ with $u_4(0)=F_4(0,0,0)/k=\pi/(4k)$, which gives \eqref{eq:U4origin}, the value found in~\cite{BEZ20b}.
\end{proof}

Formula \eqref{eq:U4} is the four-expert solution of Bayraktar, Ekren and Zhang~\cite{BEZ20b} in a different form. They invert the same transform on the branch cut, which gives an oscillatory integral $\int_{-\infty}^\infty r^{-2}e^{-\tau r^2}W(r,x)\,dr$ whose integrand contains the $2\pi$-periodic square wave and whose convergence at $r=0$ rests on a cancellation. Inverting mode by mode gives instead a series of Gaussians and complementary error functions with explicit rational coefficients and nonnegative exponents, whose terms are bounded, so that no singular terms need to cancel.

\subsection{Region III}\label{sec:region3}

In Region III the stationary solution is $u=x_1+\tfrac1kF$ with
\begin{equation}\label{eq:F3}
 F=\sum_{j=0}^3b_j\,\sigma_j(a_1,a_2,a_3),
\end{equation}
where $\sigma_j$ is the sum of the $\binom3j$ products of $j$ factors $\sinh a_i$ and $3-j$ factors $\cosh a_i$, and $b_j=\Kop^je$ evaluated at $a_4$. Here
\begin{equation}\label{eq:operators}
 (\Iop f)(X)=2\sinh^2X\int_X^\infty\frac{f(t)}{\sinh^3t}\dd,
 \qquad
 \Kop=-\coth X+\csch X\,\Iop
\end{equation}
are the operators of \cite[Section~3]{CD26}, and $e$ is the trace \eqref{eq:equad}. Only the value $e(a_4)$ requires the quadrature: the derivatives of $e$ in the $b_j$ are eliminated by the compatibility relation $6b_1=e'-3$ and the trace equation \eqref{eq:eode} below. Explicitly, \cite[Section~3.6]{CD26} gives
\begin{equation}\label{eq:bcoefficients}
\begin{gathered}
 b_0=e,\qquad 6b_1=e'-3,\qquad 42b_2=e''+6e+18\Lam\sinh L,\\
 336b_3=e'''+20e'+144\Lam\cosh L-312,
\end{gathered}
\end{equation}
all evaluated at $L=a_4$; these follow from the two integration-by-parts identities $\partial_X\Iop=\Iop(\partial_X+\Kop)$ and $\Kop^2-[\partial_X,\Kop]=\mathrm{Id}$ together with $\Kop1=-\Lam\sinh X$ and $\Kop^21=2-\Lam\cosh X$. We begin with the exponential expansion of the trace.

\begin{proposition}\label{prop:trace}
For $L>0$,
\begin{equation}\label{eq:eseries}
 e(L)=\sum_{m\geq0}\bigl(A_m+B_mL\bigr)e^{-2mL},
\end{equation}
with $A_m$ and $B_m$ as in \eqref{eq:AB}. The series converges for every $L>0$ and diverges at $L=0$, where $e(0)=45\pi^2/512$ instead.
\end{proposition}

\begin{proof}
Differentiating \eqref{eq:equad} once gives the first-order equation
\begin{equation}\label{eq:efirst}
 \cosh L\,e'(L)-\sinh L\,e(L)=-\frac{3I(L)}{\sinh^5L},
\end{equation}
which is \eqref{eq:tanhe} of Appendix~\ref{app:region12}. Insert the ansatz \eqref{eq:eseries}. With $w=e^{-2L}$, the $m$-th summand $(A_m+B_mL)w^m$ contributes $\bigl(B_m-2m(A_m+B_mL)\bigr)w^m$ to $e'$, while $2e^{-L}\cosh L=1+w$ and $2e^{-L}\sinh L=1-w$, and the closed form of $I$ in the proof of Lemma~\ref{lem:p} gives
\[
 \frac{6e^{-L}I(L)}{\sinh^5L}=\frac{N(w)+12Lw^3}{(1-w)^5},
 \qquad
 N(w)=\tfrac12\bigl(1-3w-3w^2+3w^4+3w^5-w^6\bigr).
\]
Multiplying \eqref{eq:efirst} by $2e^{-L}$ and collecting the coefficient of $w^n$ gives, with $A_{-1}=B_{-1}=0$,
\begin{equation}\label{eq:efirstrec}
\begin{aligned}
 (2n+1)B_n+(2n-3)B_{n-1}&=\tfrac12(n+1)n(n-1)(n-2),\\
 (2n+1)A_n+(2n-3)A_{n-1}&=B_n+B_{n-1}+\nu_n,
\end{aligned}
\end{equation}
where $\nu_n$ is the coefficient of $w^n$ in the Taylor series of $N(w)(1-w)^{-5}$ at $w=0$, that is, $\nu_0=\tfrac12$, $\nu_1=1$ and $\nu_n=-\tfrac12(2n-1)(n^2-n-1)$ for $n\geq2$. The coefficient $2n+1$ never vanishes, so \eqref{eq:efirstrec} determines every coefficient. It gives $B_0=B_1=B_2=0$, $B_3=\tfrac{12}7$, $A_0=A_1=\tfrac12$, $A_2=-\tfrac25$ and, from $7A_3=B_3+B_2+\nu_3-3A_2=\tfrac{12}7-\tfrac{25}2+\tfrac65$,
\begin{equation}\label{eq:A3}
 A_3=-\frac{671}{490}.
\end{equation}
The closed forms \eqref{eq:AB} satisfy both relations of \eqref{eq:efirstrec} for every $n$: after clearing denominators each is a polynomial identity in $n$, and $n=0,1$ are checked directly. Since $A_m=\O(m^2)$ and $B_m=\O(m^3)$, the series \eqref{eq:eseries} converges for $|w|<1$ together with its termwise derivative, so it solves \eqref{eq:efirst} for $L>0$, and it tends to $A_0=\tfrac12$ as $L\to\infty$. The trace $e$ solves the same equation and is bounded, so the difference of the two solves $y'=\tanh L\,y$ and is a bounded multiple of $\cosh L$, hence zero. This proves \eqref{eq:eseries} for $L>0$. At $L=0$ the terms $A_m\sim-\tfrac3{16}m^2$ do not tend to zero, while $e(0)=45\pi^2/512$~\cite[Lemma~3.1]{CD26}.
\end{proof}

\begin{remark}\label{rem:A3}
Differentiating \eqref{eq:efirst} once more, and using $I'=\sinh^4X\cosh^2X$ and \eqref{eq:efirst} itself, gives the second-order trace equation $e''+5\coth X\,e'-6e=-3\coth X$. After multiplication by $\sinh X$ it reads
\begin{equation}\label{eq:eode}
 \sinh L\,e''+5\cosh L\,e'-6\sinh L\,e+3\cosh L=0 ,
\end{equation}
and inserting \eqref{eq:eseries} gives the recursion
\begin{equation}\label{eq:erec}
\begin{aligned}
 &(5-4n)B_n+2(2n+1)(n-3)(A_n+B_nL)\\
 &\qquad+(4n+1)B_{n-1}-2(2n-3)(n+2)(A_{n-1}+B_{n-1}L)+3\delta_{n0}+3\delta_{n1}=0 .
\end{aligned}
\end{equation}
Its constant part is resonant at $n=3$, where the coefficient of $A_3$ vanishes, so \eqref{eq:erec} determines every coefficient except $A_3$, the amplitude of the homogeneous mode $e^{-6L}$; equivalently, it determines $e$ only up to a multiple of a decaying solution of \eqref{eq:eode}. This is what a local analysis of \eqref{eq:eode} cannot see. The first-order equation \eqref{eq:efirst}, which remembers the lower limit of the integral $I$, has no resonance and fixes $A_3$. The value \eqref{eq:A3} carries the same information as $e(0)=45\pi^2/512$, which comes from the closed evaluation of the quadrature in \cite[Lemma~3.1]{CD26} and ultimately from $\int_0^\infty s\csch s\,ds=\pi^2/4$. As independent checks, we have compared \eqref{eq:eseries} with the quadrature \eqref{eq:equad} to $40$ digits at $L=0.35,0.6,1,2$, and \eqref{eq:erec} with \eqref{eq:AB} as exact rationals for every $m\leq200$.
\end{remark}

Expanding each $b_j$ in the same way, write
\begin{equation}\label{eq:bexpansion}
 b_j(L)=\sum_{m\geq0}\bigl(\alpha^{(j)}_m+\beta^{(j)}_mL\bigr)e^{-2mL},\qquad j=0,1,2,3 .
\end{equation}
Using \eqref{eq:bcoefficients} and $\Lam(L)\sinh L=\sum_mJ_me^{-2mL}$, $\Lam(L)\cosh L=\sum_mK_me^{-2mL}$ with $J_m,K_m$ as in \eqref{eq:JK}, we obtain exactly the coefficients \eqref{eq:alphabeta}: differentiating \eqref{eq:eseries} term by term replaces $(A_m,B_m)$ by $(B_m-2mA_m,-2mB_m)$, and each further derivative repeats this substitution. Two properties of the resulting coefficients $c^a_{m,\epsilon},c^b_{m,\epsilon}$ of \eqref{eq:cab} are what make the finite-horizon formula work.

\begin{lemma}\label{lem:vanishing}
$c^b_{m,\epsilon}=0$ for $m\leq2$ and all $\epsilon$. Moreover
\begin{equation}\label{eq:van}
 c^a_{0,\epsilon}=\tfrac12\prod_{i=1}^3(1-\epsilon_i),
 \qquad
 c^a_{1,\epsilon}=\tfrac43\,\one\{\epsilon_1+\epsilon_2+\epsilon_3=-1\},
 \qquad
 c^a_{2,\epsilon}=0\ \text{ if }\ \epsilon_1\epsilon_2\epsilon_3=1 .
\end{equation}
Consequently $c^a_{0,\epsilon}=c^b_{0,\epsilon}=0$ unless $\epsilon=(-1,-1,-1)$, where $c^a_{0,\epsilon}=4$, and $c^a_{1,(1,1,1)}=c^b_{1,(1,1,1)}=0$.
\end{lemma}

\begin{proof}
Every $\beta^{(j)}_m$ is a multiple of $B_m$, and $B_0=B_1=B_2=0$. For the rest, \eqref{eq:alphabeta} gives $\bigl(\alpha^{(j)}_0\bigr)_j=(\tfrac12,-\tfrac12,\tfrac12,-\tfrac12)$, so $c^a_{0,\epsilon}=\tfrac12(s_0-s_1+s_2-s_3)=\tfrac12\prod(1-\epsilon_i)$; $\bigl(\alpha^{(j)}_1\bigr)_j=(\tfrac12,-\tfrac16,-\tfrac16,\tfrac12)$, so $c^a_{1,\epsilon}=\tfrac12(1+s_3)-\tfrac16(s_1+s_2)$, which is checked to be $\tfrac43$ when $s_1=-1$ and $0$ in the three other cases; and $\bigl(\alpha^{(j)}_2\bigr)_j=(-\tfrac25,\tfrac4{15},-\tfrac4{15},\tfrac25)$, so $c^a_{2,\epsilon}=-\tfrac25(1-s_3)+\tfrac4{15}(s_1-s_2)$, which vanishes when $s_3=1$.
\end{proof}

\begin{proposition}\label{prop:region3}
On Region III with $a_4>0$, for $\tau>0$, the candidate \eqref{eq:inversion} is given by \eqref{eq:U3}. Every exponent occurring with a nonzero coefficient satisfies $X_{m,\epsilon}\geq0$ on Region III, and $\partial_{\v_*}X_{m,\epsilon}=-\epsilon_3k$, $\partial_{\v_*}a_4=0$.
\end{proposition}

\begin{proof}
Expanding each $\cosh a_i$ and $\sinh a_i$ in \eqref{eq:F3} gives $\sigma_j=\tfrac18\sum_\epsilon s_j(\epsilon)e^{\epsilon\cdot a}$, so by \eqref{eq:bexpansion} and \eqref{eq:cab},
\[
 F=\tfrac18\sum_{m\geq0}\sum_\epsilon\bigl(c^a_{m,\epsilon}+c^b_{m,\epsilon}a_4\bigr)e^{-X_{m,\epsilon}} .
\]
The prefactors are $1$ and $a_4$, homogeneous of degrees $0$ and $1$, so Corollary~\ref{cor:recipe} yields \eqref{eq:U3} with $\Hk_0$ and $\Hk_1$; its hypotheses of continuity in $\tau$ and summability of the transforms are Lemma~\ref{lem:convergence}(a) and~(c). For the exponents: if $m\geq2$ then $\epsilon\cdot a\leq a_1+a_2+a_3\leq3a_4\leq2ma_4$. If $m=1$ then the only sign vector with $\epsilon\cdot a>2a_4$ possible is $\epsilon=(1,1,1)$, whose coefficients vanish by Lemma~\ref{lem:vanishing}; for the other seven, $\epsilon\cdot a\leq2a_4$ because at most two of the $a_i$ enter with a $+$ sign and each is at most $a_4$. If $m=0$ then only $\epsilon=(-1,-1,-1)$ has a nonzero coefficient, and then $X_{0,\epsilon}=a_1+a_2+a_3=y_1\geq0$. Finally the direction $\v_*$ moves $y$ along $k(1,-1,1,0)$, so $\partial_{\v_*}a_1=\partial_{\v_*}a_2=\partial_{\v_*}a_4=0$ and $\partial_{\v_*}a_3=k$.
\end{proof}

The lowest modes are worth recording. By Lemma~\ref{lem:vanishing}, at $m=0$ the only mode is $c^a=4$ at $\epsilon=(-,-,-)$, with $X_{0,\epsilon}=y_1$; at $m=1$ one has $c^a=\tfrac43$ at the three $\epsilon$ with one $+$ sign; at $m=2$, $c^a=-\tfrac{12}5$ at $(-,-,-)$ and $c^a=-\tfrac4{15}$ at the three $\epsilon$ with two $+$ signs; and at $m=3$ the linear prefactor first appears, with $c^b=\tfrac{96}7$ at $\epsilon=(-,-,-)$.

\subsection{Regions I and II}\label{sec:region12}

In Regions I and II the stationary solution is $u=x_1+\tfrac1kF$ with
\begin{equation}\label{eq:F12}
 F=F_4(z_1,z_2,z_3)+\Delta F,
 \qquad
 \Delta F=\int_{z_1}^\infty\sinh t\,p(t)\,e^{-2z_4\coth t}\,
      \Phi_t(z_1)\Phi_t(z_2)\Phi_t(|z_3|)\dd,
\end{equation}
where $\Phi_t(X)=\sinh(t-X)/\sinh t$ for $0\leq X\leq t$ and $0$ otherwise, and $p$ is the density \eqref{eq:pdef}~\cite[Theorem~2.1]{CD26}. The background $F_4$ was treated in Section~\ref{sec:four}; it remains to expand $\Delta F$.

A remark on the coordinates before we begin. The formula \eqref{eq:F12} involves $|z_3|$ and $z_4=y_4-\max(z_3,0)$, which are only piecewise linear in $x$: they are linear on Region~I, where $z_3\leq0$ and $z_4=y_4$, and linear on Region~II, where $z_3\geq0$ and $z_4=y_4-z_3$. Everything below is therefore carried out on each of the two open regions separately, so that the hypotheses of Lemma~\ref{lem:mode}, linear exponents and homogeneous prefactors, are met. The two resulting expressions agree across the interface $z_3=0$, because the correction depends on $z_3$ only through $|z_3|$ and because $z_4=y_4$ there. We first put the density in exponential form.

\begin{lemma}\label{lem:p}
With $w=e^{-2t}$ and $\Pi(w)=1-14w-94w^2-14w^3+w^4$, for $t>0$,
\begin{equation}\label{eq:prational}
 p(t)=\frac{w\,\Pi(w)}{2(1+w)(1-w)^5}+\frac{60\,t\,w^3}{(1-w)^6}.
\end{equation}
\end{lemma}

\begin{proof}
Using $\sinh(2\ell t)=\tfrac12e^{2\ell t}(1-w^\ell)$ and $\sinh^6t=\tfrac1{64}e^{6t}(1-w)^6$ in $I(X)=\tfrac1{192}\sinh6X-\tfrac1{64}\sinh4X-\tfrac1{64}\sinh2X+\tfrac X{16}$ gives
\[
 \begin{aligned}
 \frac{15I(t)}{\sinh^6t}
 &=\frac{\tfrac52\bigl(1-3w-3w^2+3w^4+3w^5-w^6\bigr)+60tw^3}{(1-w)^6}\\
 &=\frac{5\bigl(1-2w-5w^2-5w^3-2w^4+w^5\bigr)}{2(1-w)^5}
   +\frac{60tw^3}{(1-w)^6},
 \end{aligned}
\]
the last step by dividing the sextic by $1-w$. Adding $\tfrac12\tanh t=\tfrac{1-w}{2(1+w)}$ and $-3\coth t=-\tfrac{3(1+w)}{1-w}$ over the common denominator $2(1+w)(1-w)^5$ produces the numerator
\[
 (1-w)^6-6(1+w)^2(1-w)^4+5(1+w)\bigl(1-2w-5w^2-5w^3-2w^4+w^5\bigr)
 =w\,\Pi(w),
\]
which is \eqref{eq:prational}.
\end{proof}

\begin{lemma}\label{lem:dF}
For $z_1>0$, $z_1\geq z_2\geq|z_3|\geq0$, and $z_4\geq0$,
\begin{equation}\label{eq:dFseries}
 \Delta F=\frac12\sum_{\epsilon_2,\epsilon_3\in\{\pm1\}}\epsilon_2\epsilon_3
   \sum_{j\geq0}\frac{(-4z_4)^j}{j!}\sum_{n\geq j+2}
   \Bigl[\alpha_{j,n}(c)+\beta_{j,n}(c)\,z_1\Bigr]e^{-X_{j,n,\epsilon}},
\end{equation}
with $c$, $X_{j,n,\epsilon}$, $\alpha_{j,n}$ and $\beta_{j,n}$ as in \eqref{eq:cD12}--\eqref{eq:ab12}. All exponents satisfy $X_{j,n,\epsilon}\geq2z_1>0$, and $\partial_{\v_*}X_{j,n,\epsilon}=\epsilon_2k$ while $\partial_{\v_*}z_1=\partial_{\v_*}z_4=0$.
\end{lemma}

\begin{proof}
Write $d=|z_3|$ and expand the three ingredients of the integrand in $w=e^{-2t}$. First,
\[
 \sinh t\,\Phi_t(z_1)\Phi_t(z_2)\Phi_t(d)
 =\frac{\prod_{i}\sinh(t-\zeta_i)}{\sinh^2t}
 =\frac12\sum_{\epsilon\in\{\pm1\}^3}\Bigl(\prod_i\epsilon_i\Bigr)e^{-\epsilon\cdot\zeta}
   \,\frac{e^{\sigma t}w}{(1-w)^2},
\]
where $\zeta=(z_1,z_2,d)$ and $\sigma=\epsilon_1+\epsilon_2+\epsilon_3$, using $\prod_i\sinh(t-\zeta_i)=\tfrac18\sum_\epsilon(\prod\epsilon_i)e^{\sigma t}e^{-\epsilon\cdot\zeta}$ and $\sinh^{-2}t=4w(1-w)^{-2}$. Second, $\coth t=1+2w/(1-w)$ gives
\[
 e^{-2z_4\coth t}=e^{-2z_4}\sum_{j\geq0}\frac{(-4z_4)^j}{j!}\frac{w^j}{(1-w)^j},
\]
which converges absolutely for $t\geq z_1>0$. Third, Lemma~\ref{lem:p} gives
\[
 \frac{w^{1+j}\,p(t)}{(1-w)^{2+j}}
 =\frac{w^{j+2}\,\Pi(w)}{2(1+w)(1-w)^{j+7}}+\frac{60\,t\,w^{j+4}}{(1-w)^{j+8}}
 =\sum_{n}\bigl(r_{j,n}+t\,s_{j,n}\bigr)w^n .
\]
Since $e^{\sigma t}w^n=e^{-\nu t}$ with $\nu=2n-\sigma$, and $\nu\geq2(j+2)-3>0$ on the support of $r_{j,\cdot}$, the $t$-integrals are elementary:
\[
 \int_{z_1}^\infty e^{-\nu t}\dd=\frac{e^{-\nu z_1}}\nu,
 \qquad
 \int_{z_1}^\infty t\,e^{-\nu t}\dd=\frac{e^{-\nu z_1}}\nu\Bigl(z_1+\frac1\nu\Bigr).
\]
Assembling, the $\epsilon$-term carries the exponential $e^{-\nu z_1-\epsilon\cdot\zeta-2z_4}$, that is $e^{-X_{j,n,\epsilon}}$ with $X_{j,n,\epsilon}=(\nu+\epsilon_1)z_1+\epsilon_2z_2+\epsilon_3d+2z_4$. Because $\nu+\epsilon_1=2n-\epsilon_2-\epsilon_3=c$, the exponent does not depend on $\epsilon_1$, so the two branches $\nu=c\mp1$ may be summed first. With the signs $\prod_i\epsilon_i=\epsilon_1\epsilon_2\epsilon_3$,
\[
 \sum_{\epsilon_1=\pm1}\frac{\epsilon_1}{c-\epsilon_1}=\frac2{c^2-1},
 \qquad
 \sum_{\epsilon_1=\pm1}\frac{\epsilon_1}{(c-\epsilon_1)^2}=\frac{4c}{(c^2-1)^2},
\]
which give exactly \eqref{eq:ab12} and \eqref{eq:dFseries}.

For the sign of the exponents, recall that $c=2n-\epsilon_2-\epsilon_3$ with $n\geq2$. If $\epsilon_2=\epsilon_3=1$ then $X_{j,n,\epsilon}\geq c\,z_1\geq2z_1$. Otherwise $c\geq2n\geq4$, and $z_2\leq z_1$, $|z_3|\leq z_1$, $z_4\geq0$ give
\[
 X_{j,n,\epsilon}\geq c\,z_1-z_2-|z_3|+2z_4\geq(c-2)z_1\geq2z_1 .
\]
Finally $\partial_{\v_*}z_1=\partial_{\v_*}z_3=\partial_{\v_*}z_4=0$ and $\partial_{\v_*}z_2=k$.
\end{proof}

\begin{proposition}\label{prop:region12}
In Regions I and II with $z_1>0$, for $\tau>0$, the candidate \eqref{eq:inversion} is given by \eqref{eq:U12}.
\end{proposition}

\begin{proof}
Apply Corollary~\ref{cor:recipe} to \eqref{eq:F12} using Lemmas~\ref{lem:F4expansion} and~\ref{lem:dF}. The prefactor $z_4^j$ is homogeneous of degree $j$ and $z_4^jz_1$ of degree $j+1$, and both are annihilated by $\partial_{\v_*}$, which produces the kernels $\Hk_j$ and $\Hk_{j+1}$ respectively; their exponents are positive by Lemma~\ref{lem:dF}, as Corollary~\ref{cor:recipe} requires for the kernels of order at least three. The continuity in $\tau$ and the summability of the transforms that Corollary~\ref{cor:recipe} also asks for are Lemma~\ref{lem:convergence}(b) and~(c).
\end{proof}

The correction in \eqref{eq:U12} contains $z_3$ only through $|z_3|$, its background is the four-expert formula, valid on the whole four-expert sector, and at $z_3=0$ one has $z_4=y_4$ in both regions, so the single expression \eqref{eq:U12} serves Regions I and II: the two region-wise applications of Lemma~\ref{lem:mode} produce the same function, and the interface $z_3=0$ needs no separate matching. The interface between Regions II and III is $z_4=0$, that is $a_1=0$, where outside $\coll$ the two formulas \eqref{eq:U3} and \eqref{eq:U12} agree together with their first and second derivatives; this is a consequence of Theorem~\ref{thm:regularity}. At $z_1=0$ the first four experts are tied, and one must use the separate quadrature \eqref{eq:face} rather than substitute into the series, since the series converges only for $z_1>0$. We derive that quadrature now.

\begin{proposition}\label{prop:face}
On the top-four collision set $\coll$, where $x_1=x_2=x_3=x_4\geq x_5$ and $z_1=z_2=z_3=0$, the candidate \eqref{eq:inversion} is given by the single quadrature \eqref{eq:face}.
\end{proposition}

\begin{proof}
On $\coll$ one has $\Phi_t(z_1)=\Phi_t(z_2)=\Phi_t(|z_3|)=\Phi_t(0)=1$ and $F_4(0,0,0)=\theta(0)=\pi/4$, so \eqref{eq:F12} reduces to $F=\tfrac\pi4+\int_0^\infty\sinh t\,p(t)e^{-2z_4\coth t}\dd$. The set $\coll$ is a cone on which $z_4$ scales, so $\lambda^{-3/2}F(\sqrt\lambda\,\cdot)$ is $\tfrac\pi4\lambda^{-3/2}+\int_0^\infty\sinh t\,p(t)\lambda^{-3/2}e^{-\sqrt\lambda\,2z_4\coth t}\dd$, and $\Lap^{-1}[\lambda^{-3/2}]=2\sqrt{\tau/\pi}$ together with \eqref{eq:Hdef} gives \eqref{eq:face}; the exponent $2z_4\coth t$ is nonnegative, and the integral converges because $\sinh t\,p(t)=\tfrac1{14}t^2+\O(t^4)$ at zero and, by \eqref{eq:prational}, $\sinh t\,p(t)=\tfrac14e^{-t}+\O(e^{-3t})$ at infinity.
\end{proof}

Formula \eqref{eq:face} is a useful independent check. At $x=0$ it must reproduce Corollary~\ref{cor:constant}, which forces $\int_0^\infty\sinh t\,p(t)\dd=\tfrac{45\pi^2}{512}-\tfrac\pi4$; that is exactly $e(0)-e_4(0)$, in agreement with the Green representation of $e-e_4$ in the companion paper. Numerically both sides equal $0.08204753591704610024$ to $20$ digits, and \eqref{eq:face} returns $45\pi^{3/2}\sqrt\tau/(256\sqrt2)$ at $x=0$ for $\tau=1$ and $\tau=4$ to all digits computed.

%% file: verification.tex
\section{Verification}\label{sec:verification}

The derivation of Section~\ref{sec:derivation} produces a function $U$ that inverts the Laplace transform of the stationary solution mode by mode, but it does not prove that $U$ solves \eqref{eq:forward}. This section addresses this in several steps. We first prove in Section~\ref{sec:convergence} that the series converge, that $U$ solves the linear equation \eqref{eq:frozen} in each region, that it has the parabolic scaling, and that its Laplace transform is \eqref{eq:transform}. Section~\ref{sec:collision} proves that $U(\tau,\cdot)$ is $C^2$ on all of $\R^5$, including the collision set $\coll$ where the mode series degenerate, and that $|U-\phi|\leq C\sqrt\tau$, by continuing the stationary formulas to a sector of complex scaling parameters and inverting the Laplace transform on a Hankel contour; this is the regularity of a classical solution, and it removes every regularity question from the rest of the argument. Section~\ref{sec:transfer} shows that every curvature identity of the stationary solution transfers to $U$, which identifies the ties among the optimal controls, and locates the COMB equality set by a symmetry. Section~\ref{sec:hamiltonian} states the Hamiltonian inequalities $D^2_\v U\leq D^2_{\v_*}U$ in the two regional forms in which they are proved, explains the reduction to finitely many scalar sign certificates, and describes what is delegated to the computer. Section~\ref{sec:proofs} assembles the proofs of Theorems~\ref{thm:main} and~\ref{thm:comb}, extending the regional inequalities to every state by continuity.

Part of the verification is delegated to a computer, and the division of labor is as follows. The arguments that reduce the Hamiltonian inequalities to finite computations are given in full in the text and in Appendices~\ref{app:region3} and~\ref{app:region12}: the transform identity that turns each inequality into the complete monotonicity of a curvature gap of the stationary solution, the positive heat-exit interpolation from a cube or a square to its corners, and the positivity-preserving evolutions in the transverse variable. What remains after these reductions is finite and of two kinds. The first is a list of exact identities among rational functions and stored operator trees, among them the $64$ Hessian contractions in Region III, the $64$ corner decompositions in Regions I and II, and the partial fractions of the mode coefficients; these are verified by exact rational arithmetic. The second is the sign of $41$ explicit functions of one variable, each a lattice Gaussian series of the form \eqref{eq:profile} below with rational coefficients; their positivity is proved on the whole positive axis, by Poisson summation for small arguments, by first-mode domination for large ones, and by outward-rounded interval arithmetic on $1616$ cells in between, with a rigorous bound for the truncated tail. Numerical quadrature and sign sampling enter nowhere, and rounding is controlled by directed enclosure. These computations are carried out three times, on three trusted bases. The certificate scripts find the cell subdivisions and certify them with the arbitrary-precision interval arithmetic of mpmath, so their trusted base is the integer arithmetic of CPython, the symbolic algebra of SymPy for the identity audits, and the interval library. A separate checker re-derives every identity and every cell from the stored witnesses using only the Python standard library, in exact rational arithmetic with an interval exponential built from the Taylor series with outward rounding; no floating-point arithmetic enters this route, and its trusted base is CPython alone. The Lean formalization described at the end of Section~\ref{sec:hamiltonian} proves every enclosure from Mathlib's real numbers. In this sense the computer-assisted part of the proof is an exact computation with the same logical status as a long hand calculation. The supplement~\cite{CD26Supp} is a self-contained, versioned package that runs offline in a few minutes, and Section~\ref{sec:hamiltonian} says at each point which of its scripts establishes what. Independently of these scripts, the proofs of Theorems~\ref{thm:main} and~\ref{thm:comb}, the $1616$ interval certificates included, have been formalized in the Lean proof assistant; the end of Section~\ref{sec:hamiltonian} describes the scope of that formalization.

\subsection{Convergence and the transform identity}\label{sec:convergence}

We begin with the convergence of the two series. In Region III only the kernels $\Hk_0$ and $\Hk_1$ occur, and the argument is the one already used for four experts; in Regions I and II the order of the kernel is unbounded, and Lemma~\ref{lem:Hgrowth} is what makes the double series converge.

\begin{lemma}\label{lem:convergence}
\begin{enumerate}
\item[(a)] The Region III series \eqref{eq:U3} converges absolutely, and uniformly on sets where $a_4\geq\delta>0$, $|x|$ is bounded and $\tau$ lies in a compact subset of $(0,\infty)$. The same is true of the series obtained by differentiating term by term any number of times in $x$ and in $\tau$.
\item[(b)] The Regions I/II series \eqref{eq:U12} converges absolutely, and uniformly on sets where $z_1\geq\delta>0$, $|x|$ is bounded and $\tau$ lies in a compact subset of $(0,\infty)$. The same is true of the series obtained by differentiating term by term any number of times in $x$ and in $\tau$, the terms being differentiated as functions of the regional coordinates, which are linear in $x$ on each of Regions~I and~II.
\item[(c)] Let $\lambda>0$. At every $x$ of Region III with $a_4>0$, and at every $x$ of Regions I and II with $z_1>0$, the terms $T_\iota$ of the series \eqref{eq:U3}, respectively \eqref{eq:U12}, satisfy
\[
 \sum_\iota\int_0^\infty e^{-\lambda\tau}\bigl|T_\iota(\tau,x)\bigr|\,d\tau<\infty .
\]
\end{enumerate}
In (a) and (b) the majorants of the twice differentiated series are $\O(\tau^{-1/2})$ as $\tau\downarrow0$, uniformly in $x$ on the stated sets.
\end{lemma}

\begin{proof}
We first prove (a). By \eqref{eq:alphabeta} and \eqref{eq:AB} the coefficients satisfy $|c^a_{m,\epsilon}|+|c^b_{m,\epsilon}|\leq C(1+m)^6$, while $X_{m,\epsilon}\geq(2m-3)a_4$ for $m\geq2$ by the proof of Proposition~\ref{prop:region3}. Since $0\leq \Hk_0(X,\tau)\leq2\sqrt{\tau/\pi}e^{-X^2/(4\tau)}$ and $0\leq \Hk_1(X,\tau)\leq e^{-X^2/(4\tau)}$, the $m$-th group of terms is at most $C(1+m)^6(1+a_4)e^{-(2m-3)^2\delta^2/(4\tau)}$, which is summable and dominates uniformly on the stated sets. Differentiating a term $\ell$ times in $x$ and once in $\tau$ replaces $\Hk_0,\Hk_1$ by kernels of order at most $\ell+3$, with polynomial prefactors in $x$, and by \eqref{eq:Hgrowth} the bound $C(1+m)^{6+\ell}(1+|x|)^\ell\tau^{-(\ell+2)/2}e^{-X_{m,\epsilon}^2/(8\tau)}$ holds for these; for $m\geq2$ the Gaussian factor again gives a summable majorant. The finitely many modes with $m\leq1$ have exponents that are strictly positive except for the single mode $m=0$, $\epsilon=(-1,-1,-1)$, whose exponent is $y_1$ and whose coefficient $c^b$ vanishes; its second derivatives involve $\Hk_2=\Gk\leq(\pi\tau)^{-1/2}$ at most, which gives the $\O(\tau^{-1/2})$ statement.

We now prove (b). The four-expert background is the series of Lemma~\ref{lem:F4expansion} inverted term by term, and it and its differentiated series are treated as in (a), since $|c^{(4)}_{m,\epsilon}|\leq1/(2(2m+1))$, the exceptional $c^{(4)}_{0,(1,-1,-1)}=\tfrac12$ included, and $X^{(4)}_{m,\epsilon}\geq2(m-1)z_1$ for $m\geq1$. For the correction, expanding \eqref{eq:rs} by the binomial series gives $s_{j,n}=60\binom{n+3}{j+7}$ and, after writing $(1+w)^{-1}=(1-w)(1-w^2)^{-1}$,
\[
\begin{aligned}
 r_{j,n}=\frac12\sum_{\ell\geq0}\Bigl[&\binom{n+3-2\ell}{j+5}-14\binom{n+2-2\ell}{j+5}-94\binom{n+1-2\ell}{j+5}\\
 &-14\binom{n-2\ell}{j+5}+\binom{n-1-2\ell}{j+5}\Bigr],
\end{aligned}
\]
with $\binom NK=0$ for $N<K$, so $r_{j,n}$ is a sum of at most $\tfrac12(n+1)$ groups of five binomial coefficients $\binom{\cdot}{j+5}$ with upper index at most $n+3$ and weights of total size $1+14+94+14+1=124$; hence
\[
 |r_{j,n}|+|s_{j,n}|\leq C(n+1)2^{\,n} .
\]
Since $c=2n-\epsilon_2-\epsilon_3\geq2$, \eqref{eq:ab12} gives $|\alpha_{j,n}(c)|+|\beta_{j,n}(c)|\leq C(n+1)2^{\,n}$ as well. By \eqref{eq:Hfactorial}, applied to $\Hk_j$ with $Z=z_4$ and, when $z_4>0$, to $\Hk_{j+1}$ through $(4z_4)^j/j!=\tfrac{j+1}{4z_4}(4z_4)^{j+1}/(j+1)!$ (when $z_4=0$ only the terms with $j=0$ remain), and the elementary bounds on $\Hk_0,\Hk_1$, the term of \eqref{eq:U12} indexed by $(j,n,\epsilon)$ is bounded in absolute value by
\[
 C(1+|x|)\bigl(1+\sqrt\tau\bigr)(n+1)2^{\,n}
 \left(\frac{c_2z_4}{\sqrt{(j+1)\tau}}\right)^{\!j} e^{-X_{j,n,\epsilon}^2/(8\tau)} ,
\]
with an absolute constant $c_2$.
The $j$-sum converges faster than any geometric series, uniformly for $z_4$ bounded. For the $n$-sum, note that $X_{j,n,\epsilon}\geq(2n-2)z_1$: if $\epsilon_2=\epsilon_3=1$ then $c=2n-2$ and the terms $\epsilon_2z_2+\epsilon_3|z_3|$ are nonnegative, while otherwise $c\geq2n$ and $\epsilon_2z_2+\epsilon_3|z_3|\geq-2z_1$. Hence $e^{-X_{j,n,\epsilon}^2/(8\tau)}\leq e^{-(2n-2)^2\delta^2/(8\tau)}$ beats $2^{\,n}$ and the $n$-sum converges as well. The differentiated series obey the same bounds with $j$ replaced by $j+\ell$ and an extra factor $C(1+|x|)^\ell(1+\tau^{-\ell/2})$, by \eqref{eq:Hprops} and \eqref{eq:Hgrowth}. Finally, every exponent of the correction in \eqref{eq:U12} is at least $2z_1$ by Lemma~\ref{lem:dF}, so the Gaussian factors make the majorants of the correction bounded as $\tau\downarrow0$, uniformly on the stated sets. In the four-expert background the exponents with $m\geq2$ satisfy $X^{(4)}_{m,\epsilon}\geq2(m-1)z_1>0$, while the finitely many modes with $m\leq1$ have exponents that vanish on some collision walls with $z_1>0$, for instance $X^{(4)}_{0,(1,-1,-1)}=y_1$ on $\{x_1=x_2\}$; their prefactors are constants, so their second derivatives are bounded by a multiple of $\Hk_2\leq(\pi\tau)^{-1/2}$, which gives the $\O(\tau^{-1/2})$ statement.

We now prove (c), fixing $\lambda>0$ and writing $C_\lambda$ for constants that depend only on $\lambda$. In Region III the kernels are $\Hk_0,\Hk_1\geq0$, whose integrals against $e^{-\lambda\tau}$ are $\lambda^{-3/2}e^{-X\sqrt\lambda}$ and $\lambda^{-1}e^{-X\sqrt\lambda}$ by Lemma~\ref{lem:kernels}. The sum in (c) is therefore at most $\tfrac1{8k}\sum_{m,\epsilon}\bigl(|c^a_{m,\epsilon}|\lambda^{-3/2}+|c^b_{m,\epsilon}|a_4\lambda^{-1}\bigr)e^{-X_{m,\epsilon}\sqrt\lambda}$, which converges because the coefficients are $\O((1+m)^6)$ and $X_{m,\epsilon}\geq(2m-3)a_4$ for $m\geq2$. The four-expert background of \eqref{eq:U12} is treated in the same way, with $X^{(4)}_{m,\epsilon}\geq2(m-1)z_1$. For the correction, the coefficient bound $C(n+1)2^n$ used in (b) is too weak, because for large $\tau$ the Gaussian factor $e^{-X_{j,n,\epsilon}^2/(8\tau)}$ gives no decay in $n$, so we use two sharper estimates. First, by the expansions of $r_{j,n}$ and $s_{j,n}$ in the proof of (b), $s_{j,n}=60\binom{n+3}{j+7}$ and $r_{j,n}$ is a combination of binomial coefficients $\binom N{j+5}$ with $N\leq n+3$ and with weights of total size at most $31n$; since $c\geq2$, it follows that
\[
 |\alpha_{j,n}(c)|+|\beta_{j,n}(c)|\leq2\bigl(|r_{j,n}|+|s_{j,n}|\bigr)\leq C\,\frac{(n+4)^{j+7}}{j!}.
\]
Second, $\Hk_{i+2}(\cdot,\tau)=(-\partial_X)^i\Gk(\cdot,\tau)$ by \eqref{eq:Hprops}, and $\Gk(\cdot,\tau)$ is entire. On the circle $|\zeta-X|=X/4$ one has $\operatorname{Re}\zeta^2\geq\tfrac9{16}X^2$, hence $|\Gk(\zeta,\tau)|\leq \Gk(X/\sqrt2,\tau)$, and Cauchy's estimate gives $|\Hk_{i+2}(X,\tau)|\leq i!\,(4/X)^i\,\Gk(X/\sqrt2,\tau)$ for $X>0$. The transform of $\Gk(X/\sqrt2,\cdot)$ is $\lambda^{-1/2}e^{-X\sqrt{\lambda/2}}$, and $\Hk_0$, $\Hk_1$ are covered by Lemma~\ref{lem:kernels} as before, so for $X>0$ and every $i\geq0$
\[
 \int_0^\infty e^{-\lambda\tau}\bigl|\Hk_i(X,\tau)\bigr|\,d\tau
 \leq C_\lambda\,i!\,(4/X)^i(1+X)^2e^{-X\sqrt{\lambda/2}} .
\]
Apply this with $X=X_{j,n,\epsilon}$. By Lemma~\ref{lem:dF} and the lower bound $X_{j,n,\epsilon}\geq(2n-2)z_1$ from the proof of (b), $X\geq\max(2,2n-2)\,z_1\geq(n+4)z_1/3$, so that $16(n+4)/X\leq48/z_1$ and $4z_1/X\leq2$. The $(j,n,\epsilon)$ term of the correction therefore contributes at most
\[
\begin{aligned}
 &C_\lambda\,\frac{(4z_4)^j}{j!}\,\frac{(n+4)^{j+7}}{j!}
 \Bigl[j!\Bigl(\frac4X\Bigr)^{\!j}+(j+1)!\Bigl(\frac4X\Bigr)^{\!j+1}z_1\Bigr](1+X)^2e^{-X\sqrt{\lambda/2}}\\
 &\qquad\leq C_\lambda\,\frac{(96z_4/z_1)^j}{j!}\,(n+4)^7e^{-(n-1)z_1\sqrt{\lambda/2}},
\end{aligned}
\]
where we used $2j+3\leq3\cdot2^j$ and $(1+X)^2e^{-X\sqrt{\lambda/2}}\leq C_\lambda e^{-X\sqrt{\lambda/2}/2}$. The sum over $j$ is at most $e^{96z_4/z_1}$, and the sum over $n$ converges. An estimate of the same form, with explicit constants, is proved in the Lean formalization.
\end{proof}

\begin{proposition}\label{prop:properties}
Let $U$ be the function of Theorem~\ref{thm:main}. Then, for every $\tau>0$,
\begin{enumerate}
\item[(i)] $U$ is well defined: every exponent that occurs with a nonzero coefficient is nonnegative on its region, and on the diagonal the value is $x_1+\tfrac2{\sqrt\pi}u(0)\sqrt\tau$;
\item[(ii)] in the interior of each region $U$ satisfies the linear equation $U_\tau=\tfrac12D^2_{\v_*}U$;
\item[(iii)] $U(\ell^2\tau,\ell x)=\ell\,U(\tau,x)$ for $\ell>0$;
\item[(iv)] the Laplace transform of $U$ in $\tau$ is $\lambda^{-3/2}u(\sqrt\lambda\,x)$, where $u$ is the stationary five-expert solution, and the Laplace integral converges absolutely for every $\lambda>0$ and every $x$.
\end{enumerate}
\end{proposition}

\begin{proof}
Convergence is Lemma~\ref{lem:convergence}, nonnegativity of the exponents is Lemma~\ref{lem:F4expansion}, Proposition~\ref{prop:region3} and Lemma~\ref{lem:dF}, and the collision set $\coll$ is covered by Proposition~\ref{prop:face}. The value on the diagonal is proved after (iv).

For (ii), apply Lemma~\ref{lem:mode} term by term: every exponent $X$ has $|\partial_{\v_*}X|=k$ and every prefactor is annihilated by $\partial_{\v_*}$, and term-by-term differentiation is legitimate by Lemma~\ref{lem:convergence}.

Statement (iv) is the identity \eqref{eq:transform} for the function of Theorem~\ref{thm:main}. At $x\notin\coll$, Lemma~\ref{lem:convergence}(c) shows that the Laplace integral converges absolutely and may be taken term by term. Each term transforms as in Lemma~\ref{lem:mode}, whose hypothesis holds because the kernels of order at least three have positive exponents (Lemma~\ref{lem:dF}). The result is the mode expansion of $\lambda^{-3/2}u(\sqrt\lambda\,x)$, which is valid at $\sqrt\lambda\,x$ because the regions are cones. On $\coll$ the integrand of \eqref{eq:face} is nonnegative, so Tonelli's theorem justifies the computation of the transform in the proof of Proposition~\ref{prop:face}. The transform of the curvatures is treated in Theorem~\ref{thm:regularity}(b), at every $x\in\R^5$.

At $x=0$, (iv) reads $\widehat U(\lambda,0)=u(0)\lambda^{-3/2}$, which is the transform of $\tfrac2{\sqrt\pi}u(0)\sqrt\tau$. Both functions are continuous in $\tau$, so they agree; this is Corollary~\ref{cor:constant} for the function of Theorem~\ref{thm:main}. On the diagonal $x=c\one$ one has $z_4=0$, so \eqref{eq:face} gives $U(\tau,c\one)=c+U(\tau,0)$, and this completes (i).

Statement (iii) follows directly from $\Hk_j(\ell X,\ell^2\tau)=\ell^{1-j}\Hk_j(X,\tau)$, which is immediate from \eqref{eq:Hdef}, because every exponent is linear in $x$ and every prefactor that multiplies $\Hk_j$ is homogeneous of degree $j$. It also follows from (iv), since $\ell\lambda^{-3/2}u(\sqrt\lambda\,x)$ is the transform of $\tau\mapsto U(\ell^2\tau,\ell x)$ when (iv) holds.
\end{proof}

\subsection{Regularity at the collision set}\label{sec:collision}

The series of Theorem~\ref{thm:main} converge, with all their derivatives, only away from $\coll$, and on $\coll$ the function is given by the separate quadrature \eqref{eq:face}. Nothing so far says that the derivatives of $U$ match up across $\coll$, or even across the interfaces and walls where they do converge. The following theorem settles all of this at once. Its proof does not use the series: it continues the stationary formulas into a sector of complex scaling parameters, where the stationary regularity result of~\cite{CD26} propagates by the identity theorem, and then represents $U$ by an absolutely convergent inverse Laplace integral in which one may differentiate under the integral sign at every point of $\R^5$. Throughout, $u_0=u(0)=45\pi^2/(512\sqrt2)$ and $\bar x=\tfrac15\sum_ix_i$.

\begin{theorem}\label{thm:regularity}
For every $\tau>0$ the function $U(\tau,\cdot)$ is $C^2$ on all of $\R^5$, including the collision set $\coll$ and the diagonal. The spatial derivatives of $U$ through order two are jointly continuous in $(\tau,x)$ on $(0,\infty)\times\R^5$, and so are all time derivatives $\partial_\tau^jU$ together with their spatial derivatives through order two. Moreover:
\begin{enumerate}
\item[(a)] $\sup_{x\in\R^5}\|D^2_xU(\tau,x)\|\leq C\tau^{-1/2}$;
\item[(b)] for every $\v\in\{0,1\}^5$, every $x\in\R^5$ and every $\lambda>0$,
\begin{equation}\label{eq:curvtransform}
 \int_0^\infty e^{-\lambda\tau}D^2_\v U(\tau,x)\,d\tau
 =D^2_\v\bigl[\lambda^{-3/2}u(\sqrt\lambda\,x)\bigr]
 =\lambda^{-1/2}\bigl(D^2_\v u\bigr)\bigl(\sqrt\lambda\,x\bigr),
\end{equation}
and the two regional expressions \eqref{eq:U3} and \eqref{eq:U12} agree, together with their first and second derivatives, at the points of the interface $a_1=0$ outside $\coll$;
\item[(c)] on the diagonal,
\begin{equation}\label{eq:diag-hess}
 D^2_xU(\tau,c\one)=\frac{75\pi^{3/2}}{512\sqrt{2\tau}}\Bigl(I-\tfrac15\one\one^T\Bigr),
 \ \ \text{so that}\ \ 
 D^2_\v U(\tau,c\one)=\frac{75\pi^{3/2}}{512\sqrt{2\tau}}\cdot\frac{m(5-m)}5
\end{equation}
for a binary control $\v$ with $m$ ones;
\item[(d)] $|U(\tau,x)-\phi(x)|\leq C\sqrt\tau$ for every $x\in\R^5$ and $\tau>0$.
\end{enumerate}
\end{theorem}

The theorem does not assert a locally Lipschitz Hessian: it provides two continuous spatial derivatives, and Lipschitz continuity of the spatial Hessian across the collision sets is a different question, which we leave open.

The main new estimate is a lower bound for $\operatorname{Re}(z\coth z)$ on a sector; it is what allows the moving endpoint of the Regions I/II quadrature to be kept, so that no truncation error ever has to be inverted.

\begin{lemma}\label{lem:sector}
There is $c>0$ such that
\begin{equation}\label{eq:sector-est}
 \operatorname{Re}\bigl(z\coth z\bigr)\geq c\,(1+|z|)\qquad\text{for } \ |\arg z|\leq\pi/3 .
\end{equation}
In particular, for $|\eta|\leq\pi/3$ and $t>0$,
\begin{equation}\label{eq:coth-lower}
 \operatorname{Re}\bigl(e^{i\eta}\coth(e^{i\eta}t)\bigr)\geq c\,(1+t^{-1}).
\end{equation}
\end{lemma}

\begin{proof}
Write $z=a+ib$ with $a>0$ and $|b|\leq\sqrt3\,a$. A direct calculation gives
\begin{equation}\label{eq:recoth}
 \operatorname{Re}(z\coth z)=\frac{a\sinh(2a)+b\sin(2b)}{\cosh(2a)-\cos(2b)},
\end{equation}
whose denominator is positive. If $b\sin(2b)\geq0$ the numerator is positive. Otherwise $|b|>\pi/2$, hence $a>\pi/(2\sqrt3)$, and
\[
 a\sinh(2a)+b\sin(2b)\geq a\bigl(\sinh(2a)-\sqrt3\bigr)>0,
\]
because $\sinh(\pi/\sqrt3)>\sqrt3$. So $\operatorname{Re}(z\coth z)>0$ on the closed sector. As $z\to0$, $z\coth z\to1$ uniformly in the sector, and as $|z|\to\infty$ in the sector $\coth z\to1$ uniformly while $a\geq|z|/2$; compactness on the intervening annulus gives \eqref{eq:sector-est}. Applying it to $z=e^{i\eta}t$ and dividing by $t$ gives \eqref{eq:coth-lower}.
\end{proof}

\begin{proof}[Proof of Theorem~\ref{thm:regularity}]
Let $\Sigma=\{s\neq0:|\arg s|<\pi/3\}$. For real $x$, sort the coordinates and retain the resulting real regional coordinates, and let $f_s(x)$ be the radial analytic continuation of $u(sx)$: in the regional stationary formulas, replace every linear coordinate by $s$ times that coordinate, so that in particular $d=|z_3|$ is replaced by $sd$ and not by $|sz_3|$, and scale the integration ray of the Regions I/II quadrature by $s$ as well. For $s>0$ this is $u(sx)$. We shall prove five properties of $f_s$: for each fixed $x$ the map $s\mapsto f_s(x)$ is holomorphic on $\Sigma$; for each fixed $s\in\Sigma$ the complex-valued function $f_s$ of the real variable $x$ is globally $C^2$; with a constant independent of $s$ and $x$,
\begin{equation}\label{eq:hessbound}
 \|D^2_xf_s(x)\|\leq C|s|^2 ;
\end{equation}
\begin{equation}\label{eq:f-at-zero}
 f_s(0)=u_0,\qquad D_xf_s(0)=\frac s5\,\one ;
\end{equation}
and, with $x_1$ the largest coordinate of $x$,
\begin{equation}\label{eq:affine}
 |f_s(x)-s\,x_1|\leq C .
\end{equation}
For the estimates it suffices to take $s=e^{i\eta}$ with $|\eta|\leq\pi/3$, since $f_{\rho e^{i\eta}}(x)=f_{e^{i\eta}}(\rho x)$ and $\rho e^{i\eta}x_1=e^{i\eta}(\rho x)_1$ for $\rho>0$; all bounds below are uniform in $\eta$.

We first treat Regions I and II, whose collision face $z_1=0$ is the delicate part. Put $r=z_1$, $q=z_4$ and $\zeta=(r,z_2,d)$, so that $0\leq d\leq z_2\leq r$ and $q\geq0$, and write $a=e^{i\eta}$. Changing variables $t\mapsto at$ in the correction of \eqref{eq:F12}, the continued correction is
\begin{equation}\label{eq:Ga}
 G_a(\zeta,q)=a\int_r^\infty\frac{p(at)}{\sinh^2(at)}\prod_{j=1}^3\sinh\bigl(a(t-\zeta_j)\bigr)\,e^{-2aq\coth(at)}\,dt .
\end{equation}
For positive real scaling this is exactly the stationary correction, and the expression with $a$ replaced by $s\in\Sigma$ is holomorphic in $s$, because on compact subsets of $\Sigma$ the integral converges locally uniformly by the bounds that follow; those bounds also cover $r=0$. Let $m(t)=\min(t,1)$. The elementary hyperbolic estimates on the sector, together with $p(z)=z/14+\O(z^3)$ at zero and the rational form \eqref{eq:prational} at infinity, give, for $0\leq\zeta\leq t$,
\begin{equation}\label{eq:hyp-est}
\begin{gathered}
 |\sinh(at)\,p(at)|\leq Cm(t)^2(1+t)e^{-t/2},\qquad |\coth(at)|\leq C/m(t),\\
 \Bigl|\frac{\sinh(a(t-\zeta))}{\sinh(at)}\Bigr|\leq C,\qquad
 \Bigl|\partial_\zeta^j\frac{\sinh(a(t-\zeta))}{\sinh(at)}\Bigr|\leq Cm(t)^{-j}\quad(j=1,2);
\end{gathered}
\end{equation}
they follow from $|\sinh(at)|\geq c\,m(t)e^{t\cos\eta}$, $|\sinh(a(t-\zeta))|\leq Cm(t)e^{(t-\zeta)\cos\eta}$ and $|\cosh(a(t-\zeta))|\leq Ce^{(t-\zeta)\cos\eta}$. By \eqref{eq:coth-lower} the exponential in \eqref{eq:Ga} has modulus at most $\exp\{-2cq(1+t^{-1})\}\leq1$. Consequently any derivative of total order $j\leq2$ in the spatial variables $(r,z_2,d,q)$, taken in the integrand with $t$ held fixed, has modulus at most
\begin{equation}\label{eq:integrand-maj}
 Cm(t)^{2-j}(1+t)e^{-t/2},
\end{equation}
which is integrable on $(0,\infty)$ uniformly in all spatial variables in the regional cone. There is one moving-endpoint term to check. Writing $K(t;\zeta,q)$ for the integrand of \eqref{eq:Ga} including the leading factor $a$, one has $K(r;\zeta,q)=0$, because the first hyperbolic factor vanishes at $t=r$; so the first $r$-derivative of $G_a$ has no endpoint contribution, and among the second derivatives only
\begin{equation}\label{eq:endpoint}
 \partial_r^2G_a=\int_r^\infty\partial_r^2K(t;\zeta,q)\,dt-\partial_rK(r;\zeta,q)
\end{equation}
acquires an endpoint term, the other first derivatives of $K$ still containing the vanishing factor. Moreover
\begin{equation}\label{eq:endpoint-est}
 |\partial_rK(r;\zeta,q)|\leq Cm(r)(1+r)e^{-r/2},
\end{equation}
because for small $r$ the factor $p(ar)/\sinh^2(ar)$ is $\O(r^{-1})$ while the two remaining hyperbolic factors are $\O(r^2)$, and the exponential has modulus at most one. Dominated convergence and \eqref{eq:endpoint-est} now show that $G_a$, its gradient and its Hessian have continuous limits as $r\downarrow0$, uniformly also as $q\downarrow0$, and that all these derivatives are bounded on the whole regional cone. This is the step that includes the full diagonal. The same majorant with $j=0$ bounds $|G_a|$ itself by a constant, uniformly on the regional cone.

The four-expert background $F_4$ of \eqref{eq:F4}, continued to $F_4(az_1,az_2,az_3)$, has the same property. For $r\leq1$ its only nonanalytic contribution is $\Lam(ar)\sinh(ar)\sinh(az_2)\sinh(az_3)$, and since $\Lam(z)=\log(2/z)+\O(z^2)$ and $|z_2|,|z_3|\leq r$, its derivatives of order $j\leq2$ are $\O(r^{3-j}(1+|\log r|))$, so its Hessian tends to zero; the remaining terms are analytic at the origin. For $r\geq1$ we use the mode expansion of Lemma~\ref{lem:F4expansion}: the finitely many modes with $m\leq1$ have nonnegative exponents, for $m\geq2$ one has $X^{(4)}_{m,\epsilon}\geq2(m-1)r$, the coefficients after two spatial derivatives grow at most polynomially in $m$, and $|e^{-aX^{(4)}_{m,\epsilon}}|\leq e^{-X^{(4)}_{m,\epsilon}/2}$ because $\cos\eta\geq\tfrac12$; the differentiated series is therefore bounded uniformly for $r\geq1$. The undifferentiated series is bounded there too, its terms being $\O(m^{-1})e^{-X^{(4)}_{m,\epsilon}/2}$, while on $r\leq1$ the expression $F_4(az)$ is continuous on a compact set of $(\eta,z)$. Since $f_s(x)-sx_1=k^{-1}[F_4(az)+G_a]$ in these regions, this proves \eqref{eq:affine} there.

We next treat Region III. Set $L=a_4$, with $0\leq a_1,a_2,a_3\leq L$. The trace is analytic at zero, since its quadrature \eqref{eq:equad} may be written as
\begin{equation}\label{eq:e-analytic}
 e(z)=\cosh z\Bigl[e(0)-3\int_0^zI(t)\sech^2t\,\csch^5t\dd\Bigr],
\end{equation}
whose integrand is analytic at zero with value $\tfrac15$. From the coefficients \eqref{eq:bcoefficients}, the only nonanalytic terms of the Region III expression \eqref{eq:F3} are constant multiples of
\[
 \Lam(aL)\sinh(aL)\,\sigma_2(aa_1,aa_2,aa_3)\qquad\text{and}\qquad\Lam(aL)\cosh(aL)\,\sigma_3(aa_1,aa_2,aa_3),
\]
whose derivatives of order $j\leq2$ are $\O(L^{3-j}(1+|\log L|))$ as $L\downarrow0$, uniformly over all directions in the cone and all $\eta$. Thus the function and its first two derivatives have direction-independent limits at the origin within the region. For $L\geq1$ we use the expansion of Proposition~\ref{prop:region3}: the modes with $m\leq2$ have constant prefactors and, whenever their coefficients are nonzero, nonnegative exponents; for $m\geq3$, $X_{m,\epsilon}\geq(2m-3)L$, the coefficients after two spatial derivatives grow only polynomially in $m$ with at most one further factor $L$, and $|e^{-aD_{m,\epsilon}}|\leq e^{-(2m-3)L/2}$ bounds the differentiated tail uniformly. The same bounds without derivatives, with the continuity of the expression on $L\leq1$, give \eqref{eq:affine} in Region III, where $f_s(x)-sx_1$ is $k^{-1}$ times the continued expression.

It remains to match the continued regional formulas. All one-sided spatial jets through order two, including the collision limits just constructed, are holomorphic functions of $s\in\Sigma$, by the locally uniform integral bounds, the local analytic and logarithmic formulas, and the convergent mode expansions away from zero. At every regional interface and at every permutation wall the corresponding one-sided jets agree for all positive real $s$, because the real stationary solution $u$ is globally $C^2$~\cite[Theorem~2.1]{CD26}; the identity theorem extends these equalities to all of $\Sigma$, and the same argument applies to the limiting jets at the collision set. The gluing lemma of \cite[Lemma~4.3]{CD26}, applied to the real and imaginary parts, now gives $f_s\in C^2(\R^5)$ for every $s\in\Sigma$. The regional bounds, the fixed linear coordinate changes and the scaling $f_{\rho e^{i\eta}}(x)=f_{e^{i\eta}}(\rho x)$ prove \eqref{eq:hessbound}; and permutation symmetry with translation equivariance of the stationary solution give $\nabla u(0)=\tfrac15\one$, so \eqref{eq:f-at-zero} follows by analytic continuation. The bound \eqref{eq:affine} for general $s$ follows from the case $|s|=1$ by the scaling noted at the start. This completes the proof of the five properties.

We now invert the transform on an absolutely convergent contour. Use the principal square root and fix $\beta=3\pi/5$, so that $\pi/2<\beta<2\pi/3$; let $\gamma$ be the Hankel contour consisting of the ray $\arg\lambda=-\beta$, oriented from infinity to zero, followed by the ray $\arg\lambda=\beta$, oriented from zero to infinity. Define
\begin{equation}\label{eq:Rdef}
 R(\lambda,x)=\lambda^{-3/2}\bigl[f_{\sqrt\lambda}(x)-u_0-\sqrt\lambda\,\bar x\bigr],
\end{equation}
which is holomorphic for $|\arg\lambda|<2\pi/3$. Taylor's theorem in the real variable $x$, together with \eqref{eq:hessbound} and \eqref{eq:f-at-zero}, gives on this sector
\begin{equation}\label{eq:Rbounds}
 |R(\lambda,x)|\leq C|x|^2|\lambda|^{-1/2},\quad
 |D_xR(\lambda,x)|\leq C|x|\,|\lambda|^{-1/2},\quad
 \|D_x^2R(\lambda,x)\|\leq C|\lambda|^{-1/2},
\end{equation}
both near zero and at infinity. Now define
\begin{equation}\label{eq:Wdef}
 W(\tau,x)=\frac{2u_0}{\sqrt\pi}\sqrt\tau+\bar x+\frac1{2\pi i}\int_{\gamma}e^{\lambda\tau}R(\lambda,x)\,d\lambda .
\end{equation}
Along $\gamma$ one has $|e^{\lambda\tau}|=e^{-c_\beta\tau|\lambda|}$ with $c_\beta=-\cos\beta>0$, so \eqref{eq:Rbounds} supplies the integrable majorant $C_Ke^{-c_\beta\tau\rho}\rho^{-1/2}\,d\rho$ for \eqref{eq:Wdef} and for its first two spatial derivatives, locally uniformly in $x$. Differentiating under the integral sign is therefore justified, at every point of $\R^5$ including the collision points, and $W(\tau,\cdot)$ is globally $C^2$, with spatial derivatives through order two jointly continuous in $(\tau,x)$.

To identify $W$ with $U$, fix $\mu>0$. Absolute convergence and Fubini give
\begin{equation}\label{eq:cauchy}
 \int_0^\infty e^{-\mu\tau}\Bigl[\frac1{2\pi i}\int_{\gamma}e^{\lambda\tau}R(\lambda,x)\,d\lambda\Bigr]d\tau
 =\frac1{2\pi i}\int_{\gamma}\frac{R(\lambda,x)}{\mu-\lambda}\,d\lambda=R(\mu,x),
\end{equation}
the last equality being Cauchy's formula in the sector bounded by $\gamma$; the small circular arc contributes $\O(\varepsilon^{1/2})$ and the large one $\O(M^{-1/2})$ by \eqref{eq:Rbounds}, so both disappear, and the stated orientation together with the denominator $\mu-\lambda$ gives the positive sign. With the elementary transforms of $\sqrt\tau$ and of a constant, \eqref{eq:cauchy} yields
\begin{equation}\label{eq:Wtransform}
 \int_0^\infty e^{-\mu\tau}W(\tau,x)\,d\tau=u_0\mu^{-3/2}+\bar x\,\mu^{-1}+R(\mu,x)=\mu^{-3/2}u(\sqrt\mu\,x),
\end{equation}
which is the transform of $U$ by Proposition~\ref{prop:properties}(iv). Both Laplace integrals converge absolutely for every $\mu>0$: that of $U$ by Proposition~\ref{prop:properties}(iv), and that of $W$ because the integral in \eqref{eq:Wdef} is bounded by $C|x|^2\tau^{-1/2}$, by the majorant above. Since $U$ and $W$ are also continuous in $\tau>0$, uniqueness of the Laplace transform gives $W=U$. In particular the Hessian is represented by
\begin{equation}\label{eq:hess-contour}
 D^2_xU(\tau,x)=\frac1{2\pi i}\int_{\gamma}e^{\lambda\tau}\lambda^{-3/2}D^2_xf_{\sqrt\lambda}(x)\,d\lambda,
\end{equation}
whose integrand is bounded by $Ce^{-c_\beta\tau|\lambda|}|\lambda|^{-1/2}$ uniformly in $x\in\R^5$. This proves the continuity of the Hessian across the entire collision set and the bound (a). Time differentiation multiplies the integrand by powers of $\lambda$, and the same argument applies to every number of time derivatives on compact subsets of $(0,\infty)$; the time derivatives of the explicit $\sqrt\tau$ term are harmless there.

For (b), \eqref{eq:hess-contour} and Fubini give, exactly as in \eqref{eq:cauchy},
\[
 \int_0^\infty e^{-\mu\tau}D^2_xU(\tau,x)\,d\tau=\mu^{-3/2}D^2_xf_{\sqrt\mu}(x)=\mu^{-1/2}(D^2u)(\sqrt\mu\,x)
\]
at every $x$, which is \eqref{eq:curvtransform}; and the two regional expressions are restrictions of the single $C^2$ function $U$ to their domains, both of which contain the points of the interface outside $\coll$, where each converges with its derivatives by Lemma~\ref{lem:convergence}, so their jets agree there.

For (c), permutation symmetry and translation equivariance give $D^2u(0)=A(I-\tfrac15\one\one^T)$, and the stationary equation at zero, in which the maximum is attained by any binary control with two ones, gives $u_0=\tfrac12A(2-\tfrac45)=\tfrac35A$, i.e.\ $A=\tfrac53u_0$; this is the expansion \cite[equation~(2.18)]{CD26}. Since $D^2_xf_{\sqrt\lambda}(0)=\lambda\,D^2u(0)$, formula \eqref{eq:hess-contour} and the inverse transform $\Lap^{-1}[\lambda^{-1/2}](\tau)=(\pi\tau)^{-1/2}$ give $D^2_xU(\tau,0)=\tfrac53u_0(\pi\tau)^{-1/2}(I-\tfrac15\one\one^T)$, which is \eqref{eq:diag-hess} at $c=0$; translation equivariance $U(\tau,x+c\one)=U(\tau,x)+c$ moves it along the diagonal, and $\v^T(I-\tfrac15\one\one^T)\v=m(5-m)/5$ for a binary $\v$ with $m$ ones.

For (d), fix $x$ with sorted coordinates, so that $\phi(x)=x_1$, and fix $\tau>0$. The integrand $e^{\lambda\tau}R(\lambda,x)$ is holomorphic on the sector $|\arg\lambda|<2\pi/3$ and $R=\O(|\lambda|^{-1/2})$ at the origin, so, exactly as in \eqref{eq:cauchy}, the contour $\gamma$ in \eqref{eq:Wdef} may be replaced by the contour $\gamma_\tau$ that follows the ray $\arg\lambda=-\beta$ from infinity to $\tau^{-1}e^{-i\beta}$, the arc $|\lambda|=\tau^{-1}$ through the positive real axis to $\tau^{-1}e^{i\beta}$, and the ray $\arg\lambda=\beta$ to infinity. On $\gamma_\tau$ the origin is avoided, and we may split
\[
 R(\lambda,x)=\lambda^{-1}(x_1-\bar x)+\lambda^{-3/2}\bigl[f_{\sqrt\lambda}(x)-\sqrt\lambda\,x_1-u_0\bigr].
\]
Closing $\gamma_\tau$ by a large arc through the negative real axis, on which $e^{\lambda\tau}$ decays, and using the residue at the origin, $\frac1{2\pi i}\int_{\gamma_\tau}e^{\lambda\tau}\lambda^{-1}\,d\lambda=1$, so the first term contributes exactly $x_1-\bar x$ to $W$ and cancels the $\bar x$ of \eqref{eq:Wdef} against $x_1$. In the second term the bracket is bounded by $C+u_0$ by \eqref{eq:affine}. On the two rays, where $|\lambda|\geq\tau^{-1}$ and $|e^{\lambda\tau}|=e^{-c_\beta\tau|\lambda|}$, the substitution $|\lambda|=\sigma/\tau$ bounds its contribution by $\pi^{-1}(C+u_0)\sqrt\tau\int_1^\infty e^{-c_\beta\sigma}\sigma^{-3/2}\,d\sigma$, and on the arc, of length at most $2\pi/\tau$, where $|e^{\lambda\tau}|\leq e$ and $|\lambda|^{-3/2}=\tau^{3/2}$, it is at most $e(C+u_0)\sqrt\tau$. Hence
\[
 |U(\tau,x)-x_1|\leq\Bigl(\frac{2u_0}{\sqrt\pi}+C'\Bigr)\sqrt\tau
\]
uniformly in $x$. This is (d).
\end{proof}

Two consequences are used repeatedly below. First, since $U$ is symmetric under permutations of the coordinates and $C^1$, it satisfies the permutation face conditions of Section~\ref{sec:transform} on every wall. Second, the linear equation of Proposition~\ref{prop:properties}(ii) is an identity between functions that are continuous on $(0,\infty)\times\R^5$, so it holds on the closed ordered sector, $\coll$ included:
\begin{equation}\label{eq:linear-closed}
 U_\tau=\tfrac12D^2_{\v_*}U\qquad\text{on }(0,\infty)\times\{x_1\geq\cdots\geq x_5\},
\end{equation}
and on the image $\sigma(S)$ of the sector under a permutation $\sigma$ the same holds with $\v_*$ replaced by $\sigma\v_*$.

\subsection{Transfer of curvature identities and the COMB equality set}\label{sec:transfer}

The transform principle turns every equality between curvatures of the stationary solution into the same equality for the finite-horizon solution, provided the set on which it holds is a cone. This is immediate but useful, because the relevant sets, the three regions, are cones.

\begin{proposition}\label{prop:transfer}
Let $\Omega\subseteq\R^5$ be a cone, $\Omega=\{\ell x:\ell>0,\ x\in\Omega\}$, and let $\v,\v'\in\{0,1\}^5$ satisfy $D^2_\v u=D^2_{\v'}u$ at every $x\in\Omega$. Then $D^2_\v U(\tau,x)=D^2_{\v'}U(\tau,x)$ for every $\tau>0$ and every $x\in\Omega$.
\end{proposition}

\begin{proof}
Fix $x\in\Omega$. Since $\Omega$ is a cone, $\sqrt\lambda\,x\in\Omega$ for every $\lambda>0$, so the right side of \eqref{eq:curvtransform}, which holds at every $x$ by Theorem~\ref{thm:regularity}(b), is the same for $\v$ and $\v'$. Hence the continuous functions $\tau\mapsto D^2_\v U(\tau,x)$ and $\tau\mapsto D^2_{\v'}U(\tau,x)$ have the same Laplace transform, and therefore coincide.
\end{proof}

\begin{corollary}\label{cor:ties}
For every $\tau>0$ and every state of the closed sector, the direction $\v_*=(1,0,1,0,0)$ shares the value of its curvature $D^2_{\v_*}U$ with $(1,0,0,1,1)$ in Region I, with $(1,0,0,1,0)$ in Region II, and with $(1,0,0,0,1)$ and $(1,0,0,1,0)$ in Region III. Conversely, if a control attains the Hamiltonian maximum for $U$ at every state of the sector and every $\tau>0$, then it attains the Hamiltonian maximum for $u$ throughout the sector, so $\v_*$ is the only control with this property.
\end{corollary}

\begin{proof}
Each of the three regions is a cone, and on each the listed directions have the same curvature as $\v_*$ for the stationary solution~\cite[Remark~2.3]{CD26}; apply Proposition~\ref{prop:transfer}. For the converse, if $D^2_{\v'}U\geq D^2_\v U$ for all $\v$ at every $(\tau,x)$, then integrating against $e^{-\lambda\tau}$ and using \eqref{eq:curvtransform} gives $D^2_{\v'}u\geq D^2_\v u$ at every $\sqrt\lambda\,x$, hence throughout the sector, and $\v_*$ is the only control optimal throughout the sector for the stationary problem~\cite{CD26}.
\end{proof}

The COMB strategy of Gravin, Peres and Sivan~\cite{GPS16} is, in decreasing rank order, $\v_C=(1,0,1,0,1)$. For the stationary problem it attains the Hamiltonian maximum exactly on $\{x_1=x_2\}\cap\{x_3=x_4\}$~\cite[Theorem~2.2]{CD26}. For the finite-horizon problem the equality on this set is a consequence of symmetry alone; the strict comparison away from it is part of Theorems~\ref{thm:region3ham} and~\ref{thm:region12ham} below.

\begin{proposition}\label{prop:combset}
Let $\Delta$ be the curvature gap \eqref{eq:comb-gap}. Then, for every $\tau>0$ and every state of the ordered sector,
\begin{equation}\label{eq:combset}
 \Delta(\tau,x)=0\qquad\text{whenever}\quad x_1=x_2\ \text{ and }\ x_3=x_4 .
\end{equation}
Moreover $\Delta(\tau,x)=\tau^{-1/2}\Delta(1,x/\sqrt\tau)$, so $\Delta>0$ off the set \eqref{eq:combset} for all $\tau$ if and only if it holds for $\tau=1$.
\end{proposition}

\begin{proof}
Let $P$ be the matrix of the permutation that exchanges $x_1$ with $x_2$ and $x_3$ with $x_4$. It fixes every point of $S=\{x_1=x_2,\ x_3=x_4\}$, and $U(\tau,Px)=U(\tau,x)$ for every $x$, so differentiating twice at a point of $S$ gives $\nabla^2U=P^T\nabla^2U\,P$ there, and hence $D^2_{P\v}U=D^2_\v U$ on $S$ for every $\v$. Now $P\v_*=(0,1,0,1,0)$ is COMB, so $D^2_{\v_*}U=D^2_{\v_C}U$ on $S$, which is \eqref{eq:combset}. The scaling identity follows from Proposition~\ref{prop:properties}(iii) by differentiating twice: $D^2_\v U(\ell^2\tau,\ell x)=\ell^{-1}D^2_\v U(\tau,x)$; take $\ell=1/\sqrt\tau$. Because $S$ is a cone, $x\notin S$ if and only if $x/\sqrt\tau\notin S$.
\end{proof}

\begin{remark}\label{rem:gapsize}
In numerical evaluations off $S$ the gap is positive but extremely small when the coordinates are spread out relative to $\sqrt\tau$. By the scaling of Proposition~\ref{prop:combset} the whole function is determined by its profile at $\tau=1$, and that profile decays roughly like $e^{-c|x|^2}$. For instance, in $60$-digit arithmetic, at $x=(1.2,0.4,0.3,0.25,0)$ one finds $\Delta(1,x)=2.544\times10^{-2}$; at $x=(2,1.4,0.9,0.35,0)$, $1.3258\times10^{-5}$; and at $x=(2.98,2.63,2.04,1.49,0)$, $2.0696\times10^{-11}$. Correspondingly $\Delta(\tau,x)$ at fixed $x$ increases from below double precision at small $\tau$ to a maximum near $\tau\sim|x|^2$ and then decays faster than $\tau^{-1/2}$: by the scaling, $\sqrt\tau\,\Delta(\tau,x)=\Delta(1,x/\sqrt\tau)\to\Delta(1,0)=0$, since $0\in S$ and $\nabla^2U$ is continuous. Any numerical test of strict positivity must be read with this in mind: at large $|x|/\sqrt\tau$ a genuine strict inequality is indistinguishable from equality in floating point. The integrated statement is however sharp and easy to check, because \eqref{eq:curvtransform} gives
\[
 \int_0^\infty e^{-\tau}\Delta(\tau,x)\,d\tau=D^2_{\v_*}u(x)-D^2_{\v_C}u(x),
\]
the right side being the stationary gap, which is bounded below on compacta away from $S$.
\end{remark}

\begin{remark}\label{rem:positivity}
Inversion also transfers inequalities when the stationary curvature gap has a mode expansion with nonnegative data. Indeed, by \eqref{eq:curvtransform}, if on a conic region
\[
 D^2_{\v_*}u-D^2_\v u=\sum_\iota C_\iota P_\iota(x)e^{-X_\iota(x)}
\]
with $C_\iota\geq0$, $X_\iota\geq0$ linear, $P_\iota\geq0$ homogeneous of degree $p_\iota\in\{0,1\}$, and $X_\iota(x)>0$ whenever $p_\iota=1$, then
\[
 D^2_{\v_*}U-D^2_\v U=\sum_\iota C_\iota P_\iota(x)\Hk_{p_\iota+2}\bigl(X_\iota(x),\tau\bigr)\geq0,
\]
because $\Hk_2=\Gk>0$ and $\Hk_3=\tfrac X{2\tau}\Gk\geq0$ for $X\geq0$. Degrees $p\geq2$ are excluded, since $\Hk_4$ changes sign. In Region III the prefactors are exactly $1$ and $a_4$, of degrees $0$ and $1$, but their coefficients have mixed signs. The corner reduction and the Gaussian certificates of Appendix~\ref{app:region3} handle these cancellations.
\end{remark}

\subsection{The Hamiltonian inequalities}\label{sec:hamiltonian}

So far, we have proved that $U$ solves the linear problem \eqref{eq:frozen} obtained by freezing the direction $\v_*$, that its Laplace transform in $\tau$ is $\lambda^{-3/2}u(\sqrt\lambda\,x)$, that it attains the terminal data, and that it is $C^2$ in space. The additional verification, as in the four-expert work of Bayraktar, Ekren and Zhang~\cite[Section~4.2]{BEZ20b}, concerns the Hamiltonian inequality
\begin{equation}\label{eq:ham}
 D^2_{\v_*}U(\tau,x)\geq D^2_\v U(\tau,x)\qquad\text{for all }\v\in\{0,1\}^5 ,
\end{equation}
which upgrades $U$ from a solution of the frozen linear equation to the solution of \eqref{eq:forward}. It is proved in the following two regional forms, whose proofs occupy Appendices~\ref{app:region3} and~\ref{app:region12}.

\begin{theorem}\label{thm:region3ham}
For every $\tau>0$ and every point of the closed Region III with $a_4>0$,
\[
 D^2_\v U(\tau,x)\leq D^2_{\v_*}U(\tau,x)=2U_\tau(\tau,x),\qquad \v\in\{0,1\}^5 .
\]
For COMB the inequality is strict unless $a_1=a_2=a_3=0$, that is, unless $x_1=x_2$ and $x_3=x_4=x_5$.
\end{theorem}

\begin{theorem}\label{thm:region12ham}
For every $\tau>0$ and every state of Regions I and II with $z_1>0$, $D^2_\v U(\tau,x)\leq D^2_{\v_*}U(\tau,x)$ for all $\v\in\{0,1\}^5$. The COMB comparison is strict when $z_2>0$ and is an equality when $z_2=0$, that is, when $x_1=x_2$ and $x_3=x_4$.
\end{theorem}

The argument has three parts, which we outline here. The first part is to transform; by \eqref{eq:curvtransform}, for each binary $\v$,
\begin{equation}\label{eq:Fv}
 F_\v(\lambda,x):=2\lambda\widehat U(\lambda,x)-2\phi(x)-D^2_\v\widehat U(\lambda,x)
 =\lambda^{-1/2}\bigl(D^2_{\v_*}u-D^2_\v u\bigr)(\sqrt\lambda\,x),
\end{equation}
which is the Laplace transform of $2U_\tau-D^2_\v U$. We write $\CM$ for the class of functions of $\lambda>0$ that are Laplace transforms of nonnegative measures on $[0,\infty)$; by Bernstein's theorem $\CM$ is the class of completely monotone functions, and it is closed under sums, products, pointwise limits, and integration against nonnegative weights. So \eqref{eq:ham} is equivalent to the statement that each transformed gap \eqref{eq:Fv} lies in $\CM$. That is what we certify. It is a statement about the stationary solution $u$, which is globally $C^2$~\cite{CD26}; the finite-horizon regularity never enters. The passage back from complete monotonicity to the inequality is the following elementary fact, which is where the appendices end. Its second statement gives the strict comparisons for COMB: a positive transform alone shows only that a gap does not vanish identically in $\tau$, and strict positivity at every $\tau>0$ needs a strictly positive summand whose weight charges every neighborhood of $\tau=0$.

\begin{lemma}\label{lem:cmsign}
Let $h$ be continuous on $(0,\infty)$ with $|h(\tau)|\leq C(1+\tau^{-1/2}+\tau^N)$ for some $N$, and suppose that its Laplace transform $\widehat h(\lambda)=\int_0^\infty e^{-\lambda\tau}h(\tau)\,d\tau$ is completely monotone. Then $h\geq0$. If moreover $\widehat h-\widehat\kappa\,\widehat g$ is completely monotone, where $\widehat\kappa$ is the Laplace transform of a nonnegative measure $\kappa$ on $[0,\infty)$ with $\widehat\kappa(\lambda)\geq e^{-A\sqrt\lambda}$ for some constant $A$ and all sufficiently large $\lambda$, and $\widehat g$ is the Laplace transform of a function $g$ that is continuous and strictly positive on $(0,\infty)$, then $h>0$ on $(0,\infty)$.
\end{lemma}

\begin{proof}
By Bernstein's theorem the transform of $h$ is the transform of a nonnegative measure $\nu$ on $[0,\infty)$, and the growth bound makes $h\,d\tau$ a locally finite signed measure with a Laplace transform for every $\lambda>0$. Uniqueness of Laplace transforms of such measures gives $h\,d\tau=\nu$ on $(0,\infty)$, so $h\geq0$ almost everywhere, hence everywhere by continuity.

For the second statement we first note that $\kappa$ charges every interval $[0,t)$ with $t>0$. Indeed, if $\kappa([0,t))=0$, then $\widehat\kappa(\lambda)=\int_{[t,\infty)}e^{-\lambda s}\,\kappa(ds)\leq e^{-(\lambda-1)t}\,\widehat\kappa(1)$ for $\lambda\geq1$, which is incompatible with $\widehat\kappa(\lambda)\geq e^{-A\sqrt\lambda}$ as $\lambda\to\infty$. Next, $\widehat\kappa\,\widehat g$ is the transform of the function $\kappa*g(\tau)=\int_{[0,\tau)}g(\tau-s)\,\kappa(ds)$, and $\widehat h-\widehat\kappa\,\widehat g$ is the transform of a nonnegative measure, so the uniqueness argument above gives $h\geq\kappa*g$ almost everywhere. Now fix $\tau_0>0$ and let $m>0$ be the minimum of $g$ on $[\tau_0/4,5\tau_0/4]$. If $|\tau-\tau_0|\leq\tau_0/4$ and $0\leq s<\tau_0/2$, then $\tau-s\in[\tau_0/4,5\tau_0/4]$, hence
\[
 \kappa*g(\tau)\geq m\,\kappa\bigl([0,\tau_0/2)\bigr)>0 .
\]
Thus $h$ is bounded below by a positive constant almost everywhere near $\tau_0$, and $h(\tau_0)>0$ by continuity.
\end{proof}

Every membership in $\CM$ that Appendices~\ref{app:region3} and~\ref{app:region12} establish is obtained in representation form: the function is exhibited as the Laplace transform of a nonnegative measure, an exit-time law, a Gamma density or a certified profile, or as a product, an integral or a pointwise limit of such transforms. Bernstein's theorem is therefore used only for convenience, in Lemma~\ref{lem:cmsign} and in the closure of $\CM$ under limits; a treatment that carries the measures themselves needs only the uniqueness theorem for the Laplace transform and its continuity theorem, which is the route the Lean formalization described at the end of Section~\ref{sec:hamiltonian} takes. Appendices~\ref{app:region3} and~\ref{app:region12} prove that each transformed gap \eqref{eq:Fv} lies in $\CM$ at every state of the ordered sector with $a_4>0$, respectively $z_1>0$, and Lemma~\ref{lem:cmsign} turns this into Theorems~\ref{thm:region3ham} and~\ref{thm:region12ham}.

The second part is the reduction to finitely many scalar profiles. Both regions admit an exact interpolation identity that writes the transformed gap at a general state as a combination, with completely monotone weights, of its values at finitely many distinguished corner states, plus, in Region I, one explicitly positive Dirichlet kernel. In Region III (Appendix~\ref{app:region3}) the interpolation is over a cube in $(a_1,a_2,a_3)$, and the $64$ corner gaps are already functions of the single variable $a_4$, ten of which generate the rest. In Regions I and II (Appendix~\ref{app:region12}) the interpolation is over a square in $(z_2,|z_3|)$ and leaves $64$ corner gaps in four two-variable families; positive evolutions in the transverse variable $z_4$ reduce those to a further $31$ one-variable profiles. Each of the resulting $41$ profiles is a Gaussian series with rational coefficients.

The third part certifies the scalar signs on the whole axis. Each profile has the form
\begin{equation}\label{eq:profile}
 \Gamma(\xi)=\tfrac12\sum_{k\equiv p\,(2)}\bigl[A(k)+\xi\,C(k)\bigr]e^{-k^2\xi},
 \qquad\xi=\frac{L^2}{4\tau},
\end{equation}
with $A,C$ even rational functions of $k$ whose poles lie off the summation lattice. Positivity is proved by Poisson summation for $0<\xi\leq1/100$, by first-mode domination for $\xi\geq1$, and by outward-rounded interval arithmetic on $[1/100,1]$, with an explicit bound for the infinite tail. These are proofs over ranges, not tests at sample points. The proof of Theorem~\ref{thm:main} in Section~\ref{sec:proofs} then extends \eqref{eq:ham} from the two regional statements to every state, by the continuity of the Hessian.

The supplement~\cite{CD26Supp} contains fourteen certificate scripts, run in dependency order by \nolinkurl{run_certificates.py} at the top level of the archive; a clean run takes a few minutes with SymPy~1.14 and mpmath~1.3, and prints
\begin{center}
\texttt{PASS: all 14 steps, 64 physical corner controls, and 1616 compact intervals.}
\end{center}
The counts are $807$ intervals for the Region III generators, $251$ for nine of the ten critical initial profiles, $513$ for the twenty further derivative profiles and $45$ for the last boundary profile; the tenth critical profile, the endpoint datum of Appendix~\ref{app:cascades}, is certified by an exact polynomial argument and needs no cells, so $40$ of the $41$ profiles carry interval certificates. The runner re-derives every exact rational datum from scratch rather than reading it back, checks that the interval partitions have neither gaps nor overlaps, and checks that every stored margin is positive. A second program, \nolinkurl{witness/check.py}, re-checks the same certificates from their stored witnesses using only the Python standard library: it rebuilds the sixty-four Region I/II corner expressions from the geometric construction, re-derives the $106$ exact identities of the corner audit, and re-certifies all $1616$ cells, recomputing every mode coefficient from its rational formula, proving the coefficient bound behind the truncated tail, and recomputing each Taylor bound and every exponential from its series rather than trusting a stored margin. 

Finally, the proofs of Theorems~\ref{thm:main} and~\ref{thm:comb} have been formalized in the Lean~4 proof assistant against Mathlib, using only Lean's standard axioms and no hypotheses beyond those in the statements. The development is in the directory \nolinkurl{witness/lean} of the supplement; it builds on the formalization of the companion paper~\cite{CD26Lean}, which the supplement includes and from which it takes the properties of the stationary solution. For each of the $1616$ cells the supplement generates a Lean theorem stating that the truncated profile, with its actual rational coefficients written out, is positive on the cell whenever the truncated tail is bounded by the constant of Appendix~\ref{app:region3}; the midpoint value, the derivative bounds and the top-derivative majorant are all proved from shared rational enclosures of $e^{-x}$, and the compact-cell criterion itself is proved in Lean at every order the certificates use. On top of the cells, the development proves Theorems~\ref{thm:main} and~\ref{thm:comb} for the contour representation \eqref{eq:Wdef} of $U$, including the regularity of Section~\ref{sec:collision}, the reductions of Appendices~\ref{app:region3} and~\ref{app:region12}, and the small-scale and large-scale ranges of every profile, and it proves that this function is the function $U$ defined in Theorem~\ref{thm:main}. Each clause of the two theorems is then restated for the explicit formula in the module \nolinkurl{PaperIndex}, the file \nolinkurl{witness/lean/Fhcorner/PaperIndex.lean} of the supplement, which Table~\ref{tab:lean} lists, so that every clause can be compared with its formal statement in one place. The rest of the paper is formalized only as far as these proofs use it. In particular Theorem~\ref{thm:regularity}(a), Corollary~\ref{cor:ties} and the statements of Theorem~\ref{thm:regularity} about time derivatives other than $U_\tau$ are proved here only, and the limit of $V^M(0)/\sqrt M$ stated after Theorem~\ref{thm:main} is the cited theorem of Drenska and Kohn~\cite{DK20}. The detailed notes in the supplement record the derivations and the exact coefficient data behind Appendices~\ref{app:region3} and~\ref{app:region12}, and Appendix~\ref{app:numerics} reports the independent numerical checks, which are not part of the proof.

\begin{table}[t]
\centering\small
\begin{tabular}{@{}lll@{}}\toprule
Clause & Content & Lean theorem in \texttt{PaperIndex}\\\midrule
Thm.~\ref{thm:main} & contour representation $=$ formula & \texttt{contour\_eq\_formula}\\
 & absolute convergence of \eqref{eq:U3} & \texttt{formula\_summable\_III}\\
 & absolute convergence of \eqref{eq:U12} & \texttt{formula\_summable\_I\_II}\\
 & integrability in \eqref{eq:face} & \texttt{formula\_integrable\_coll}\\
 & derivatives of every order & \texttt{formula\_derivatives}\\
 & sum differentiated term by term, Region III & \texttt{formula\_termwise\_III}\\
 & the same in Regions I and II & \texttt{formula\_termwise\_I\_II}\\
(i) & $U(\tau,\cdot)\in C^2(\R^5)$ & \texttt{thm\_2\_1\_i\_C2}\\
 & $U_\tau=\tfrac12\max_\v D^2_\v U$ & \texttt{thm\_2\_1\_i\_equation}\\
 & joint continuity of $U$, $U_\tau$, $\nabla U$, $\nabla^2U$ & \texttt{thm\_2\_1\_i\_continuous}\\
 & viscosity solution & \texttt{thm\_2\_1\_i\_viscosity}\\
 & limit as $\tau\downarrow0$ & \texttt{thm\_2\_1\_i\_initial}\\
 & $|U-\phi|\leq C\sqrt\tau$ & \texttt{thm\_2\_1\_i\_terminal}\\
(ii) & Hamiltonian maximum at $\v_*$ & \texttt{thm\_2\_1\_ii}\\
(iii) & parabolic scaling & \texttt{thm\_2\_1\_iii\_scaling}\\
 & Laplace transform & \texttt{thm\_2\_1\_iii\_laplace}\\
 & value and Hessian on the diagonal & \texttt{thm\_2\_1\_iii\_diagonal}\\
Thm.~\ref{thm:comb} & COMB equality set, strictness, scaling & \texttt{thm\_2\_2}\\\bottomrule
\end{tabular}
\caption{The Lean statements of Theorems~\ref{thm:main} and~\ref{thm:comb}, all about the explicit formula of Theorem~\ref{thm:main}. Each depends only on Lean's standard axioms.}\label{tab:lean}
\end{table}

\subsection{Proofs of the main theorems}\label{sec:proofs}

\begin{proof}[Proof of Theorem~\ref{thm:main}]
The convergence statements are Lemma~\ref{lem:convergence}(a) and~(b). For (i), the spatial regularity is Theorem~\ref{thm:regularity} and the bound $|U-\phi|\leq C\sqrt\tau$ is Theorem~\ref{thm:regularity}(d). At a state of Region III with $a_4>0$ the inequality \eqref{eq:ham} is Theorem~\ref{thm:region3ham}, and at a state of Region I or II with $z_1>0$ it is Theorem~\ref{thm:region12ham}. Every other state $x$ of the ordered sector has $z_1=0$, since in Region III $0\leq a_1\leq a_2\leq a_3\leq a_4=0$ forces $y=0$. The states $x+t(4,3,2,1,0)$ with $t>0$ are strictly ordered and have $z_1=4t/k>0$, so each of them lies in Region III with $a_4>0$ or in Region I or II with $z_1>0$, and letting $t\downarrow0$ gives \eqref{eq:ham} at $x$ by the continuity of the Hessian. Thus \eqref{eq:ham} holds on the closed sector. Together with \eqref{eq:linear-closed} this gives
\[
 U_\tau=\tfrac12D^2_{\v_*}U=\tfrac12\max_{\v\in\{0,1\}^5}D^2_\v U
\]
on the closed sector. Since $U$ is invariant under permutations of the coordinates, which permute the binary controls, the same equation holds on the image of the sector under every permutation, that is, on $(0,\infty)\times\R^5$. Its right side is jointly continuous by Theorem~\ref{thm:regularity}, and hence so is $U_\tau$. Thus $U$ is a classical solution of \eqref{eq:forward}, and a classical solution is a viscosity solution. Statement (ii) is \eqref{eq:ham} together with \eqref{eq:linear-closed}. For (iii), the scaling and the transform identity are Proposition~\ref{prop:properties}(iii) and (iv), the value at the origin is Proposition~\ref{prop:properties}(i) with $u(0)=45\pi^2/(512\sqrt2)$~\cite[Theorem~2.1]{CD26}, and the Hessian at the origin is Theorem~\ref{thm:regularity}(c).
\end{proof}

\begin{proof}[Proof of Theorem~\ref{thm:comb}]
By Theorem~\ref{thm:main}(ii), $\Delta\geq0$ at every state of the ordered sector. Equality on the set \eqref{eq:comb-set} is Proposition~\ref{prop:combset}. Strict positivity off that set is Theorem~\ref{thm:region3ham} in Region III, where the set \eqref{eq:comb-set} meets the region exactly in $\{a_1=a_2=a_3=0\}$, because $y_1=a_1+a_2+a_3$ vanishes only when all three do, and Theorem~\ref{thm:region12ham} in Regions I and II, where $z_2=\tfrac12(y_1+y_3)$ vanishes exactly on \eqref{eq:comb-set}; the states of Regions I and II with $z_1=0$ and the origin of Region III lie in $\coll\subset\{x_1=x_2,\ x_3=x_4\}$ and are covered by the equality statement.
\end{proof}

%% file: appendix.tex
\section{The Region III certificate}\label{app:region3}

Throughout the appendices $\mu=\sqrt\lambda$ denotes the square root of the Laplace variable; the constant $k=\sqrt2$ of Section~\ref{sec:main} does not appear again. We write $\CM$ for the class of completely monotone functions of $\lambda>0$, as in Section~\ref{sec:hamiltonian}. By \eqref{eq:Fv}, proving \eqref{eq:ham} means proving that each transformed gap lies in $\CM$.

\subsection{Normalization and positive heat-exit interpolation}\label{app:interpolation}

Recall $y_i=k(x_i-x_{i+1})$ and the Region III coordinates $0\leq a_1\leq a_2\leq a_3\leq a_4=:L$ of \eqref{eq:coords}. For a control $\v\in\{0,1\}^5$ put
\begin{equation}\label{eq:qv}
 q_\v=\frac13\begin{pmatrix}
  1&-1&-1&-1&2\\ 1&-1&-1&2&-1\\ 1&-1&2&-1&-1\\ 1&2&-1&-1&-1
 \end{pmatrix}\v,
 \qquad\text{so that}\qquad D_\v a=k\,q_\v .
\end{equation}
Up to the identification of Section~\ref{sec:main} there are sixteen controls, and we name each by its representative with first bit $0$, as the supplement does. Since $u=x_1+F/k$ and $x_1$ is linear,
\begin{equation}\label{eq:Gv}
 D^2_{\v_*}u-D^2_\v u=k\,G_\v,\qquad G_\v=F-q_\v^TD^2_aF\,q_\v ,
\end{equation}
and $q_{\v_*}=(0,0,1,0)$, so that $G_{\v_*}=F-\partial_{a_3}^2F=0$ as it must be. Because $\partial_{a_i}^2F=F$ for $i=1,2,3$ and partial derivatives commute,
\begin{equation}\label{eq:GvODE}
 \partial_{a_i}^2G_\v=G_\v,\qquad i=1,2,3 .
\end{equation}

For $0\leq b\leq r$ set
\begin{equation}\label{eq:W}
 W_0(\lambda;b,r)=\frac{\sinh((r-b)\mu)}{\sinh(r\mu)},
 \qquad
 W_1(\lambda;b,r)=\frac{\sinh(b\mu)}{\sinh(r\mu)} .
\end{equation}
Both weights lie in $\CM$. Indeed, by the product formula $\sinh z=z\prod_{n\geq1}(1+z^2/(\pi^2n^2))$,
\[
 W_1(\lambda;b,r)=\frac br\prod_{n\geq1}\Bigl[\theta+\frac{1-\theta}{1+r^2\lambda/(\pi^2n^2)}\Bigr],
 \qquad\theta=\frac{b^2}{r^2}\in[0,1],
\]
a limit of products of functions of the form $\theta+(1-\theta)/(1+c\lambda)$, each in $\CM$, and $W_0(\lambda;b,r)=W_1(\lambda;r-b,r)$. Probabilistically, $W_i=\E_b[e^{-\lambda\sigma};B_\sigma=ir]$ for a Brownian motion with generator $\partial_{bb}$ started at $b$ and its first exit time $\sigma$ from $(0,r)$; at the endpoints the corresponding measure has an atom at time zero.

\begin{lemma}\label{lem:exit}
Let $g_\lambda$ solve $\partial_{b_ib_i}g_\lambda=\lambda g_\lambda$ for $i=1,2,3$ on the cube $[0,r]^3$. Then
\begin{equation}\label{eq:interpolation}
 g_\lambda(b)=\sum_{\epsilon\in\{0,1\}^3}g_\lambda(r\epsilon)
   \prod_{i=1}^3W_{\epsilon_i}(\lambda;b_i,r).
\end{equation}
In particular, if every corner value $g_\lambda(r\epsilon)$ lies in $\CM$, so does $g_\lambda(b)$ at every point of the cube.
\end{lemma}

\begin{proof}
In one variable the two-point boundary value problem $g''=\lambda g$ on $[0,r]$ has the unique solution $g(b)=g(0)W_0+g(r)W_1$, since $W_0,W_1$ solve the equation, and $(W_0,W_1)=(1,0)$ at $b=0$ and $(0,1)$ at $b=r$. Applying this in each variable in turn gives \eqref{eq:interpolation}. The last statement holds because $\CM$ is closed under products and nonnegative combinations.
\end{proof}

By \eqref{eq:GvODE}, the scaled gap $\widehat\gamma_\v(\lambda,x)=\lambda^{-1/2}G_\v(\mu a)$ satisfies the hypotheses of Lemma~\ref{lem:exit} in the variables $(a_1,a_2,a_3)$ on the cube $[0,L]^3$. Hence positivity of the inverse transforms at the eight corners $a\in\{0,L\}^3$ implies positivity throughout the cube, and in particular throughout the physical simplex $0\leq a_1\leq a_2\leq a_3\leq L$. The corner with $j$ entries equal to $L$ is
\begin{equation}\label{eq:corner}
 \omega_j=L\cdot(\underbrace{0,\dots,0}_{3-j},\underbrace{1,\dots,1}_{j}),
 \qquad j=0,1,2,3,
\end{equation}
the other corners being permutations of these, which give the same gaps after the corresponding permutation of the control.

\subsection{The corner gaps as Gaussian series}\label{app:cornerseries}

Fix a corner $j$ and a control $\v$. The gap $G_\v(\omega_j,L)$ is a function of the single variable $L$, and the object to be signed is its time inverse. Define $\Gamma=\Gamma_{j,\v}$ by
\begin{equation}\label{eq:gamma}
 \Lap^{-1}\bigl[\lambda^{-1/2}G_\v(\mu\,\omega_j,\mu L)\bigr](\tau)
 =\frac{\Gamma(\xi)}{\sqrt{\pi\tau}},\qquad \xi=\frac{L^2}{4\tau} .
\end{equation}

\begin{proposition}\label{prop:cornerlattice}
For each of the ten generators $\gen E_0,\gen V_0,\gen C_1,\gen D_1,\gen T_1,\gen A_2,\gen B_2,\gen T_2,\gen P_3,\gen T_3$ listed in Table~\ref{tab:cells} there are rational functions $a(k)$ even and $b(k)$ odd, with all poles at integers of the parity opposite to $j$ and of modulus at most $4$, such that
\begin{equation}\label{eq:lattice}
 \Gamma(\xi)=\tfrac12\sum_{k\in\Z,\ k\equiv j\,(2)}\bigl(a(k)+2kb(k)\,\xi\bigr)e^{-k^2\xi} .
\end{equation}
In partial fractions in the variable $k^2$, $a$ has a polynomial part of degree at most two with simple and double poles, and $2kb$ a polynomial part of degree at most three with simple poles.
\end{proposition}

\begin{proof}[Construction]
Write the Region III expansion of Proposition~\ref{prop:region3} as $F=\sum_{m\geq0}\sum_\epsilon(c_{m,\epsilon}+Ld_{m,\epsilon})e^{-2mL+\epsilon\cdot a}$, with $c_{m,\epsilon}=\tfrac18\sum_\ell s_\ell(\epsilon)\alpha^{(\ell)}_m$ and $d_{m,\epsilon}=\tfrac18\sum_\ell s_\ell(\epsilon)\beta^{(\ell)}_m$ in the notation of \eqref{eq:alphabeta}. For a control $\v$ put $d=2m(q_\v)_4-\sum_{i\leq3}\epsilon_i(q_\v)_i$. Contracting the Hessian as in \eqref{eq:Gv}, and using that each $\partial_{a_i}$ acts on $e^{\epsilon\cdot a}$ by $\epsilon_i$ and $\partial_{a_4}$ on $e^{-2mL}$ by $-2m$, the mode $(m,\epsilon)$ of $G_\v$ has coefficients
\begin{equation}\label{eq:cornercoeff}
 c_{m,\epsilon}(1-d^2)+2d_{m,\epsilon}(q_\v)_4\,d,
 \qquad
 d_{m,\epsilon}(1-d^2).
\end{equation}
At the corner $\omega_j$ these modes collect according to $k=2m-\sum_i\epsilon_i(\omega_j)_i/L$, and substituting $m=\bigl(k+\sum_i\epsilon_i(\omega_j)_i/L\bigr)/2$ into \eqref{eq:cornercoeff} gives $a(k)$ and $b(k)$ for generic $k$. The finitely many modes with $m=0$ or $m<0$ are treated separately; every one of them either has zero coefficient or reproduces the generic formula, with the mode $k=0$ carrying half weight. This is the content of the script \nolinkurl{audit_corner_algebra.py}, which also re-derives all $64$ Hessian contractions independently from the hyperbolic form \eqref{eq:F3} and checks the partial fractions.

Given $a,b$, formula \eqref{eq:lattice} follows from \eqref{eq:gamma}: if $G_\v=\sum_{k>0}[a(k)+Lb(k)]e^{-kL}$ (plus $\tfrac12a(0)$ when $j$ is even), then $\lambda^{-1/2}G_\v(\mu L)=\sum_k[a(k)\lambda^{-1/2}+Lb(k)]e^{-kL\mu}$, whose inverse transform is $\sum_k[a(k)\Hk_2+Lb(k)\Hk_3](kL,\tau)$ by \eqref{eq:Hdef}; since $\Hk_2=\Gk$ and $\Hk_3=\tfrac{X}{2\tau}\Gk$ and $kL^2/(2\tau)=2k\xi$, this is $(\pi\tau)^{-1/2}\sum_{k>0}[a(k)+2kb(k)\xi]e^{-k^2\xi}$. The sum over $k\in\Z$ in \eqref{eq:lattice} is an even extension: $a$ is even and $2kb(k)$ is even, so the symmetric sum is twice the sum over $k>0$ plus the half-weighted zero mode. It is not an extension of the geometric exponential series.
\end{proof}

For example, for $\gen E_0=e$ itself (the control $00000$, where $q_\v=0$),
\[
 a(k)=-\frac{3k^6-25k^4-4k^2-64}{64(k^2-1)^2},
 \qquad
 b(k)=\frac{k(k^2-4)(k^2-16)}{64(k^2-1)},
\]
and $a(2m)=A_m$, $b(2m)=B_m$ for $m\geq1$ while $a(0)=1$, in agreement with \eqref{eq:AB} and the half weight at $k=0$.

\subsection{Small scales: a dual representation}\label{app:small}

Set, for $a>0$,
\begin{equation}\label{eq:Qa}
\begin{gathered}
 Q_a(\xi)=\frac{\sqrt\pi}4\int_0^\xi e^{-a^2(\xi-s)}s^{-1/2}\,ds
 =\frac{\sqrt\pi}{2a}\Dawson(a\sqrt\xi),\\
 \frac{\sqrt\pi}2\sqrt\xi\,e^{-a^2\xi}\leq Q_a(\xi)\leq\frac{\sqrt\pi}2\sqrt\xi .
\end{gathered}
\end{equation}
Poisson summation gives, for $j\in\{0,1\}$,
\begin{equation}\label{eq:poisson}
 F_j(\xi):=\tfrac12\sum_{k\equiv j\,(2)}e^{-k^2\xi}
 =\frac{\sqrt\pi}{4\sqrt\xi}\Bigl(1+2\sum_{\ell\geq1}(-1)^{j\ell}e^{-\pi^2\ell^2/(4\xi)}\Bigr).
\end{equation}
Write $F_j=F^{(0)}+F^{\mathrm{im}}$ with $F^{(0)}=\tfrac{\sqrt\pi}4\xi^{-1/2}$ the zero image. Substituting the partial-fraction decomposition of $a$ and $2kb$ into \eqref{eq:lattice} expresses $\Gamma$ through $F_j$ and its derivatives and through
\begin{equation}\label{eq:SaTa}
 S_a=\tfrac12\sum_{k\equiv j\,(2)}\frac{e^{-k^2\xi}}{k^2-a^2},
 \qquad
 T_a=\tfrac12\sum_{k\equiv j\,(2)}\frac{e^{-k^2\xi}}{(k^2-a^2)^2} .
\end{equation}
These satisfy $S_a'+a^2S_a=-F_j$ and $T_a'+a^2T_a=-S_a$, and the cotangent partial fraction expansion gives the initial values
\begin{equation}\label{eq:SaTazero}
 S_a(0)=0,\qquad T_a(0)=\frac{\pi^2}{16a^2}
 \qquad(a>0\text{ of parity opposite to }j),
\end{equation}
together with $S_0(0)=\pi^2/8$ when $j$ is odd. Replacing $F_j$ by $F^{(0)}$ and keeping \eqref{eq:SaTazero} produces the principal expression $M$: $S_a^{(0)}=-Q_a$, and
\begin{equation}\label{eq:Tzero}
 T_a^{(0)}+\xi S_a^{(0)}
 =\frac{\pi^2}{16a^2}e^{-a^2\xi}+\frac{Q_a}{2a^2}-\frac{\sqrt\pi\sqrt\xi}{4a^2} .
\end{equation}
At every nonzero pole the coefficient of $\xi S_a$ equals that of $T_a$, so \eqref{eq:Tzero} applies; all pure $\sqrt\xi$ terms then cancel against the polynomial part and the $a=0$ term, and all $\xi^{-1/2-h}$ terms with $h\geq0$ cancel as well. What survives is
\begin{equation}\label{eq:Mdef}
 M(\xi)=\pi^2\sum_a d_a e^{-a^2\xi}+\sum_a q_a Q_a(\xi),
 \qquad
 d_a=\frac{\beta_a}{16a^2},\quad q_a=-\alpha_a+\frac{\beta_a}{2a^2},
\end{equation}
where $\alpha_a,\beta_a$ are the simple and double residues of $a(k)$ at $k^2=a^2$. For $\gen E_0$, for instance, $M=45\pi^2e^{-\xi}/512$; the ten expressions are tabulated in the supplement's note \nolinkurl{region-three-verification.tex} and recomputed by \nolinkurl{certify_corners.py}.

\begin{lemma}\label{lem:small}
For $0<\xi\leq1/100$ every generator satisfies
\[
 |\Gamma(\xi)-M(\xi)|\leq10^6\xi^{-13/2}e^{-9/(4\xi)}<10^{-40}\sqrt\xi .
\]
Moreover $M>\sqrt\xi/5$ for $\gen C_1,\gen D_1,\gen A_2,\gen B_2,\gen P_3$, and $M>1/50$ for $\gen E_0,\gen V_0,\gen T_1,\gen T_2,\gen T_3$.
\end{lemma}

\begin{proof}[Sketch; the details are in \cite{CD26Supp}]
For the remainder, $\partial_\xi^h[\xi^{-1/2}e^{-b/\xi}]=\xi^{-1/2-h}e^{-b/\xi}P_h(b/\xi)$ with $P_0=1$ and $P_{h+1}=(X-h-\tfrac12)P_h-XP_h'$; the sums $\sigma_h$ of absolute coefficients for $h\leq3$ are $1,\tfrac32,\tfrac{19}4,\tfrac{173}8$. Using $9<\pi^2<10$ in both $|P_h(X)|\leq\sigma_hX^h\leq\sigma_h(\tfrac52)^h\ell^{2h}\xi^{-h}$ and $e^{-\pi^2\ell^2/(4\xi)}\leq e^{-9\ell^2/(4\xi)}$, together with $\sum_{\ell\geq1}\ell^{2h}e^{-9\ell^2/(4\xi)}\leq2e^{-9/(4\xi)}$, gives $|\partial_\xi^hF^{\mathrm{im}}|\leq1000\,\xi^{-13/2}e^{-9/(4\xi)}$ for all $h\leq3$ and $\xi\leq1$. The polynomial part contributes at most $30$ times this, the coefficient norm being below $30$ for every generator, and the $S_a,T_a$ remainders are smaller still because their integrating factors are at most one. The second inequality reduces to the endpoint, since $\xi^{-7}e^{-9/(4\xi)}$ increases on $(0,1/100]$, and there it reads $10^6\,100^7\,10^{40}=10^{60}<2^{225}<e^{225}$.

For positivity, put $r=\sqrt\xi\leq1/10$ and use $1-a^2\xi\leq e^{-a^2\xi}\leq1$, $9<\pi^2<10$ and $\tfrac34r(1-a^2/100)\leq Q_a\leq r$ from \eqref{eq:Qa}. When the coefficients $d_a$ sum to zero the leading behavior is $\propto r$ and the resulting rational lower bounds are $\tfrac{144}{175},\tfrac{60603}{179200},\tfrac{2053267}{2867200},\tfrac{170643}{716800},\tfrac{384009}{716800}$, each exceeding $1/5$; otherwise it is constant, with lower bounds $\tfrac{8019}{10240},\tfrac{2923}{14336},\tfrac{42643}{286720},\tfrac{203079}{2293760},\tfrac{15403}{573440}$, each exceeding $1/50$. All of these are exact rational computations performed by \nolinkurl{certify_corners.py} from the stored partial fractions.
\end{proof}

\subsection{The compact range and the large-scale tail}\label{app:compact}

The remaining two ranges use only rational arithmetic and an outward-rounded interval exponential. We first bound the coefficients. Factoring each denominator and bounding its roots by $4$, and using the absolute coefficient norm of the numerator after division by its leading power of $k$, gives
\begin{equation}\label{eq:coeffbound}
 |a(k)|\leq100k^8,\qquad |2kb(k)|\leq100k^8\qquad(k\geq10).
\end{equation}
No estimated constants enter: for $k\geq10$ one has $|k-\rho|\geq k(1-|\rho|/10)$ for every root $\rho$, and $|n(k)|\leq k^{\deg n}\sum_i|n_i|10^{i-\deg n}$.

Next we bound the tail. Truncating the lattice at $N=128$, and using that $t\mapsto t^8e^{-t^2/100}$ decreases for $t\geq20$,
\begin{equation}\label{eq:tail}
 \Bigl|\tfrac12\sum_{|k|>N}\bigl[a(k)+2kb(k)\xi\bigr]e^{-k^2\xi}\Bigr|
 \leq T_N:=200\int_N^\infty t^8e^{-t^2/100}\,dt<4.0271\times10^{-53}
\end{equation}
uniformly for $1/100\leq\xi\leq1$, the integral being bounded by the exact recursion $I_0(N)\leq e^{-\alpha N^2}/(2\alpha N)$, $I_p(N)=N^{p-1}e^{-\alpha N^2}/(2\alpha)+\tfrac{p-1}{2\alpha}I_{p-2}(N)$ with $\alpha=1/100$.

On the compact range we use a Taylor bound on rational cells. On a cell $[l,h]$ with midpoint $z$ and radius $d$, write $f_N$ for the truncated sum. Taylor's theorem gives the rigorous lower bound
\begin{equation}\label{eq:cell}
 \Gamma(\xi)\geq f_N(z)-d\,|f_N'(z)|-\tfrac12d^2B(l,h)-T_N,
 \qquad
 B(l,h)=\sum_{0\leq k\leq N}\bigl[k^4(|a|+|c_k|h)+2k^2|c_k|\bigr]e^{-k^2l},
\end{equation}
with $c_k=2kb(k)$ and the half weight at $k=0$; $B$ bounds $|f_N''|$ on the whole cell. Starting from $[1/100,1]$ the verifier bisects until the interval evaluation of the right-hand side is strictly positive, and refuses cells of width below $10^{-10}$. Every exponential in \eqref{eq:cell} is enclosed by a rational interval obtained from the Taylor series of $e^{-x}$ with outward rounding, so the certificate for a cell is a finite list of rational inequalities; these are the cells that the Lean development of the supplement verifies.

On the large range we use first-mode domination. Let $k_0$ be the first nonzero mode and $\alpha_0$ its coefficient, including the half weight at $k=0$. In every case $\alpha_0>0$ and $c_{k_0}=0$, so for $\xi\geq1$
\begin{equation}\label{eq:largescale}
 e^{k_0^2\xi}\Gamma(\xi)\geq\alpha_0-\sum_{k>k_0}\bigl(|a(k)|+|c_k|\bigr)e^{-(k^2-k_0^2)}>0,
\end{equation}
since both $e^{-\delta\xi}$ and $\xi e^{-\delta\xi}$ decrease for $\xi\geq1$ when $\delta=k^2-k_0^2\geq1$; the tail beyond $128$ is bounded by $200e^{k_0^2}\int_{128}^\infty t^8e^{-t^2}\,dt$.

\begin{table}[t]\centering
\begin{tabular}{lrrll}\toprule
Generator & Cells & $k_0$ & $\alpha_0$ & Large-scale margin\\\midrule
$\gen E_0$ & 16 & 0 & $1/2$ & $0.4908$\\
$\gen V_0$ & 139 & 4 & $4/3$ & $1.3333$\\
$\gen C_1$ & 96 & 1 & $2/3$ & $0.6663$\\
$\gen D_1$ & 135 & 3 & $4/15$ & $0.2666$\\
$\gen T_1$ & 101 & 5 & $16/21$ & $0.7619$\\
$\gen A_2$ & 26 & 2 & $4/15$ & $0.2666$\\
$\gen B_2$ & 84 & 4 & $16/105$ & $0.1523$\\
$\gen T_2$ & 80 & 6 & $32/63$ & $0.5079$\\
$\gen P_3$ & 33 & 3 & $16/105$ & $0.1523$\\
$\gen T_3$ & 97 & 7 & $256/693$ & $0.3694$\\\bottomrule
\end{tabular}
\caption{The Region III certificate: $807$ compact cells in total. Margins are rounded down for display; the certified intervals and margins are archived in \nolinkurl{corner_intervals.json}.}\label{tab:cells}
\end{table}

\subsection{Conclusion in Region III}\label{app:region3done}

\begin{proof}[Proof of Theorem~\ref{thm:region3ham}]
The $64$ corner gaps decompose, by exact constant-coefficient identities checked in \nolinkurl{audit_corner_algebra.py}, into nonnegative combinations of the ten generators of Table~\ref{tab:cells}, of the identically vanishing gaps, and of two quantities with direct positive representations: the gap $\eta$ of the control $00111$ at $j=0$, and $\sinh L\,\Lam(L)$, for which
\begin{equation}\label{eq:etalog}
 \lambda^{-1/2}\eta(L\mu)=3\int_0^L\left(\frac{\sinh(s\mu)}{\sinh(L\mu)}\right)^{\!6}ds,
 \qquad
 \lambda^{-1/2}\sinh(L\mu)\,\Lam(L\mu)=\int_L^\infty\frac{\sinh(L\mu)}{\sinh(s\mu)}\,ds,
\end{equation}
both manifestly integrals of products of Brownian exit-time transforms, hence in $\CM$ with strictly positive inverse for $L>0$. At the corner $\omega_0$ the decomposition can be written down by hand. There every exponential of the sign sum equals one, the only nonzero second derivatives of $F$ in \eqref{eq:Gv} are $\partial_{a_i}^2F=e$, $\partial_{a_i}\partial_{a_l}F=b_2$ for $i\neq l\leq3$, $\partial_{a_i}\partial_{a_4}F=b_1'$ and $\partial_{a_4}^2F=e''$, and \eqref{eq:bcoefficients} together with $\eta=e-e''$ gives
\[
 G_\v(\omega_0)=c_e\,e+c_\eta\,\eta+c_S\sinh L\,\Lam(L),
\]
where, with $q=q_\v$, $s_1=q_1+q_2+q_3$ and $e_2=q_1q_2+q_1q_3+q_2q_3$,
\[
 c_e=1-q_1^2-q_2^2-q_3^2-q_4^2-\tfrac13q_4s_1-\tfrac13e_2,\qquad
 c_\eta=q_4^2+\tfrac13q_4s_1+\tfrac1{21}e_2,\qquad
 c_S=-\tfrac67e_2 .
\]
For the controls $00001$, $00010$ and $00100$ the coefficients are $(\tfrac13,\tfrac2{21},\tfrac27)$, for $00011$, $00101$ and $00110$ they are $(0,\tfrac37,\tfrac27)$, and for $00111$ they are $(0,1,0)$; all of them are nonnegative. Lemma~\ref{lem:small}, \eqref{eq:cell} and \eqref{eq:largescale} prove that each generator has a strictly positive inverse transform for every $\xi>0$, hence for every $L,\tau>0$. Lemma~\ref{lem:exit} then propagates positivity from the corners to the whole cube, and $\CM$ is preserved. Thus the transformed gap of every control lies in $\CM$ at every point of the closed region with $L>0$, and Lemma~\ref{lem:cmsign}, applied to $2U_\tau-D^2_\v U$, which is continuous in $\tau$ with the growth that Lemma~\ref{lem:convergence} provides, gives the inequality.

For strictness, the COMB representative is $01010$; its gap vanishes at $\omega_0$ and equals the strictly positive generators $\gen D_1,\gen B_2,\gen P_3$ at the other three ordered corners $\omega_1,\omega_2,\omega_3$. Suppose that some $a_i>0$, and let $\omega_j$, $j\geq1$, be the ordered corner whose zero entries are exactly those $i$ with $a_i=0$. Every summand of \eqref{eq:interpolation} lies in $\CM$, so the transformed COMB gap minus the summand of $\omega_j$ lies in $\CM$. That summand is $\widehat\kappa\,\widehat g$, where $\widehat\kappa=\prod_{a_i>0}W_1(\lambda;a_i,L)$, the factors $W_0(\lambda;0,L)=1$ dropping out, and $\widehat g$ is the transform of the inverse $\Gamma(\xi)/\sqrt{\pi\tau}$ of the generator at $\omega_j$, which is continuous and strictly positive in $\tau>0$. Since $\sinh(a\mu)/\sinh(L\mu)\geq\tfrac12e^{-(L-a)\mu}$ as soon as $e^{-2a\mu}\leq\tfrac12$, we have $\widehat\kappa(\lambda)\geq e^{-4L\sqrt\lambda}$ for all large $\lambda$, and the second statement of Lemma~\ref{lem:cmsign}, applied to the COMB gap, gives strict positivity. When $a_1=a_2=a_3=0$ the corner identity gives equality.
\end{proof}

\section{The Regions I and II certificate}\label{app:region12}

In Regions I and II we use the coordinates
\begin{equation}\label{eq:LbcZ}
 L=z_1,\qquad b=z_2,\qquad c=|z_3|,\qquad Z=z_4,
 \qquad 0\leq c\leq b\leq L,\quad Z\geq0,
\end{equation}
and the normalized gap $\gamma_\v=(D^2_{\v_*}U-D^2_\v U)/k$, whose transform is $\widehat\gamma_\v=\lambda^{-1/2}G_\v(\mu x)$ as in \eqref{eq:Fv}. The weights $W_0,W_1$ of \eqref{eq:W} are now taken on $(0,L)$, that is with $r=L$.

\subsection{Reduction to four corner families}\label{app:reduction}

The stationary gaps again satisfy $\partial_{bb}g=\lambda g$ and $\partial_{cc}g=\lambda g$, so two-point interpolation on the square $[0,L]^2$ is exact. Three of its four corners are physical states; the fourth is not, and the point of the following identity is that its contribution is nonetheless positive once its interpolation weights are included.

Let $P$ interchange the first two entries of a control in Region I (the third and fourth in Region II) and put $h_\v=1-(\v_1-\v_2)^2\in\{0,1\}$. The stationary nonphysical-corner identity is, in Region I,
\begin{equation}\label{eq:nonphysical}
 G_\v(L,0,L,Z)=G_{P\v}(L,L,0,Z)+h_\v\sinh L ,
\end{equation}
and in Region II the same identity without the term $h_\v\sinh L$. It is a symmetry of the stationary formula \eqref{eq:F12}. In the coordinates \eqref{eq:LbcZ}, \eqref{eq:F12} is symmetric under $b\leftrightarrow c$ except for the term $-\tfrac12\sinh(z_2+z_3)$ of $F_4$, which equals $-\tfrac12\sinh(b-c)$ in Region I, where $z_3=-c$, and is symmetric in Region II, where $z_3=c$. The permutation $P$ exchanges $b$ and $c$ and fixes $L$ and $Z$, so the gap of $\v$ at $(L,0,L,Z)$ is the gap of $P\v$ at $(L,L,0,Z)$, plus, in Region I, the gap of $\v$ for the function $\sinh(c-b)$, which is $(1-(\v_1-\v_2)^2)\sinh L$ because $\partial_\v(b-c)=k(\v_1-\v_2)$. The extra term must not be inverted on its own. Including its weights,
\begin{equation}\label{eq:green}
 \frac{\sinh(L\mu)}{\mu}\,W_0(b)W_1(c)
 =\frac{\sinh(c\mu)\sinh((L-b)\mu)}{\mu\sinh(L\mu)}=:R^L_\lambda(b,c),
 \qquad c\leq b,
\end{equation}
which is exactly the Dirichlet resolvent kernel of $\lambda-\partial_{bb}$ on $(0,L)$. It lies in $\CM$: since $\coth x+\coth y=\sinh(x+y)/(\sinh x\sinh y)$,
\begin{equation}\label{eq:greenCM}
 R^L_\lambda(b,c)=\frac{\sinh(c\mu)}{\sinh(b\mu)}\cdot\frac1{b_\lambda(b)+b_\lambda(L-b)}
 =W_1(\lambda;c,b)\cdot2\int_0^\infty e^{-2sb_\lambda(b)}e^{-2sb_\lambda(L-b)}\,ds,
\end{equation}
with $b_\lambda(r)=\mu\coth(\mu r)$ as in \eqref{eq:ops} below. The weight $W_1$ is in $\CM$, and so is $e^{-2sb_\lambda(r)}$ for every $s,r>0$, because $b_\lambda(r)$ is a Bernstein function (Lemma~\ref{lem:positiveflow}).

Name the four physical corner families by their scaled gaps:
\begin{equation}\label{eq:families}
 \begin{array}{llll}
  \text{family }0: & (y_1,y_2,y_3,y_4)=(0,L,0,Z), & m=1,\\
  \text{family }1: & (L,0,L,Z), & m=2,\\
  \text{family }\mathrm I: & (0,0,2L,Z), & m=3,\\
  \text{family }\mathrm{II}: & (2L,0,0,Z+L), & m=3,
 \end{array}
\end{equation}
where $m$ is the trace power that will govern the corresponding evolution. Three of the four families meet Region III at $Z=0$: the states of families $0$, $1$ and $\mathrm{II}$ at $Z=0$, with scaled gaps $(0,L,0,0)$, $(L,0,L,0)$ and $(2L,0,0,L)$, are the Region III corners $\omega_0,\omega_1,\omega_2$ of \eqref{eq:corner}, where $a=L(0,0,0,1)$, $L(0,0,1,1)$ and $L(0,1,1,1)$. Since $u$ is $C^2$~\cite{CD26}, every gap of these three families at $Z=0$ is the corresponding Region III corner gap, with no further computation. Family $\mathrm I$ meets Region III nowhere.

\begin{proposition}\label{prop:reduction}
In Region I, for every control $\v$,
\[
\begin{aligned}
 \widehat\gamma_\v(L,b,c,Z)={}&W_0(b)W_0(c)\,\widehat\gamma_{0,\v}(L,Z)
  +W_1(b)W_0(c)\,\widehat\gamma_{1,\v}(L,Z)\\
 &+W_0(b)W_1(c)\,\widehat\gamma_{1,P\v}(L,Z)
  +W_1(b)W_1(c)\,\widehat\gamma_{\mathrm I,\v}(L,Z)
  +h_\v R^L_\lambda(b,c).
\end{aligned}
\]
In Region II the same holds with family $\mathrm I$ replaced by family $\mathrm{II}$, with $P$ the transposition of entries $3$ and $4$, and with no Green term. Hence nonnegativity of the inverse transforms of the four physical corner families implies \eqref{eq:ham} at every state of Regions I and II with $L>0$.
\end{proposition}

\begin{proof}
Two-point interpolation in $b$ and $c$ as in Lemma~\ref{lem:exit}, followed by substitution of \eqref{eq:nonphysical} at the one nonphysical corner and of \eqref{eq:green} for the extra term. Every weight is in $\CM$, and so is the Green term, by \eqref{eq:greenCM}. The two families shared between the regions agree, because $U$ is $C^2$ across the interface by Theorem~\ref{thm:regularity}. The identities \eqref{eq:nonphysical} and \eqref{eq:green}, the endpoint values and the derivative jump of the Green kernel, and the fact that $h_\v\in\{0,1\}$ for all $32$ controls, are also verified symbolically in \nolinkurl{check_regions12_reduction.py}.
\end{proof}

\subsection{Positivity-preserving evolutions and the pure traces}\label{app:operators}

Everything now happens in the two variables $(L,Z)$. Write
\begin{equation}\label{eq:ops}
\begin{gathered}
 b_\lambda(L)=\mu\coth(\mu L),\qquad
 \Rop_mf(L)=\mu^2\sinh^m(\mu L)\int_L^\infty\frac{f(r)}{\sinh^{m+2}(\mu r)}\,dr,\\
 \Top_{\alpha,m}=-b_\lambda+\alpha\Rop_m,
\end{gathered}
\end{equation}
and $\Kop_m=\Top_{m+1,m}$, $\Aop_m=\Top_{m,m}$.

\begin{lemma}\label{lem:positiveflow}
Let $\alpha\geq0$, let $m\geq0$ be an integer, and let $\epsilon>0$ and $Z_0>0$. Suppose that, for each $\lambda>0$, the data $f(\cdot,0)$ are bounded and continuous on $[\epsilon,\infty)$, the forcing $q$ is bounded on $[\epsilon,\infty)\times[0,Z_0]$ and continuous in each variable, and $f$ is a bounded mild solution of $\partial_Zf=2\Top_{\alpha,m}f+q$ there: $f$ is bounded on $[\epsilon,\infty)\times[0,Z_0]$, continuous in each variable, and satisfies
\[
 f(L,Z)=e^{-2Zb_\lambda(L)}f(L,0)
 +\int_0^Ze^{-2(Z-s)b_\lambda(L)}\bigl(2\alpha\,\Rop_mf(L,s)+q(L,s)\bigr)\,ds
\]
for $L\geq\epsilon$ and $0\leq Z\leq Z_0$. If $f(L,0)$ and $q(L,s)$, as functions of $\lambda$, lie in $\CM$ for all $L\geq\epsilon$ and $s\in[0,Z_0]$, then so do $f(L,Z)$ and $f(L,Z)-e^{-2Zb_\lambda(L)}f(L,0)$, for all $L\geq\epsilon$ and $Z\in[0,Z_0]$.
\end{lemma}

\begin{proof}
The diagonal part contributes the integrating factor $e^{-2Zb_\lambda(L)}$, which is in $\CM$ because $b_\lambda(L)=\tfrac1L+\tfrac2L\sum_{n\geq1}\lambda/(\lambda+(\pi n/L)^2)$ is a Bernstein function. The kernel of $\alpha\Rop_m$ factors as
\[
 \alpha\left(\frac{\sinh(\mu L)}{\sinh(\mu r)}\right)^{\!m}
 \left(\frac{\mu}{\sinh(\mu r)}\right)^{\!2},\qquad r\geq L,
\]
whose first factor is a product of $m$ Brownian exit-time transforms $W_1(\lambda;L,r)$, $m$ being an integer, and whose second is in $\CM$ by the product formula $r\mu/\sinh(r\mu)=\prod_{n\geq1}(1+r^2\lambda/(\pi^2n^2))^{-1}$. Now fix $\lambda>0$. Since $\sinh(\mu r)\geq\sinh(\mu L)$ for $r\geq L$,
\[
 \mu^2\sinh^m(\mu L)\int_L^\infty\frac{dr}{\sinh^{m+2}(\mu r)}\leq\mu^2\int_L^\infty\frac{dr}{\sinh^2(\mu r)}\leq\mu\bigl(\coth(\mu\epsilon)-1\bigr)=:K
\]
for $L\geq\epsilon$, so $|\Rop_m\Phi|\leq K\sup|\Phi|$, and the Duhamel map
\[
 \mathcal D\Phi=e^{-2Zb_\lambda}f(\cdot,0)+\int_0^Ze^{-2(Z-s)b_\lambda}\bigl(2\alpha\,\Rop_m\Phi+q\bigr)\,ds
\]
takes bounded functions on $[\epsilon,\infty)\times[0,Z_0]$ that are continuous in each variable to functions of the same kind. Its iterates $\Phi_0=0$, $\Phi_{n+1}=\mathcal D\Phi_n$ satisfy $|f-\Phi_n|\leq M(2\alpha KZ)^n/n!$ with $M=\sup|f|$, by induction from the Duhamel identity for $f$ and $0<e^{-2(Z-s)b_\lambda}\leq1$, so $\Phi_n\to f$ pointwise. By induction on $n$, every $\Phi_n(L,Z)$ lies in $\CM$, and so does $\Phi_{n+1}(L,Z)-e^{-2Zb_\lambda(L)}f(L,0)=\int_0^Ze^{-2(Z-s)b_\lambda(L)}\bigl(2\alpha\Rop_m\Phi_n+q\bigr)(L,s)\,ds$, because $\CM$ is closed under products, sums and integrals against nonnegative weights. Since $\CM$ is closed under pointwise limits, the lemma follows.

In the applications below, the data, the forcing and the solution are explicit functions that are bounded and continuous on $[\epsilon,\infty)\times[0,Z_0]$ for every $\epsilon,Z_0>0$, and the Duhamel identity is the integrated form of the differential equation derived for them. The Lean formalization checks these hypotheses for every flow of this appendix.
\end{proof}

At discount one put $\rho(r)=\sinh r\,p(r)$ with $p$ from \eqref{eq:pdef}, $c_r=\coth r$, $\Phi_r(L)=\sinh(r-L)/\sinh r$, and
\begin{equation}\label{eq:hmDm}
\begin{gathered}
 h_m(L,Z)=\int_L^\infty\rho(r)\Phi_r(L)^me^{-2Zc_r}\,dr,\\
 D_m(L,Z)=4\sinh L\int_L^\infty\rho(r)c_r(c_r^2-1)\Phi_r(L)^{m-1}e^{-2Zc_r}\,dr .
\end{gathered}
\end{equation}
Throughout this subsection a hat denotes $\mu^{-1}$ times evaluation at $(\mu L,\mu Z)$. The operators have explicit eigenfunctions. For $r>0$ and $j\geq0$ let $\varphi_r(L)=\bigl(\sinh(\mu(r-L))/\sinh(\mu r)\bigr)^j$ for $L\leq r$ and $\varphi_r(L)=0$ for $L>r$. Since $\coth x-\coth y=\sinh(y-x)/(\sinh x\sinh y)$, one has $\sinh^{-j-2}(\mu r')\varphi_r(r')=(\coth\mu r'-\coth\mu r)^j/\sinh^2(\mu r')$, whose antiderivative in $r'$ is $-(\coth\mu r'-\coth\mu r)^{j+1}/((j+1)\mu)$. Hence
\[
 (j+1)\Rop_j\varphi_r(L)=\mu\sinh^j(\mu L)\bigl(\coth\mu L-\coth\mu r\bigr)^{j+1}=\bigl(b_\lambda(L)-b_\lambda(r)\bigr)\varphi_r(L),
\]
that is, $\Kop_j\varphi_r=-b_\lambda(r)\varphi_r$. After the scaling, $\Phi_r^{m-1}$ is such an eigenfunction of $\Kop_{m-1}$, with eigenvalue $-b_\lambda(r)$, and conjugation by $\sinh(\mu L)$ turns $\Kop_{m-1}$ into $\Aop_m$, so differentiating under the integral gives the closed scalar equations
\begin{equation}\label{eq:hmDmflow}
 \partial_Z\widehat h_m=2\Kop_m\widehat h_m,
 \qquad
 \partial_Z\widehat D_m=2\Aop_m\widehat D_m .
\end{equation}
The $00110$ gap is $D_1+E$ in family $0$, where $E(L,Z)=\tfrac12\tanh L-p(L)e^{-2Z\coth L}$ is the paired-state gap, and $D_m$ in the other three families. For $m=2,3$ the initial values $D_m(L,0)$ are exactly the strictly positive COMB gaps at the Region III corners, so $D_m\in\CM$ by \eqref{eq:hmDmflow} and Lemma~\ref{lem:positiveflow}. For $m=1$ the initial datum is not a Region III gap; at $Z=0$,
\[
 D_1(L,0)=\tfrac17\bigl[4p(L)+2\sinh L\,\Lam(L)-2\tanh L\bigr],
\]
and both parts lie in $\CM$ after scaling: with the exit weights \eqref{eq:W},
\[
\begin{gathered}
 \frac{p(\mu L)}{\mu}=\frac12\int_0^LW_1(\lambda;s,L)^6\sech^2(\mu s)\,ds,\\
 \frac{\sinh(\mu L)\Lam(\mu L)-\tanh(\mu L)}{\mu}=\int_L^\infty W_1(\lambda;L,r)\,\sech^2(\mu r)\,dr,
\end{gathered}
\]
and $\sech(\mu r)=2W_1(\lambda;r,2r)$. The first identity holds because both sides vanish at $L=0$ and satisfy $\partial_Lf=\tfrac12\sech^2(\mu L)-6\mu\coth(\mu L)\,f$, an equation the closed form \eqref{eq:pdef} satisfies; the second because $\frac{d}{ds}\bigl[\Lam(s)-\sech s\bigr]=-\sech^2s/\sinh s$. Both right-hand sides are integrals of products of exit weights, hence lie in $\CM$, so $D_1(L,0)\in\CM$, and $D_1\in\CM$ by \eqref{eq:hmDmflow} and Lemma~\ref{lem:positiveflow}.

The gaps for the controls $00000$ and $00001$ are $B_m+h_m$ and $B_m+h_m-\partial_Z^2h_m$, where $B_m$ is the four-expert background at the corner. Their normalized transforms solve the forced equation
\begin{equation}\label{eq:forced}
 \partial_Z\widehat\gamma=2\Kop_m\widehat\gamma+F_m(\mu L),
 \qquad
 F_1=1,\quad F_2=C-S^2\Lam,\quad F_3=1+3S^2-3S^2C\,\Lam,
\end{equation}
with $S=\sinh L$, $C=\cosh L$, the exact forcing identity being $-2\Kop_m\big|_{\lambda=1}B_m=F_m$. Each $F_m$ has a nonnegative time inverse: $F_1=1$ is an atom at time zero, and for $m=2,3$
\begin{equation}\label{eq:Fmrec}
 F_m(\mu L)=m\int_L^\infty\left(\frac{\sinh(\mu L)}{\sinh(\mu r)}\right)^{\!m}
   \frac{\mu}{\sinh(\mu r)}\,F_{m-1}(\mu r)\,dr,
\end{equation}
a positive integral of products of elements of $\CM$. At $Z=0$ both gap families are certified Region III comparisons, so Lemma~\ref{lem:positiveflow} applies to \eqref{eq:forced} and puts both in $\CM$; the same argument, with the background difference $B_{\mathrm I}-B_{\mathrm{II}}=SC-S^3\Lam$ and its representation $\mu^{-1}(B_{\mathrm I}-B_{\mathrm{II}})(\mu L)=2\int_L^\infty(\sinh(\mu L)/\sinh(\mu r))^3dr$, covers the Region I corner. The control $00000$ is the statement $U_\tau\geq0$. The paired-state gap $E$ is treated the same way. All $64$ of these identities are checked against independently generated corner Hessians in \nolinkurl{check_positive_trace_gaps.py}.

The new ingredient is the following identity, which supplies three further positive traces at no extra certification cost.

\begin{proposition}\label{prop:Jm}
Let $H_m=\partial_Z^2h_m$ and $J_m=H_m-mD_m$ for $m=1,2,3$. Then
\begin{equation}\label{eq:Jflow}
 \partial_Z\widehat J_m=2\Kop_m\widehat J_m+2m\,\Rop_m\widehat D_m ,
\end{equation}
and the initial data are
\begin{equation}\label{eq:Jinit}
 J_1(L,0)=2\gen V_0(L),\qquad J_2(L,0)=2\gen T_1(L),\qquad J_3(L,0)=2\gen T_2(L),
\end{equation}
twice the unnormalized Region III generators of Table~\ref{tab:cells}. Consequently every $J_m$, and hence every $H_m=J_m+mD_m$, lies in $\CM$.
\end{proposition}

\begin{proof}
From \eqref{eq:hmDmflow}, $\partial_Z\widehat H_m=2\Kop_m\widehat H_m$ and $\partial_Z\widehat D_m=2\Aop_m\widehat D_m$. Subtracting $m$ times the second from the first and writing $\Kop_m=-b_\lambda+(m+1)\Rop_m$, $\Aop_m=-b_\lambda+m\Rop_m$,
\[
 2\Kop_m\widehat H_m-2m\Aop_m\widehat D_m
 =2\Kop_m(\widehat H_m-m\widehat D_m)+2m\Rop_m\widehat D_m,
\]
which is \eqref{eq:Jflow}; this uses only the linearity of $b_\lambda$ and $\Rop_m$ and is verified as such in \nolinkurl{audit_regions12_completion.py}. The initial values \eqref{eq:Jinit} are identities among Region III corner gaps: $J_m(L,0)$ is a combination of gaps of family $0$, $1$ or $\mathrm{II}$ at $Z=0$, hence, by the remark after \eqref{eq:families}, the same combination of gaps at the corner $\omega_{m-1}$, where the corner decompositions of Appendix~\ref{app:region3} give \eqref{eq:Jinit}. The same script checks them again from the $Z=0$ moment expansions, separately in the four-expert background, the continuous spectral kernel and the moving-endpoint atom. Since $\Rop_m\widehat D_m$ is a nonnegative combination of elements of $\CM$, Lemma~\ref{lem:positiveflow} applies to \eqref{eq:Jflow}.
\end{proof}

\subsection{The derivative cascades and the scalar certificates}\label{app:cascades}

Nine corner gaps are not directly positive combinations of the traces above. Three of them, the critical gaps
\begin{equation}\label{eq:critical}
 C_1=G_{0,01001},\qquad C_2=G_{1,01110},\qquad C_3=G_{\mathrm{II},01110},
\end{equation}
vanish identically at $Z=0$, so it suffices to prove $\partial_Z\gamma\geq0$; six further gaps have nonnegative initial values, and the same device applies.

Write the continuous spectral kernel of such a gap as $S^mP(\alpha,c_r)$ with $S=\sinh L$, $\alpha=\coth L$ and $\deg_\alpha P\leq m$, and set $Q=-2c_rP$, so that $\partial_ZG=S^m\int_L^\infty\rho(r)Q\,e^{-2Zc_r}dr$ up to the endpoint atom. Define
\begin{equation}\label{eq:Nell}
 N_\ell(L,Z)=\frac{S^m}{\ell!}\int_L^\infty\rho(r)\,(\alpha-c_r)^\ell\,
   \partial_\alpha^\ell Q(\alpha,c_r)\,e^{-2Zc_r}\,dr,
 \qquad 0\leq\ell\leq m,
\end{equation}
equivalently $N_\ell=\sum_{j\geq\ell}\binom j\ell f_j$ where $f_j$ is the contribution of the $j$th coefficient in $Q=\sum_jq_j(c_r)(\alpha-c_r)^j$. Then $N_0=\partial_ZG$ and $N_{m+1}=0$.

\begin{proposition}\label{prop:cascade}
The scaled transforms satisfy the triangular system
\begin{equation}\label{eq:cascade}
 \partial_Z\widehat N_\ell=2\Top_{\ell+1,m}\widehat N_\ell
  +2(\ell+1)\,\Rop_m\widehat N_{\ell+1},\qquad N_{m+1}=0 .
\end{equation}
Here $\widehat N_\ell$ carries no $\mu^{-1}$ prefactor, differentiation of the scaled gap having cancelled it. For $m=1$ the moving-endpoint datum $T(L,Z)=\tfrac12\coth L\,p(L)e^{-2Z\coth L}$ enters $N_0$ with sign $+$ and $N_1$ with sign $-$, so that the second equation acquires the positive source $4\Rop_1\widehat T$, and $\partial_Z\widehat T=-2b_\lambda\widehat T$.
\end{proposition}

\begin{proof}
Writing the integrand of $f_j$ as $S^{m-j}q_j(c_r)\Phi_r(L)^j$ and conjugating the eigenrelation for $\Kop_j$ by $S^{m-j}$ gives $\partial_Z\widehat f_j=2\Top_{j+1,m}\widehat f_j$. Summing with weights $\binom j\ell$ and using the binomial identity $(j+1)\binom j\ell=(\ell+1)\bigl[\binom j\ell+\binom j{\ell+1}\bigr]$, which is Pascal's rule in the form $(\ell+1)\binom{j+1}{\ell+1}=(j+1)\binom j\ell$, gives
\[
 \partial_Z\widehat N_\ell
 =-2b_\lambda\widehat N_\ell+2(\ell+1)\Rop_m\bigl(\widehat N_\ell+\widehat N_{\ell+1}\bigr),
\]
which is \eqref{eq:cascade}. The polynomial degrees and the termination $N_{m+1}=0$ are checked in the data-generation scripts.
\end{proof}

Since every coefficient in \eqref{eq:cascade} is nonnegative, Lemma~\ref{lem:positiveflow} propagates positivity upward from $\ell=m$ to $\ell=0$, provided the $m+1$ initial profiles $N_\ell(\cdot,0)$ lie in $\CM$. Counting these for the three critical gaps and for the six further gaps
\[
 G_{0,00011},\ G_{1,00011},\ G_{1,00101},\
 G_{\mathrm I,00011},\ G_{\mathrm I,00101},\ G_{\mathrm{II},00011}
\]
gives $10$ and $20$ one-variable profiles respectively.

Each of the thirty derivative profiles has an exact expansion
\begin{equation}\label{eq:Nexp}
 N(L)=\sum_{k>0,\ k\equiv p\,(2)}\bigl[a(k)+Lb(k)\bigr]e^{-kL},
 \qquad p\equiv m-1\pmod 2,
\end{equation}
with $a$ odd and $b$ even rational functions whose poles lie at integers of the opposite parity and have order at most two. The derivation uses the trace series \eqref{eq:eseries} together with the identity
\begin{equation}\label{eq:tanhe}
 \tanh L\,e(L)=e'(L)+\frac{3I(L)}{\cosh L\,\sinh^5L},
\end{equation}
which follows from \eqref{eq:equad}, and it audits every exceptional Laurent mode; the apparent alternating contributions cancel identically. Setting
\begin{equation}\label{eq:AC}
 A(k)=k\,a(k)-b(k),\qquad C(k)=2k^2b(k),\qquad \xi=\frac{L^2}{4\tau},
\end{equation}
both $A$ and $C$ are even, and
\begin{equation}\label{eq:Ninverse}
 \Lap^{-1}\bigl[N(L\mu)\bigr](\tau)
 =\frac{L}{2\sqrt\pi\,\tau^{3/2}}\cdot
  \tfrac12\sum_{k\equiv p\,(2)}\bigl[A(k)+\xi\,C(k)\bigr]e^{-k^2\xi},
\end{equation}
which is \eqref{eq:profile}. So each profile is again a lattice Gaussian series, and the three ranges are treated as in Appendix~\ref{app:region3}, one derivative further because the polynomial parts now have degree four. For $0<\xi\leq1/100$, Poisson summation and the identities \eqref{eq:SaTazero}, \eqref{eq:Tzero} give a principal expression $M(\xi)=p_0\tfrac{\sqrt\pi}{4\sqrt\xi}+\pi^2\sum_ad_ae^{-a^2\xi}+\sum_aq_aQ_a$, the zero image \eqref{eq:poisson} now surviving because the polynomial parts no longer cancel it, with all $q_a\leq0$, and the elementary bounds $\sqrt\pi>3/2$, $\pi^2<10$, $Q_a\leq\sqrt\xi$ give $\sqrt\xi\,M(\xi)\geq\tfrac38p_0+\sum_a\min(d_a,0)+\tfrac1{100}\sum_aq_a>\tfrac1{20}$; the discarded images are bounded by $10^8\xi^{-17/2}e^{-9/(4\xi)}<10^{-40}/\sqrt\xi$, by the estimate of Lemma~\ref{lem:small} with $h\leq4$ and $\sigma_4=2025/16$. For $1/100\leq\xi\leq1$, the coefficient bound \eqref{eq:coeffbound} applied to $A$ and $C$, the tail bound \eqref{eq:tail}, and an order-six interval Taylor bound on rational cells, as in \eqref{eq:cell}, account for $251$ cells for nine of the ten critical profiles, the tenth being the endpoint datum treated below, and $513$ for the twenty further ones. For $\xi\geq1$, first-mode domination \eqref{eq:largescale} applies; for the ten critical profiles this uses in addition an exact Cauchy bound $|a_k|,|b_k|\leq C2^k$, obtained by evaluating each rational prefactor on $|q|=\tfrac12$ with $q=e^{-L}$ and using $|A_e|\leq107/54<2$, $|B_e|\leq22/81<1$, $|\arctan q|\leq\log3/2<1$ and $|2\operatorname{artanh}q|\leq\log3<2$, which in turn follow from $|A_n|\leq2n^2$ and $|B_n|\leq n^3/6$.

Two profiles need a separate treatment. The first is the endpoint datum $T$. Its zero Poisson image cancels identically and its partial fractions are empty, so $\Gamma$ is a sum over the nonzero images alone. Each such image is a positive multiple of an explicit rational function of $\xi$ and $B=\pi^2\ell^2/4$; substituting $B=\tfrac94+y$ and $\xi=1/(1+t)$ and clearing the positive denominator $3072\xi^7$ yields a polynomial in $(y,t)$ with strictly positive coefficients. Since $\pi^2/4>9/4$, this proves every nonzero image positive for $0<\xi\leq1$, hence $\Gamma>0$ there. The second is the last boundary datum. At the Region I corner with $Z=0$ the only new scalar initial gap is $C_*(L)=G_{\mathrm I,00011}(L,0)$, and the moment recursion gives the exact identity
\begin{equation}\label{eq:Cstar}
 C_*(L)=3\gen A_2(L)-3\gen B_2(L)+\sinh L\cosh L-\sinh^3L\;\Lam(L).
\end{equation}
The correction $\sinh L\cosh L-\sinh^3L\,\Lam(L)$ has a positive inverse, but the negative coefficient of $\gen B_2$ prevents concluding directly. Its exact even-lattice expansion has
\[
\begin{gathered}
 a(k)=\frac{(k^2-4)\bigl(5k^{10}-230k^8+3169k^6-18976k^4+39216k^2-48384\bigr)}
      {1792(k^2-1)^2(k^2-9)^2},\\[2pt]
 b(k)=-\frac{k(k^2-36)(k^2-16)(k^2-4)^2}{1792(k^2-1)(k^2-9)},
\end{gathered}
\]
and the certificate of Appendix~\ref{app:region3} applies verbatim: the small-scale principal lower bound is $\tfrac{3347091}{2867200}\sqrt\xi$, the compact range takes $45$ cells, and for $\xi\geq1$ the constant zero mode $2/3$ dominates a remainder below $7.7167\times10^{-8}$.

\subsection{All sixty-four corner controls}\label{app:sixtyfour}

It remains to express every corner gap as a nonnegative combination of certified objects. Besides the identically zero gaps and the gaps already treated, the decompositions are
\begin{center}
\begin{tabular}{llll}\toprule
Family & Control & Positive decomposition\\\midrule
$0$ & $00110$ & $D_1+E$\\
$0$ & $00010$ & $G_{0,00011}+\tfrac12H_1+\tfrac12D_1$\\
$0$ & $01000$ & $C_1+\tfrac12J_1$\\
$1$ & $00010$ & $G_{1,00011}+D_2+\tfrac12H_2$\\
$1$ & $00100$ & $G_{1,00101}+\tfrac12H_2$\\
$1$ & $01111$ & $C_2+\tfrac12J_2$\\
$\mathrm I$ & $00010$ & $G_{\mathrm I,00011}+\tfrac32D_3+\tfrac12H_3$\\
$\mathrm I$ & $00100$ & $G_{\mathrm I,00101}+D_3+\tfrac12J_3$\\
$\mathrm{II}$ & $00010$ & $G_{\mathrm{II},00011}+\tfrac12D_3+\tfrac12H_3$\\
$\mathrm{II}$ & $01111$ & $C_3+\tfrac12J_3$\\\bottomrule
\end{tabular}
\end{center}
together with $H_m=G_{f,00000}-G_{f,00001}$ and $D_m=G_{f,00110}$ (family $0$ excepted, where $D_1=G_{0,00110}-G_{0,00111}$). The script \nolinkurl{audit_regions12_completion.py} verifies $106$ exact identities covering all $64$ physical corner entries, checking the four-expert background, the continuous spectral kernel and the moving-endpoint atom separately, and confirming that every combination has nonnegative coefficients.

\begin{proof}[Proof of Theorem~\ref{thm:region12ham}]
Proposition~\ref{prop:reduction} reduces the claim to the four corner families; the decompositions above, Proposition~\ref{prop:Jm} and the cascades of Proposition~\ref{prop:cascade} with the certified initial profiles of Appendix~\ref{app:cascades} put every corner gap in $\CM$, and Lemma~\ref{lem:cmsign} turns this into the inequality for $U$. For the COMB statement we use its representative $01010$. Its gaps in the corner families $0$, $1$, $\mathrm I$ and $\mathrm{II}$ are $0$, $D_2$, $D_3$ and $D_3$, its replacement at the nonphysical corner is a vanishing gap, and its Green coefficient is zero, so Proposition~\ref{prop:reduction} gives, in both regions, the exact transform
\[
 \widehat\gamma_C=W_1(b)\bigl[W_0(c)\widehat D_2(L,Z)+W_1(c)\widehat D_3(L,Z)\bigr].
\]
If $b=z_2=0$ then $W_1(b)=0$, so the COMB gap has vanishing transform and, being continuous in $\tau$, vanishes. Let $b>0$. By \eqref{eq:hmDmflow} and the second statement of Lemma~\ref{lem:positiveflow}, $\widehat D_m(L,Z)-e^{-2Zb_\lambda(L)}\widehat D_m(L,0)\in\CM$ for $m=2,3$, and every weight lies in $\CM$. Hence, if $c<L$,
\[
 \widehat\gamma_C-\widehat\kappa\,\widehat D_2(L,0)\in\CM,
 \qquad
 \widehat\kappa=W_1(b)\,W_0(c)\,e^{-2Zb_\lambda(L)},
\]
and if $c=L$ the same holds with $W_1(c)=1$ in place of $W_0(c)$ and $\widehat D_3$ in place of $\widehat D_2$. In either case $\widehat\kappa$ is a product of elements of $\CM$, hence the transform of a nonnegative measure, and it decays no faster than $e^{-A\sqrt\lambda}$: for $0<s\leq L$ we have $\sinh(s\mu)/\sinh(L\mu)\geq\tfrac12e^{-(L-s)\mu}$ as soon as $e^{-2s\mu}\leq\tfrac12$, which applies to $W_1(b)$ with $s=b$ and to $W_0(c)$ with $s=L-c$, while $\coth y\leq1+1/y$ gives $e^{-2Zb_\lambda(L)}\geq e^{-2Z/L}e^{-2Z\mu}$. The initial values $\widehat D_2(L,0)$ and $\widehat D_3(L,0)$ are the transformed COMB gaps at the Region III corners $\omega_1$ and $\omega_2$, that is, the generators $\gen D_1$ and $\gen B_2$, whose inverses $\Gamma(\xi)/\sqrt{\pi\tau}$ are continuous and strictly positive for $\tau>0$ by Appendix~\ref{app:region3}. The second statement of Lemma~\ref{lem:cmsign}, applied to the COMB gap, which is continuous in $\tau$ with the growth that Lemma~\ref{lem:convergence} provides, therefore gives strict positivity. Since $z_2=(y_1+y_3)/2$, the equality set is $y_1=y_3=0$.
\end{proof}

\section{Independent numerical checks}\label{app:numerics}

The checks in this appendix are not part of the proof. They test the formula and its final claims by methods that bypass the reductions of Appendices~\ref{app:region3} and~\ref{app:region12}, and they tie the construction back to the stationary solution. All floating-point computations use either $40$-digit arithmetic (mpmath) or double precision with the analytic term-by-term derivatives; the scripts are in the \nolinkurl{derivation/} directory of the supplement~\cite{CD26Supp}.

The expansions of Section~\ref{sec:derivation} agree with what they expand: the Region III series with the closed form \eqref{eq:F3} to $30$ digits at six configurations, the Regions I/II series with direct quadrature of \eqref{eq:F12} to $55$ digits at five configurations, the four-expert expansion of Lemma~\ref{lem:F4expansion} with \eqref{eq:F4} to $19$ digits at a point close to the four-coordinate diagonal, where the convergence is slowest, and the trace series as recorded in Remark~\ref{rem:A3}. The function $U$ behaves as it should. Central differences in $40$-digit arithmetic give $|U_\tau-\tfrac12D^2_{\v_*}U|\lesssim10^{-21}$, the noise floor of the difference quotient, at seven points covering the three regions and $\tau\in\{0.4,1,2.5\}$; $U(\tau,x)-\max_ix_i$ falls below $10^{-14}$ by $\tau=10^{-2}$ at points with $x_1>x_2$ and decays like $\sqrt\tau$ where the maximum is attained twice; $U(\ell^2\tau,\ell x)=\ell U(\tau,x)$ holds to machine precision; and at twelve random points of the interface $a_1=0$ the two regional expressions agree in value, in $U_\tau$ and in all sixteen curvatures $D^2_\v U$ to $9\times10^{-16}$.

The most stringent global check is the transform identity \eqref{eq:transform} and its differentiated form \eqref{eq:curvtransform}, tested by numerical quadrature in $\tau$ against an independent evaluation of the stationary solution $u$ from the companion paper. At $\lambda=1$ and three points spanning Regions I, II and III,
\[
 \int_0^\infty e^{-\tau}U(\tau,x)\,d\tau=u(x)
 \quad\text{and}\quad
 \int_0^\infty e^{-\tau}\bigl(D^2_{\v_*}U-D^2_{\v_C}U\bigr)(\tau,x)\,d\tau
 =\bigl(D^2_{\v_*}u-D^2_{\v_C}u\bigr)(x)
\]
agree to $12$ significant digits, the accuracy of the quadrature.

The next check does not use the derivation of the formula at all. An independent monotone finite-difference solution of \eqref{eq:forward} on the gap lattice, of the kind analyzed in~\cite{CDM26}, run with the full Bellman maximum over all $\v\in\{0,1\}^5$, was compared with the formula at $\tau=1$ at seventeen points. The scheme converges at second order in the mesh; after one Richardson step from $h=\tfrac18$ and $h=\tfrac1{16}$ the largest discrepancy is $6\times10^{-6}$, consistent with the residual discretization error. At the symmetric point, Richardson gives $0.6921242$ against the exact value $45\pi^{3/2}/(256\sqrt2)=0.6921215$ of \eqref{eq:constant}. The discrete solve also confirms that fixing the policy $\v_*$ and taking the full Bellman maximum give the same value to $10^{-13}$.

Finally, we evaluated the five-expert Hamiltonian inequality directly from the series, which bypasses the reduction of Appendices~\ref{app:region3} and~\ref{app:region12}. By the exact scaling of Theorem~\ref{thm:main}(iii) it suffices to scan at $\tau=1$, since $D^2_\v U(\tau,x)=\tau^{-1/2}D^2_\v U(1,x/\sqrt\tau)$. The script \nolinkurl{scan.py} draws $1500$ points of the ordered sector with coordinates uniform in $[0,1.6]$ (\texttt{numpy} default generator, seed $20260914$), of which $30\%$ are given one forced collision $x_i=x_{i+1}$ and $20\%$ are placed on $S=\{x_1=x_2,\ x_3=x_4\}$, and evaluates $U_\tau$ and all sixteen curvatures $D^2_\v U$ in double precision from the analytic term-by-term derivatives. It discards a point when $\vartheta>10$, with $\vartheta$ as in \eqref{eq:theta} below, or when dropping the last eight modes moves a curvature by more than $10^{-10}$ relative. This discarded $34$ points and left $1466$: $611$ in Region~I, $643$ in Region~II and $212$ in Region~III, points on the interface $z_3=0$ being counted in Region~II, of which $737$ are strictly ordered, $446$ have one forced collision and $283$ lie on $S$. Over these points $\max|U_\tau-\tfrac12D^2_{\v_*}U|=5.0\times10^{-14}$, and $\v_*$ attains the maximum at all $1466$ of them. The largest apparent excess, $\max_\v(D^2_\v U-D^2_{\v_*}U)=5.6\times10^{-11}$, occurs at $x=(0.1793,0.1302,0.1302,0.0443,0)$ for $\v=(1,1,0,0,0)$, where $x_2=x_3$ and the two curvatures are equal exactly, by the symmetry explained below. Among the pairs of a point and a control that no such symmetry ties, the largest excess is $1.9\times10^{-11}$, for $(1,1,0,0,0)$ at $x=(0.2853,0.0214,0.0107,0,0)$, where $a_4$ is small and the Region III series needs many modes; recomputing that point in $40$-digit arithmetic gives curvatures that agree to $14$ digits, while the stationary gap there is $3.7\times10^{-4}$.

The ties are consistent with Section~\ref{sec:transfer}. On $S$, where Proposition~\ref{prop:combset} makes the COMB gap vanish identically, the computed $|\Delta|$ still reaches $5.4\times10^{-10}$, so a tie tolerance must sit above that figure. We use $10^{-9}$ relative, and the last column records the count at $10^{-12}$ for comparison.
\begin{center}
\begin{tabular}{lrrrrrr}\toprule
control & total & I & II & III & at $10^{-12}$ & symmetric\\\midrule
$(1,0,1,0,0)=\v_*$ & $1466$ & $611$ & $643$ & $212$ & $1466$ & $1466$\\
$(1,0,0,1,1)$ & $894$ & $611$ & $283$ & $0$ & $890$ & $394$\\
$(1,0,0,1,0)$ & $856$ & $1$ & $643$ & $212$ & $855$ & $397$\\
$(1,1,0,0,0)$ & $391$ & $177$ & $170$ & $44$ & $295$ & $113$\\
$(1,0,1,0,1)=\v_C$ & $284$ & $1$ & $283$ & $0$ & $279$ & $283$\\
$(1,0,0,0,1)$ & $212$ & $0$ & $0$ & $212$ & $212$ & $0$\\
\bottomrule
\end{tabular}
\end{center}
The three ties of Corollary~\ref{cor:ties} appear with exactly the predicted region coverage: $(1,0,0,1,1)$ at all $611$ points of Region~I, $(1,0,0,1,0)$ at all $643+212$ points of Regions~II and~III, and $(1,0,0,0,1)$ at the $212$ points of Region~III and nowhere else. Many ties are forced by a symmetry of the state: if a permutation $P$ of the coordinates fixes $x$, then $\nabla^2U(\tau,x)=P^T\nabla^2U(\tau,x)P$ by the permutation invariance of $U$, so $D^2_{P\v_*}U=D^2_{\v_*}U$ at $x$, for the finite-horizon and for the stationary solution alike; the last column counts these ties. Of the $284$ COMB ties, $283$ are the points of $S$, which lie on the interface $z_3=0$ and are counted in Region~II; there the exchanges of $x_1$ with $x_2$ and of $x_3$ with $x_4$ carry $\v_*$ to the complement of $(1,0,0,1,1)$, to $(1,0,0,1,0)$ and to $\v_C$, which accounts for the $283$ ties of $(1,0,0,1,1)$ in Region~II, and the exchange of $x_2$ with $x_3$ carries $\v_*$ to $(1,1,0,0,0)$, which accounts for its $113$ ties at points with $x_2=x_3$. The remaining $278$ ties of $(1,1,0,0,0)$ are numerical: at each of these points the stationary gap $D^2_{\v_*}u-D^2_{(1,1,0,0,0)}u$, evaluated with \nolinkurl{geom.py}, lies between $2.9\times10^{-5}$ and $0.30$, so by \eqref{eq:curvtransform} the finite-horizon gap is not identically zero in $\tau$, and the tie is a gap below the tolerance at $\tau=1$ (Remark~\ref{rem:gapsize}). The one tie of $(1,0,0,1,0)$ in Region~I is at $x=(1.5496,1.5496,0.0904,0.0782,0)$, where $z_3=-0.009$ places the point just inside Region~I, next to the interface on which Corollary~\ref{cor:ties} makes this control tie $\v_*$; the gap there is $2\times10^{-10}$. Off $S$ the COMB gap was resolved and positive at all $1183$ points, and negative at none; its one tie off $S$ at the tolerance $10^{-9}$ is resolved positive at $10^{-12}$, as Theorem~\ref{thm:comb} requires.

The series converge for every $a_4>0$, respectively $z_1>0$, but the number of modes needed grows as the smallest relevant length scale shrinks. In Region III one needs $m\lesssim\sqrt\tau/a_4$ modes, since $\Hk_0(X_{m,\epsilon},\tau)$ decays like $e^{-m^2a_4^2/\tau}$, and in Regions I and II one needs $n\lesssim\sqrt\tau/z_1$. In Regions I and II the expansion of $e^{-2z_4\coth t}$ in powers of $z_4$ also has intermediate terms of size $e^{\vartheta}$,
\begin{equation}\label{eq:theta}
 \vartheta=\frac{4z_4e^{-2z_1}}{1-e^{-2z_1}},
\end{equation}
so that double precision loses about $\vartheta/\ln10$ digits on approach to the face $x_1+x_2=x_3+x_4$, where $z_1=0$. Near that face one should either work in extended precision or evaluate the integral in \eqref{eq:F12} directly under the inverse transform, and on $\coll$ itself the quadrature \eqref{eq:face} replaces the series.